%% file: main.tex
\documentclass[oneside,reqno,11pt]{amsart}
\usepackage[a4paper,margin=3cm]{geometry}
\usepackage{amsmath, amsfonts, amsthm, amssymb, amsmath}
\usepackage{mathtools}
\usepackage{xr}
\usepackage{stmaryrd}
\usepackage{graphicx}
\usepackage{xspace} 
\usepackage{url}
\usepackage{xcolor}
\usepackage[shortlabels]{enumitem}
\usepackage{hyperref}
\usepackage[integrals]{wasysym}
\usepackage[capitalize,nosort]{cleveref}

\setlist[enumerate]{label={\upshape(\roman*)}}

\usepackage{macros}
\input{prooftree}

\numberwithin{equation}{section}

\title[]{Kripke-Joyal forcing for type theory \\ and uniform fibrations}

\author[] {Steve Awodey}
\address{Departments of Philosophy and Mathematics, Carnegie Mellon University, USA}
\email{awodey@cmu.edu}

\author[]{Nicola Gambino} 
\address{Department of Mathematics, University of Manchester, UK}
\email{nicola.gambino@manchester.ac.uk}

\author[]{Sina Hazratpour}
\address{Department of Mathematics, Johns Hopkins University, USA}
\email{sina@jhu.edu}

\newcommand{\noopsort}[1]{}

\date{May 8th, 2024}

\begin{document}

\begin{abstract}
We introduce a new method for precisely relating algebraic structures in a presheaf category and judgements of its internal type theory. 
The method provides a systematic way to organise complex diagrammatic reasoning and generalises the well-known Kripke-Joyal forcing for logic. 
As an application, we prove several properties of algebraic weak factorisation systems considered in Homotopy Type Theory.
\end{abstract}

\maketitle

\begin{flushright}
{\em To Andr\'e Joyal} \bigskip
\end{flushright}

\section*{Introduction}
\label{sec:kjs-intro}

In one of the seminal papers on topos theory~\cite{joyal}, the links between algebraic topology and mathematical logic are described 
as being still embryonic. The situation has changed dramatically since then, thanks to the discovery of a
deep connection between Quillen's homotopical algebra~\cite{quillen-homotopical} and Martin-L\"of type theory~\cite{nordstrom-petersson-smith:ml},
which has given rise to Homotopy Type Theory~\cite{hottbook} and Voevodsky's Univalent Foundations programme~\cite{voevodsky:library}. At the core of this connection is the relationship between the orthogonality property of weak factorisation systems (wfs's for short) and the axioms for identity types~\cite{awodey-warren:homotopy-idtype,gambino-garner:idtypewfs}. 

Presheaf categories play a fundamental role in this context because of the abundance of wfs's and Quillen model structures on them~\cite{cisinski-asterisque}. For applications in both 
type theory~\cite{CoquandT:cubttc} and homotopical algebra~\cite{riehl-cat-homotopy}, however, it is sometimes useful to work with refinements of wfs's known as algebraic weak factorisation systems (awfs's for short)~\cite{garner:small-object-argument,grandis-tholen-nwfs}, where maps satisfying an orthogonality property are replaced by maps equipped with algebraic structure, providing diagonal fillers for lifting problems and satisfying suitable uniformity conditions. The theory of awfs's is analogous to that of wfs's and is supported by a myriad of examples~\cite{bourke-garner-I,bourke-garner-II}, including the notions of a uniform Kan fibration studied in Homotopy Type Theory~\cite{Awodey:cubical-model,van-den-berg-faber,gambino2017frobenius,sattler17,SwanA:algwfs,swan2018identity}. Yet, working with awfs's diagrammatically can be very complex, since the algebraic structure on a map is in general not unique, and thus needs to be carried around explicitly. Because of this, the construction of some
important objects (such as that of fibration structures on a given map) using only diagrams can be a daunting task.

An alternative approach is provided by what may be called the \emph{internal type theory} of a presheaf category, which is a highly expressive extensional dependent type theory~\cite{MaiettiME:modcdt,pitts-catlog,seely-lccc}, capable of handling complex categorical structure and amenable to computer-assisted formalisation of proofs using Coq~\cite{CoqArt} or Agda~\cite{Norell2009}. 
This idea, suggested by Thierry Coquand~\cite{coquand-mfps}, has been successfully applied to Homotopy Type Theory in~\cite{orton-pitts}, where part 
of the theory of uniform fibrations was developed  type-theoretically and implemented in Agda. As we will see here, this approach makes it feasible to write down, say, the object of fibration structures on a map and reason about it formally.  At present, however, in order to link these two approaches one needs to unfold the interpretation of the internal type theory into the presheaf category~\cite{HofmannM:synsdt}, which can be a difficult and laborious process due to the subtleties introduced by dependent types.

Our first main contribution is to introduce and develop a new tool to relate the category-theoretic and the type-theoretic methods for describing structures in presheaf categories: namely, we extend the
Kripke-Joyal forcing of the higher-order logic of a presheaf category~$\tE$ --- a powerful, long-established, technique of categorical logic~\cite{mm:sgl,OsiusG:notkj} --- to a precisely stated  internal type theory of~$\tE$.
In particular, we show how Kripke-Joyal forcing for this type theory allows us to test the validity of type-theoretic judgements in~$\tE$~(\cref{cor:SandC}), unfold explicitly what forcing amounts to for various forms of type, such as dependent products (\cref{prop:forcing.dprd}), which are very important for our applications, and relate it precisely to the ordinary forcing of higher-order logic (\cref{prop:forcing.props-in-ctx}).    This technique then also permits intricate type-theoretic constructions of bespoke \emph{classifiers} which, via forcing, represent familiar and important fibrations of structures (as studied 
in~\cite{shulman2019infty1toposes,swan2021cttalk}), a possibility that promises many further applications. 

The second main contribution of the paper is to the study of the awfs's considered in Homotopy Type Theory, using a combination of the new Kripke-Joyal forcing for the internal type theory and the usual internal reasoning therein. We thereby reconcile the category-theoretic and type-theoretic descriptions of awfs's considered in~\cite{Awodey:cubical-model,CoquandT:modttc,gambino2017frobenius,sattler17,SwanA:algwfs,swan2018identity}
and give simple proofs of several properties of cofibrations and uniform (trivial) fibrations. In particular, we use Kripke-Joyal forcing to provide a type-theoretic characterisation of cofibrations (\cref{prop:forcing.cofibrations}) and a new proof that the pointed endofunctor of cofibrant partial elements admits the structure of a monad~\cref{thm:monad-structure-partial-elts}). 

One of our key results (\cref{prop:ext-implies-tfib}) shows how the uniformity condition that is part of the definition of a uniform (trivial) fibration arises naturally from the forcing conditions for dependent products. We also use Kripke-Joyal forcing to provide a new, simpler, proof that the uniform fibration structures on a map are exactly the algebra structures for a pointed endofunctor~(\cref{prop:plusistfib}).  With these results in place, we then
give an account of the basic properties of uniform fibrations, combining category-theoretic and type-theoretic reasoning. In particular, we describe type-theoretically the object of uniform fibration structures on a map (\cref{prop:uniform-fibration-from-fib-type}) and use it to construct a universal uniform fibration~(\cref{prop:classifier-fib}).  We also describe the homotopical semantics of identity types as path types~(\cref{prop:IdTypeRules}). To the best of our knowledge, all of these results are original and fill an awkward gap in the research in the area, by relating precisely the category-theoretic and the type-theoretic approaches to uniform fibrations.

We hope that these applications demonstrate the utility of Kripke-Joyal forcing for researchers in mathematical logic, category theory and algebraic topology. We should also mention that, while our applications are to the theory of awfs's and Homotopy Type Theory, our Kripke-Joyal forcing for type theory is applicable to many other kinds of algebraic structures that may occur on presheaves, a topic that we leave for future research. 

In order to briefly explain our approach in more detail, let us recall that the standard Kripke-Joyal forcing provides a way to test the validity of a formula of the internal logic of a presheaf category~$\tE = [\op{\mathbb{C}}, \cSet]$. Given a formula $\sigma \co X \to \iprop$, \ie a map into the subobject classifier of~$\tE$,  and a generalised element 
$x \co \yon c \to X$ (where $\yon c$ is the Yoneda embedding applied to~$c \in \cC$), we say that $c$ forces $\sigma(x)$  if there exists a (necessarily unique) map as indicated by the dotted arrow in the diagram
\[
  \begin{tikzcd}
	 & \propcomp{x \co X}{\sigma(x)} \ar[d, tail] \ar[r] 
	 & 1 \ar[d, "\ttrue"] \\
\yon c \ar[r, swap, "{x}"]  \ar[ur, dotted, bend left = 20]  
	& X \ar[r, swap, "{\sigma}"] 
	& \iprop \mathrlap{,} 
 \end{tikzcd}
 \]
 where $\propcomp{x \co X}{\sigma(x)} \mono X$ is the subobject of $X$ obtained by pullback of 
 the subobject classifier $\ttrue \co 1  \to \iprop$ along~$\sigma$. In order to introduce our Kripke-Joyal forcing, we follow~\cite{Awo18-naturalModels} and set up the internal type theory of~$\tE$ by regarding types as maps into the Hofmann-Streicher universe  $\UV$ in~$\tE$ associated to an inaccessible cardinal~$\kappa$ in~$\cSet$~\cite{hofmann-streicher-universes}, an approach which deals correctly with coherence issues~\cite{HofmannM:intttl,pitts-catlog}.
Then, given a type $\alpha \co X \to \UV$ and a generalised element 
$x \co \yon c \to X$, we define our Kripke-Joyal forcing by saying that $c$ forces $a \oftype \alpha(x)$ if we have a (no longer unique) map as indicated by the dotted arrow in the
diagram
\[
  \begin{tikzcd}
	 & \typecomp{X}{\alpha} \ar[d] \ar[r] 
	 & \UVptd \ar[d, "\dspu"] \\
\yon c \ar[r, swap, "{x}"]  \ar[ur, dotted, bend left = 20, "{(x,a)}"]  
	& X \ar[r, swap, "{\alpha}"] 
	& \UV \mathrlap{,} 
 \end{tikzcd}
 \]
where $\typecomp{X}{\alpha} \to X$ is the display map obtained by pullback of the small map classifier~$\dspu~\co~\UVptd~\to~\UV$, in the sense of~\cite{joyal-moerdijk}, along~$\alpha$.  As is the case for Kripke-Joyal forcing of logical formulas, unfolding this condition recursively according to the form of the type $\alpha$ provides a quasi-mechanical method to determine necessary and sufficient conditions for the validity of type-theoretic judgements in $\tE$, something which may be difficult to obtain by other means.

Our definition of the Kripke-Joyal forcing involves generalised elements of a presheaf with domain a representable of the form $x \co \yon c \to X$, where $c \in \cC$, rather than $u \co Y \to X$, where $Y$ is another presheaf. Our choice is motivated by examples of simplicial sets and cubical sets, where generalised elements with domain a representable have a clear topological meaning (as $n$-simplices or as $n$-cubes of $X$, respectively). It should also be pointed out that any quantification over the generalised elements from representables ranges over a set, rather than over a class. Also, our notion of a uniform fibration arises direcly from considering generalised elements with a representable domain. We leave the interesting question of exploring a Kripke-Joyal forcing for arbitrary generalised elements to future work.

Our approach to the internal type theory of $\tE$ also allows us to obtain some new results regarding the relationship between propositions and types in a presheaf category. In particular, we show how the image factorisation of the small map classifier provides a convenient way of expressing the operation of propositional truncation in the type theory (\cref{thm:prop-trunc}). We use this to provide a precise account of the relationship between provability of formulas and inhabitation of associated types (\cref{prop:prop-as-types,prop:prop-as-types-isos}), which is of independent interest. Finally, we note that our Kripke-Joyal forcing semantics for type theory is \emph{complete} with respect to the standard notion of deduction for Martin-L\"of type theory (\cref{Kripke-Joyal-completeness}), in the same way that conventional Kripke semantics is complete for (intuitionistic) first-order logic, something that fails for Kripke-Joyal forcing for higher-order logic.

\smallskip

\noindent
\emph{Outline of the paper.} 
The paper is divided into two parts. The first part, which includes 
\cref{sec:presheaves,sec:small-vs-subobject,sec:types,sec:kjshott-definition}, sets up the general framework,
definition and fundamental results on Kripke-Joyal forcing for the presheaf type theory. The second part,
which includes~\cref{sec:cof,subsec:partiality,sec:tfib,sec:unifib,sec:path-types}, applies Kripke-Joyal forcing to the study of cofibrations, partial elements, uniform trivial fibrations, uniform fibrations and path types, respectively. We end the paper in \cref{sec:con} with some directions for future work.

\smallskip

\noindent
\emph{Notation} For a category $\tC$ and $X, Y \in \tC$, we write $\tC(X,Y)$ for the class of maps from~$X$ to~$Y$. For maps $f \co X \to Y$ and $g \co Y \to Z$, we write $g \circ f$ or simply $gf$ for their composite. For $X \in \tC$, we write $\id_X \co X \to X$ for the identity map at $X$. We write $0$ and $1$ for the initial and terminal object, $X \times Y$ for product, $X + Y$ for coproduct, and more generally~$X \times_Z Y$ for pullback and $X +_Z Y$ for pushout. Given categories $\tC$ and $\tD$ and functors $F \co \tC \to \tD$ and $G \co \tD \to \tC$, we write $F \dashv G$ to mean that~$F$ is left adjoint to~$G$. In this case, there is a family of natural bijections $\tC(X, GA) \cong \tD(FX, A)$, for~$X \in \tC$ and~$A \in \tD$. The transpose of a map $f$ (in either direction) is written $f^\sharp$. As usual, unit and counit of the adjunction are denoted $\eta \co \id_{\tC} \rightarrow GF$ and $\varepsilon \co FG \rightarrow \id_{\tD}$. For a category~$\tE$ and~$X \in \tE$, we write $\tE/_{X}$ for the slice category of objects over $X$.

Throughout the paper, we shall work assuming a fixed inaccessible cardinal $\kappa$.\footnote{Alternatively, we could have assumed an internal full subcategory of $\tE$ closed under the appropriate category-theoretic operations.} We say that a set is $\kappa$-\strong{small} if it has cardinality less than $\kappa$. We say that a category is \strong{locally small} (or \strong{locally $\kappa$-small})  if its hom-classes are sets (or $\kappa$-small sets, respectively). Similarly, we say that a category is \strong{small} (or \strong{$\kappa$-small}) if its class of objects is a set and it is locally small (or its class of objects is $\kappa$-small and it is locally $\kappa$-small, respectively).  We write~$\cSet$ for the category of sets and functions, ${\cSet}^\bullet$ for the category of pointed sets and point-preserving functions, ${\cSet}_{\kappa}$ for the category of sets of rank $< \kappa$
and functions (which is equivalent to the category of $\kappa$-small sets and functions, but has the advantage of being small) and~${\cSet}^\bullet_{\kappa}$ for the category of pointed sets of rank $< \kappa$. 

\smallskip

\noindent
{\em Acknowledgements.} We are grateful to Thierry Coquand and Andy Pitts for helpful conversations on the topic of this paper and comments on a
draft version of it.
Steve Awodey and Nicola Gambino are grateful for the support of the Centre for Advanced Study at the Norwegian Academy of Science and Letters, where this research was begun.  Awodey was also supported by the Air Force Office of Scientific Research under MURI grant FA9550-21-1-0009, and award number FA9550-21-1-0009.
Gambino was supported by EPSRC via grant~EP/V002325/2 and AFOSR 
under agreements FA9550-17-1-0290 and FA9550-21-1-0007. The work in this paper was carried out while Sina Haztratpour was a research
fellow at the University of Leeds, supported by AFOSR under agreement~FA9550-17-1-0290. We are grateful to the referee for their careful reading of the paper. Proof trees were typeset using Paul Taylor's prooftree package.


\section{Presheaf categories}
\label{sec:presheaves} 

\subsection*{Preliminaries}
For a small category $\cC$, we write $\PshC$ for the category
of presheaves over~$\cC$ and $\yon \maps \cC \to \PshC$ for the Yoneda embedding. Sometimes we identify objects and maps in $\cC$ with the representable presheaves and natural transformations between them. Furthermore, for a presheaf
$X$, we often identify elements $x \in X(c)$ with maps~$x \maps \yon c \to X$, as 
permitted by the Yoneda lemma. For $f \co d \to c$ in $\cC$, we write~$x.f$ for the element of~$X(d)$
obtained by applying~$X(f) \co X(c) \to X(d)$ to~$x \in X(c)$, which corresponds to the composite map
$x.f \co \yon d \to \yon c \to X$ under the identification above.
For  $X \in \PshC$, we write $\int\! X$ for its category of elements. Recall the equivalence of categories 
\begin{equation}
\label{eq:slice-of-presheaf-cat}
    \textstyle
\PshC/_{X} \simeq  \Psh(\int\! X) \mathrlap{.}
\end{equation}
If $X = \yon c$, then $\int\! X = \cC/_{c}$ and we have an equivalence of categories 
\begin{equation}\label{eq:sliceoverrep}
\PshC/_{\yon c} \simeq \Psh(\cC/_{c}) \mathrlap{.}
\end{equation}

Let us now fix a small category $\cC$ and define $\tE \defeq \PshC$. We now review some of the basic structure of $\tE$ in order to fix notation. Foremost, $\tE$ is locally cartesian closed, \ie it has a terminal object $1$ and all of its slice categories $\tE/_{X}$ are cartesian closed. Thus for every $f \co Y \to X$ the pullback functor $\pbk{f} \co \tE/_{X} \to \tE/_{Y}$ 
has both a left and a right adjoint, written $\ladj{f} \co \tE/_{Y} \to \tE/_{X}$ and~$\radj{f} \co   \tE/_{Y} \to \tE/_{X}$, respectively.
The action of $\ladj{f}$ is  given simply by composition with $f \co Y \to X$.

\subsection*{Small maps}
We introduce the notion of a small map in $\tE$, which is determined by the Grothendieck universe of $\kappa$-small sets in the ambient set theory.

\begin{defn} \label{thm:def-small-map}
\leavevmode
\begin{enumerate} 
\item\label{item:smallmap} We say that a map $p \co A \to X$ in $\tE$ is a \strong{small map}
if all of its fibers $p^{-1}_c\{x\} \subseteq A(c)$, for $c \in \cC$ and $x \in X(c)$, are $\kappa$-small sets.   Note that this is equivalent to the condition that, for every $x \co \yon c \to X$, the set of all lifts $a \co \yon c \to A$ across $p \co A \to X$  is $\kappa$-small.
\[
\begin{tikzcd}
& A \ar[d, "p"] \\
\yon c \ar[ru, dotted, "a"] \ar[r, swap, "x"] & X
\end{tikzcd}
\]
\item An object $A \in \tE$ is \strong{small} if  $A \to 1$ is a small map.  In this case, all of the values of $A$ are $\kappa$-small sets.
\end{enumerate}
\end{defn}

We write $\tS$ for the class of small maps in $\tE$. For $X \in \tE$, we write $\tS/_{X}$ for the full
subcategory of~$\tE/_{X}$ spanned by small maps.  The category $\tE$ admits a classifier for small maps  (cf.~\cite[Proposition 82]{awodey-cubical-git}), 
given by the Hofmann-Streicher universe $\UV \in \tE$ \cite{hofmann-streicher-universes}, which is defined by letting, for $c \in \cC$,
\begin{equation}
\label{equ:hofmann-streicher-universe}
\UV(c) \defeq \mathrm{Obj} \big[ \op{(\cC/_{c})} \, ,  {\cSet}_\kappa \big]  \, .
\end{equation}
This is set-sized by our choice of ${\cSet}_\kappa$ in the Introduction.
Explicitly, an element~$\alpha~\in~\UV(c)$ is a presheaf on the slice category $\cC/_{c}$ whose values are small sets. Note that we consider~$\UV(c)$ as a set. The action on $U$ of 
an arrow~$d\to c$ in $\cC$ is by precomposition with the composition functor $\cC/_d \to \cC/_c$.  Similarly, define $\UVptd \in \tE$ for $c \in \cC$ by letting,
\[
\UVptd(c) \defeq \mathrm{Obj} \big[ \op{(\cC/_{c})} \, ,  {\cSet}^\bullet_\kappa \big]  \, .
\]
Thus, an element $\alpha \in \UVptd(c)$ is a contravariant functor on the slice category $\cC/_{c}$ whose values are \emph{pointed} small sets.
The forgetful functor ${\cSet}^\bullet_\kappa \to {\cSet}_\kappa$ induces a natural transformation 
\begin{equation}
\label{equ:dspu}
\dspu \co \UVptd \to \UV
\end{equation}
by composition.  We call $\dspu$ the \strong{small map classifier}. This terminology is justified by the fact that~$\dspu$ is small and, for every small map $p \co A \to X$, there exists a pullback diagram
\begin{equation}
\label{equ:char-for-small}
\begin{tikzcd}
A \ar[r] \ar[d, swap, "p"] & \UVptd \ar[d, "\dspu"] \\
X \ar[r, swap, "{\alpha}"]  & \UV \mathrlap{.}
\end{tikzcd}
\end{equation}
Indeed, for a small map $p \co A \to X$, a canonical map $\alpha$ can be defined as follows. For~$c~\in~\cC$ and~$x \in X(c)$, we require an element~$\alpha_{x}\in \UV(c)$, which is
a presheaf $\alpha_{x}~\co~\op{{\cC}/_{c}}~\to~\cSet_\kappa$. This is defined for $f \co d \to c$  by, 
\begin{equation}
\label{equ:char-map-of-small-map}
\alpha_{x}(f) =  \{ a \in A(d) \ | \ p_d(a) = x.f \}  \mathrlap{,}
\end{equation}
which is the set of all $a$ making the following commute,
\begin{equation*}
\begin{tikzcd}
 \yon d  \ar[r, dotted, "a"] \ar[d, swap, "f"] & A  \ar[d, "p"] \\
 \yon c  \ar[r, swap, "x"] & X {.}
\end{tikzcd}
\end{equation*}
Note that the values of $\alpha_{x}$ are small, as $p \co A \to X$ is a small map.\footnote{By our definition of $\cSet_\kappa$, one should replace the set defining
$\alpha_x(f)$ in~\eqref{equ:char-map-of-small-map} with an isomorphic set of rank~$< \kappa$.
Similar considerations should be applied elsewhere in the paper.} We call this map~$\alpha~\co~X \to~\UV$ the \strong{classifying map} of $p \co A \to X$ and say that~$p$ is \emph{classified} 
by~$\alpha$.

We introduce some notation for canonical pullbacks of $\dspu \co \UVptd \to \UV$.
First, for~$\alpha\co\yon c \to \UV$, there is a canonical pullback square 
\begin{equation}
\label{diag:context-extension-repr}
 \begin{tikzcd}
      \typecomp{\yon c}{\alpha} \ar[r] \ar[d, swap, "\dsp{\alpha}"]  & \UVptd
      \ar[d, "\dspu"] \\
       \yon c \ar[r, swap, "\alpha"] & \UV \mathrlap{,}
           \end{tikzcd}   
\end{equation} 
given by taking $\typecomp{\yon c}{\alpha}$ at $d\in \cC$ to be $(\typecomp{\yon c}{\alpha})(d)\ \defeq\ \coprod_{f \in \cC(d,c)}\alpha(f)$. 
Note that this is the object of $\tE/_{\yon c} \simeq \Psh(\cC/_{c})$
associated  to $\alpha \co \op{{\cC}/_{c}} \to \cSet_\kappa$ under \eqref{eq:sliceoverrep}.  
The rest of the pullback in \eqref{diag:context-extension-repr} is then evident.

This assignment then determines a choice of pullback for each $\alpha \co X \to \UV$, 
 \begin{equation}
\label{diag:context-extension}
 \begin{tikzcd}
      \typecomp{X}{\alpha} \ar[r] \ar[d, swap, "\dsp{\alpha}"]  & \UVptd
      \ar[d, "\dspu"] \\
       X \ar[r, swap, "\alpha"] & \UV \mathrlap{.}
           \end{tikzcd}   
\end{equation}
We call $\typecomp{X}{\alpha}$ the \strong{comprehension of $\alpha$ with respect to $\UV$} and 
$\dsp{\alpha}\co\typecomp{X}{\alpha} \to X$ the \strong{display map} associated to $\alpha$.
Since $\pi$ is a small map, so is $\typecomp{X}{\alpha} \to X$. 
In this way, every small map is isomorphic to a display map. Explicitly, for every small map $p \co A \to X$, there is a  diagram
\begin{equation}
\label{diag:context-extension-ess-surj}
\begin{tikzcd} 
A \ar[drr, bend left = 25] \ar[ddr, bend right = 25, swap, "p"] \ar[dr, pos= 0.6, "\cong"] & & \\
&  \typecomp{X}{\alpha} \ar[r] \ar[d] & \UVptd \ar[d, "\dspu"] \\ 
 & X \ar[r, swap, "\alpha"] & \UV \mathrlap{,}
 \end{tikzcd} 
\end{equation}
where $\alpha$ is the classifying map of $p$. In this situation, we say that
the small map $p$ is \strong{displayed}. We summarise the notation for these equivalent notions:
\[
p \co A \to X \text{ small map} \mathrlap{,}  \qquad
\alpha \co X \to \UV \mathrlap{,}  \qquad
\dsp{\alpha} \co \typecomp{X}{\alpha} \to X  \text{ display map} \mathrlap{.} 
\]
In the following, all three points of view will be exploited.

\begin{rmk}
\label{rmk:universe-in-slice}
Let $X$ be an object of $\tE = \Psh(\cC)$ and consider the slice topos $\tE/_{X}$.  Up to the equivalence of categories~\eqref{eq:slice-of-presheaf-cat}, this is the presheaf topos $\Psh(\int\! X)$, with a notion of small map as in \cref{thm:def-small-map}, and an associated small map classifier $\pi_X \maps E_X \to U_X$.  One can show that this map is obtained by composition (and whiskering) of the presheaves~$E$ and~$U$ (and the natural transformation $\pi \co E\to U$) with the forgetful functor $\op{(\int\! X)} \to \op{\cC}$.  In particular, 
$E_X (c,x) = E(c)$ and $U_X (c,x) = U(c)$, for $c$ in $\cC$ and $x \in X(c)$.
In~$\tE/_{X}$, the small map classifier $\pi_X \maps E_X \to U_X$ corresponds under~\eqref{eq:slice-of-presheaf-cat} to the base change~$X^*\pi$ of the small map classifier $\pi$ in $\tE$, 
\begin{equation}
\label{eq:small-map-classifier-in-slice}
\begin{tikzcd}
X^*\UVptd \ar[d, swap, " X^*\pi "]  \ar[r, equals] 
		& \UVptd \times X \ar[r] \ar[d, swap, " \dspu \times X "] & \UVptd \ar[d, "\dspu"] \\
X^*\UV  \ar[r, equals] & \UV \times X \ar[r, swap]  & \UV \mathrlap{.}
\end{tikzcd}
\end{equation}
Thus in particular,  
\begin{equation}
\label{eq:small-map-classifier-in-slice-sections}
\textstyle \Psh(\cC)(X,U)\ \iso\ \Psh(\cC)/_{X}(1, X^*U)\ \iso\ \Psh(\int\! X) (1, U_X)  \, .
\end{equation} 
We will sometimes notationally \emph{identify} a type $\alpha \co X \to \UV$ in $\Psh(\cC)$ with the associated closed type $\alpha \co 1 \to \UV_X$ in $\Psh(\int\! X) $ under this isomorphism.
\end{rmk}

\subsection*{Dependent sums and products of small maps}
Small maps are closed under a variety of operations.
In particular, the pullback of a small map along any map $f\co Y\to X$ is again a small map. 
Furthermore, if $p \co A \to X$ is a small map, then the left and right adjoints to pullback along $p$ 
 restrict to small maps, in the sense that there are serially-commuting diagrams as follows.
\begin{equation}
\label{equ:ladj-radj-small}
\begin{tikzcd} 
\tS/_{A} \ar[r, bend left = 20, "\ladj{p}"]  \ar[d] & 
\tS/_{X} \ar[l, bend left = 20,  "\pbk{p}"]  \ar[d] \\[2ex]
\tE/_{A} \ar[r, bend left = 20,  "\ladj{p}"]  & 
\tE/_{X} \ar[l, bend left = 20,  "\pbk{p}"]  
\end{tikzcd} 
\qquad \qquad
\begin{tikzcd} 
\tS/_{X} \ar[r, bend left = 20, "\pbk{p}"]  \ar[d] & 
\tS/_{A} \ar[l, bend left = 20,  "\radj{p}"]  \ar[d] \\[2ex]
\tE/_{X} \ar[r, bend left = 20,  "\pbk{p}"]  & 
\tE/_{A} \mathrlap{.} \ar[l, bend left = 20,  "\radj{p}"]  
\end{tikzcd} 
\end{equation}
The action of these adjoints is reflected into operations on classifying maps, as we now explain.

Consider first the pullback along any map $f \co Y \to X$.  Given $\alpha \co X \to \UV$ with its small display map $\dsp{\alpha}\co\typecomp{X}{\alpha} \to X$ the pullback $f^*(\dsp{\alpha})$ is classified by the composite $\alpha f\co Y \to \UV$, by the two pullbacks lemma,
\begin{equation}
 \label{equ:pullbackcomp}
 \begin{tikzcd}
\typecomp{Y}{\alpha f} \ar[r] \ar[d] & \typecomp{X}{\alpha} \ar[r] \ar[d]  &[10pt] \UVptd \ar[d] \\
 Y \ar[r, swap, "f"] & X \ar[r, swap, "\alpha"] & \UV \mathrlap{.}
 \end{tikzcd} 
 \end{equation}
In this way, the pullback operation on small maps $f^*\co \tS/_{X}\to \tS/_{Y}$ is induced, up to isomorphism, by precomposition of the associated classifying maps,
$f^*p_\alpha \cong p_{\alpha f}$. We have the following characterisation of the sections of the small display map $ \typecomp{Y}{\alpha f} \to Y$, which will be used to validate the 
substitution rule in~\cref{lem:subst}.

\begin{prop}\label{thm:pullback-subst-terms} 
Let $f\co Y\to X$ and  $\alpha\co X\to \UV$. 
Then the following data are in bijective correspondence:
\begin{enumerate}
  \item sections $a\co Y \to \typecomp{Y}{\alpha f}$ over $Y$,
  \item lifts $a' \co Y \to \typecomp{X}{\alpha}$ of $f$ across $p_\alpha$. \qed
  \end{enumerate}
\end{prop}

Now consider the left adjoint. 
Given $\alpha \co X \to \UV$ and $\beta \co  \typecomp{X}{\alpha} \to \UV$, let $p \co  \typecomp{X}{\alpha} \to X$ and~$q \co \typecomp{X.\alpha}{\beta}  \to \typecomp{X}{\alpha}$ be the associated display maps. Since $p$ and $q$ are small, we obtain a small map $\ladj{p}(q)$  in~$\tS/_{X}$ by application of the functor $\Sigma_p$ in~\eqref{equ:ladj-radj-small}. This is given simply by composition, so that $\ladj{p}(q) \defeq q \circ p$.  Let 
$\Sigma_\alpha(\beta) \co X \to \UV$ be  the classifying map of $\ladj{p}(q) \co \typecomp{X.\alpha}{\beta} \to X$, giving rise to the following pullback diagram, in which the map labelled $\tpair_{\alpha, \beta}$ is the canonical one: 
\begin{equation}
 \label{equ:sigma-pair}
 \begin{tikzcd}
\typecomp{X}{\Sigma_\alpha(\beta)} \ar[r, "\tpair_{\alpha, \beta}"] \ar[d]  &[10pt] \UVptd \ar[d] \\
 X \ar[r, swap, "\Sigma_\alpha(\beta)"] & \UV \mathrlap{.} 
 \end{tikzcd} 
 \end{equation}
Note that there is an isomorphism $\typecomp{X.\alpha}{\beta} \cong \typecomp{X}{\Sigma_{\alpha}(\beta)}$ over $X$.
 We therefore obtain maps~$\tproj_1$ and~$\tproj_2$ fitting into the diagrams
\begin{equation}
\label{equ:sigma-tproj}
 \begin{tikzcd}
\typecomp{X}{\Sigma_\alpha(\beta)} \ar[d] \ar[r, "\tproj_1"] & \UVptd \ar[d] \\ 
X \ar[r, swap, "\alpha"] & \UV \mathrlap{,}
\end{tikzcd} \qquad
\begin{tikzcd}
\typecomp{X}{\Sigma_\alpha(\beta)} 
              \ar[d] \ar[r, "\tproj_2"] & \UVptd \ar[d] \\ 
\typecomp{X}{\alpha} \ar[r, swap, "\beta"] & \UV  \mathrlap{.} 
\end{tikzcd}
\end{equation}
From this description we obtain \cref{thm:seely-sigma} below, which characterises the sections of the display map~$\typecomp{X}{\Sigma_\alpha(\beta)}~\to~X$. It is essentially contained already in~\cite{seely-lccc}, and will be used to show the validity of the rules for dependent sum types in~\cref{thm:rules-sigma}. 

\begin{prop}\label{thm:seely-sigma} 
Let $\alpha\co X\to \UV$ and $\beta \co \typecomp{X}{\alpha} \to \UV$. 
Then the following data are in bijective correspondence:
\begin{enumerate}
  \item pairs of sections $a \co X \to \typecomp{X}{\alpha}$ and $b \co X \to \typecomp{X}{\beta(a)}$, both over $X$, 
  \item sections $c \co X \to \typecomp{X}{\Sigma_\alpha(\beta)}$ over $X$. \qed
  \end{enumerate} 
\end{prop}

\begin{rmk*} 
As shown in~\cite{Awo18-naturalModels}, the operation mapping $\alpha$ and $\beta$ to $\Sigma_\alpha(\beta)$ can be internalised in~$\tE$, by
making use of the polynomial functor $P \co \tE \to \tE$ associated to the map $\dspu \co \UVptd \to \UV$, in the sense of~\cite{gambino-kock}, \ie the composite
\begin{equation}
\label{equ:poly-dspu}
\begin{tikzcd} 
\tE \ar[r, "(-) \times \UVptd"]  &[10pt]  \tE/_{\UVptd} \ar[r, "\radj{\pi}"] &[10pt] \tE/_{\UV} \ar[r, "\ladj{\UV}"]  &[10pt] \tE \mathrlap{.} 
\end{tikzcd} 
\end{equation}
Indeed, there is a pullback diagram 
\[
\begin{tikzcd} 
Q \ar[r, "\tpair"]  \ar[d] & \UVptd \ar[d] \\
P(\UV) \ar[r, swap, "\Sigma"]  & \UV \mathrlap{,} 
\end{tikzcd}
\]
and for every $\alpha$ and $\beta$ as above, we have a pasting
\begin{equation*}
\begin{tikzcd}
\typecomp{X}{\Sigma(\alpha,\beta)} \ar[r] \ar[d] & Q \ar[r, "\tpair"] \ar[d] &[20pt] \UVptd \ar[d] \\
 X \ar[r, swap, "{(\alpha, \beta)}" ] & P(\UV) \ar[r, swap, "\Sigma" ] & \UV  \mathrlap{,}
 \end{tikzcd} 
\end{equation*}
where $(\alpha, \beta) \co X \to P(\UV)$ is determined by the universal property of $P(\UV)$, and $Q$ can be seen as the object of pairs of sections $(a,b)$ as in \cref{thm:seely-sigma}. 
Comparing this pasting with~\eqref{equ:sigma-pair}, not only $X.\Sigma_{\alpha}(\beta)$ is isomorphic to  $X.\Sigma(\alpha, \beta)$ over $X$, but also the composite map $\Sigma(\alpha, \beta) \co X \to U$ above is isomorphic 
to the map classifying map $\Sigma_\alpha(\beta) \co X \to U$, in the sense to be made precise when we discuss the internal category of types in \cref{sec:small-vs-subobject}.
\end{rmk*}

 We now consider the right adjoint. With $p \co  \typecomp{X}{\alpha} \to X$ and $q \co \typecomp{X.\alpha}{\beta}  \to \typecomp{X}{\alpha}$ as before, let $\Pi_\alpha(\beta)$ be the classifying map of the small map $\radj{p}(q)$, giving rise to a pullback diagram of the form
 \begin{equation}
 \label{equ:pi-lambda}
 \begin{tikzcd}
\typecomp{X}{\Pi_\alpha(\beta)}  \ar[r, "\lambda_{\alpha, \beta}"] \ar[d] &[20pt] \UVptd \ar[d] \\
 X \ar[r, swap, "{\Pi_\alpha(\beta)}" ] & \UV  \mathrlap{.}
 \end{tikzcd} 
 \end{equation}
The map on the left hand side  is isomorphic to $\radj{p}(q)$ in $\tE/_{X}$ and hence
inherits its universal property. In particular, the counit of the adjunction induces a 
map~$\varepsilon \co \typecomp{X}{\alpha.\Pi_\alpha(\beta)} \to \typecomp{X}{\alpha.\beta}$ 
in~$\tE/_{\typecomp{X}{\alpha}}$ which gives rise to a commutative diagram
\begin{equation}
\label{equ:pi-app}
 \begin{tikzcd}
\typecomp{X.\alpha}{\Pi_\alpha(\beta)}  \ar[r, "\funapp"] \ar[d] &[20pt] \UVptd \ar[d] \\
 \typecomp{X}{\alpha} \ar[r, swap, "{\beta}" ] & \UV  \mathrlap{.}
 \end{tikzcd} 
 \end{equation}
 
The resulting characterisation of the sections of the display map $\typecomp{X}{\Pi_\alpha(\beta)} \to X$ is also essentially contained already in~\cite{seely-lccc}, and will be used to show the validity of the rules for dependent product types in~\cref{thm:rules-pi}. 

\begin{prop}\label{thm:seely-pi} 
Let $\alpha\co X\to \UV$ and $\beta \co \typecomp{X}{\alpha} \to \UV$. 
Then the following data are in bijective correspondence:
\begin{enumerate}
  \item sections $b \co \typecomp{X}{\alpha} \to \typecomp{X}{\alpha.\beta}$ over $\typecomp{X}{\alpha}$,
  \item sections ${b}^\sharp \co X \to \typecomp{X}{\Pi_\alpha(\beta)}$ over $X$.
  \end{enumerate}
 The bijection is defined by $\varepsilon \circ~p^*b^\sharp =~b$, using pullback along $p$ and composition with the counit $\varepsilon$. \qed
 \end{prop}
 
\medskip

 \begin{rmk*} 
As was the case for the left adjoint, there is a pullback diagram 
\[
\begin{tikzcd} 
P(\UVptd) \ar[r, "\lambda"]  \ar[d, "P(\pi)"] & \UVptd \ar[d, "\pi"] \\
P(\UV) \ar[r, swap, "\Pi"]  & \UV \mathrlap{,} 
\end{tikzcd}
\]
which gives rise for every $\alpha$ and $\beta$ to the pasting diagram
\begin{equation}\label{diagram:Picomp}
\begin{tikzcd}
\typecomp{X}{\Pi(\alpha,\beta)} \ar[r] \ar[d] & P(\UVptd) \ar[r, "\lambda"] \ar[d] &[20pt] \UVptd \ar[d] \\
 X \ar[r, swap, "{(\alpha, \beta)}" ] & P(\UV) \ar[r, swap, "\Pi" ] & \UV  \mathrlap{.}
 \end{tikzcd} 
\end{equation}
This is related to the one in~\eqref{equ:pi-lambda} in a way that is analogous to the one discussed for the left adjoint.
\end{rmk*}

\subsection*{The subobject classifier}\label{subsec:SOC}

We conclude this review section by introducing some notation and recalling some basic facts about the subobject classifier of $\tE$,
written $\iprop$. This presheaf is defined  by letting, for $c \in \cC$, 
\begin{equation}
\label{equ:subobject-classifier}
\Omega(c) \defeq \mathrm{Obj} \big[ \op{ (\cC/_{c})} \, , \mathbf{2} \big] \mathrlap{,}
\end{equation}
where $\mathbf{2}$ is the poset $\{0\leq 1\}$.  The map $\ttrue \co 1 \to \iprop$ is defined by letting its component at $c$ be the constant functor with value~$1$. This classifies subobjects
in the sense that for every subobject $S \mono X$, there exists a unique $\sigma \co X \to \Omega$ such that
\begin{equation}
\label{equ:char-for-sub}
\begin{tikzcd}
S \ar[d, tail] \ar[r] & 1 \ar[d, "\ttrue"] \\
X \ar[r, swap, "\sigma"] & \iprop 
\end{tikzcd} 
\end{equation}
is a pullback diagram. In this case, we call $\sigma$ the \emph{characteristic map} of~$S$ and say that~$S$ is
\strong{classified} by $\sigma$.  We introduce some notation for
the pullbacks of $\ttrue \co 1 \to \iprop$. Given~$\sigma \co X \to \Omega$, we write
\begin{equation} 
\label{equ:subobject-classifier-2} 
\begin{tikzcd}
\propcomp{x \oftype X}{\sigma(x)} \ar[r] \ar[d, tail]  &1\ar[d, tail, "\ttrue"] \\
X  \ar[r, swap, "\sigma"]  & \iprop
\end{tikzcd}
\end{equation} 
for the subobject which is the pullback of $\ttrue$ along $\sigma$.

For $X \in \tE$, we write $\Sub(X)$ for the poset of subobjects of $X$. This poset has a Heyting algebra structure, which is reflected into operations on morphisms from~$X$ to~$\Omega$, for which we shall use standard logical notation. For example, if $S, T \in \Sub(X)$ are classified by~$\sigma$ and~$\tau$, 
respectively, so that $S = \{ x \oftype X \ | \ \sigma(x) \}$ and $ T =\{ x \oftype X \ | \ \tau(x) \}$, 
then we write~$\sigma \implies \tau$ for the characteristic map of the Heyting implication $S \implies T$ in~$\Sub(X)$. Thus the subobject $\{ x \oftype X \ | \ (\sigma \implies \tau)(x) \}$ has the universal property of the Heyting implication~$S~\implies~T$.

For our purposes, quantification along small maps will play an important role, so we introduce some notation for it. Fix $\alpha \co X \to \UV$ and let $p \co \typecomp{X}{\alpha} \to X$ be the associated small display map. Pullback along $p$ restricts to a function $\pbk{p} \co \Sub(X) \to \Sub(\typecomp{X}{\alpha})$, which is a Heyting algebra morphism with both left and right adjoints, 
written~$\exists_p$ and~$\forall_p$, respectively.  If $S \in \Sub(\typecomp{X}{\alpha})$ is classified by $\sigma \co \typecomp{X}{\alpha} \to \iprop$, 
so that $S = \{ (x,a) \oftype \typecomp{X}{\alpha} \ | \ \sigma(x,a) \}$, 
then we write $\exists_\alpha(\sigma) \co X \to \iprop$ for the characteristic map of $\exists_p(S) \in \Sub(X)$, so that
\[
\exists_p(S) = \exists_p\{ (x,a) \oftype \typecomp{X}{\alpha} \ | \ \sigma(x,a) \} = \{ x \oftype X \ | \ (\exists_\alpha \sigma) (x) \} \mathrlap{.}
\]
Since $\exists_p$ is calculated by taking the image factorisation of the composite of $S \mono \typecomp{X}{\alpha}$ and $p \co \typecomp{X}{\alpha} \to X$, we obtain a diagram 
\begin{equation}\label{diag:imageclassify}
\begin{tikzcd} 
S  \ar[r, twoheadrightarrow] \ar[d, tail] &[20pt] \exists_p(S)  \ar[d, tail]  \ar[r] &[20pt] 1 \ar[d, "\ttrue"] \\
\typecomp{X}{\alpha} \ar[r, swap, "p"] & X \ar[r,  swap, "{\exists_\alpha(\sigma)}"] & \iprop \mathrlap{,}
\end{tikzcd}
\end{equation}
whose right hand side is a pullback. The universal quantification $\forall_p(S) \in \Sub(X)$ is similarly classified by a map written $\forall_\alpha(\sigma) \co X \to \iprop$, giving  
\[
\forall_p(S) = \forall_p\{ (x,a) \oftype \typecomp{X}{\alpha} \ | \ \sigma(x,a) \}  = \{ x \oftype X \ | \ (\forall_\alpha \sigma)(x) \} \mathrlap{.}
\]


\section{Small maps and subobjects}
\label{sec:small-vs-subobject}

\subsection*{Subobjects as small maps} 

There is an evident analogy between the small map classifier $\dspu \co \UVptd \to \UV$ and the subobject classifier $\ttrue \co 1 \to \iprop$.  For one thing, recall that the universe~$\UV$ was defined in~\eqref{equ:hofmann-streicher-universe} by mapping into the category $\cSet_{\kappa}$ of small sets, while the subobject classifer~$\iprop$ was defined in~\eqref{equ:subobject-classifier} by mapping into the category~$\mathbf{2}$ of truth values. 
Note also the similarity between the classifying map of a small map in~\eqref{equ:char-for-small} and the characteristic map of 
a subobject in~\eqref{equ:char-for-sub}, as well as between the comprehension of a map $\alpha \co X \to \UV$ with respect to $\UV$ in~\eqref{diag:context-extension} and the comprehension of a map~$\sigma \co X \to \Omega$ with respect to $\Omega$ in~\eqref{equ:subobject-classifier-2}.
Our aim in this section is to make the relation between~$\UV$ and~$\iprop$ more precise.  Anticipating the use to which they will be put in the sequel, we may regard~$\UV$ as an object of \emph{types} and~$\iprop$ as an object of \emph{propositions}.  Observe first that there is an inclusion functor of truth values into sets, $\mathbf{2} \hookrightarrow {\cSet}_{\kappa}$,
where, to be specific, we take the ordinals $0= \emptyset$, $1 = \{0\}$, $2 = \{0, 1\}$, the latter ordered by inclusion when regarded as the category $\mathbf{2}$.  By the definitions of $\UV$ and $\iprop$ in~\eqref{equ:hofmann-streicher-universe} and~\eqref{equ:subobject-classifier}, respectively,
this functor induces by composition a natural transformation 
\begin{equation}
\label{equ:pat-Omega-to-U}
\pat{-} \co \iprop \to \UV \mathrlap{,}
\end{equation}
which we call the \emph{inclusion} of propositions into types.

\begin{prop} \label{diag:iprop-vs-V} The inclusion map $\pat{-} \maps \iprop \to \UV$ is a monomorphism in $\tE$, and fits into a pullback diagram of the form
\begin{equation*}
\begin{tikzcd}
1 \ar[r] \ar[d, swap, tail,  "\ttrue"] 
& \UVptd \ar[d, "\dspu"] \\
\iprop \ar[r, swap, tail, "\pat{-}"] 
& \UV \rlap{.}
\end{tikzcd}
\end{equation*} 
\end{prop}

\begin{proof} 
The inclusion functor $\mathbf{2} \hookrightarrow {\cSet}_{\kappa}$ is injective on objects, and so composing with it is pointwise monic. 
Similarly, the pullback is induced via composition by the evident pullback of categories and functors
\[
\begin{tikzcd}
\mathbf{1} \ar[r] \ar[d] 
&  {\cSet}^\bullet_{\kappa} \ar[d] \\
\mathbf{2} \ar[r] 
& {\cSet}_{\kappa}  \mathrlap{.} 
\end{tikzcd} \vspace{-1em}
\] 
\end{proof} 

\begin{cor}\label{cor:compUvOmega} Let $\sigma \colon X \to \Omega$. Then $\typecomp{X}{\pat{\sigma}} = \propcomp{x \co X}{\sigma(x)}$, 
\ie  the comprehension of $\sigma$ with respect
to $\Omega$ is equal, as a subobject of $X$, to the comprehension of its inclusion $\pat{\sigma}$ with respect to $\UV$. 
\end{cor}

\begin{proof} The display map $p_{\pat{\sigma}}\co \typecomp{X}{\pat{\sigma}} \to X$ fits into the diagram
\[
\begin{tikzcd}
 \typecomp{X}{\pat{\sigma}} \ar[r] \ar[d]  & 1\arrow[d, tail, "\ttrue"]  \arrow[r]& \UVptd \arrow[d, "\dspu"]  \\
  X  \ar[r, swap, "\sigma"] & \iprop \ar[r, "\pat{-}", swap]  & \UV \mathrlap{.} 
\end{tikzcd} 
\]
Since the outer rectangle is a pullback by definition and the right hand square is a pullback by~\cref{diag:iprop-vs-V},  the left hand square is also a pullback. Thus the mono $p_{\pat{\sigma}}\co \typecomp{X}{\pat{\sigma}} \to X$ represents the subobject $\propcomp{x \co X}{\sigma(x)}$ classified by $\sigma \co X \to \Omega$. 
\end{proof}

\begin{rmk}
Propositional comprehension of a subobject $S\mono X$ from a map $\sigma\co X\to \iprop$ is determined only 
up to equivalence of the representing monomorphisms $s \maps S \mono X$, whereas for a type $\alpha\co X\to \UV$, we have made a canonical choice $p_\alpha \co \typecomp{X}{\alpha} \to X$, which however is only one among many isomorphic small maps $p \co A \to X$ classified by $\alpha$.  \cref{cor:compUvOmega} says that we can use these canonical small maps as canonical representatives of subobjects when convenient.  See \cref{subsec:PaT} for a warning about this use.
\end{rmk}

\subsection*{Propositional truncation} \label{subsec:prop-trunc}
Consider the universal small map $\dspu \co \UVptd \to \UV$ and its image factorisation $\UVptd \twoheadrightarrow \mathsf{im}(\dspu)  \mono \UV$
and let $\supp \co \UV \to \iprop$ be the characteristic map of the subobject $\mathsf{im}(\dspu) \mono {\UV}$, fitting into the pullback diagram
\[
\begin{tikzcd}
 \mathsf{im}(\dspu)  \ar[d,tail] \ar[r]  &  1 \ar[d, tail, "\ttrue"] \\
  \UV \ar[r, swap, "\supp"] & \iprop \mathrlap{.} 
\end{tikzcd}
\] 
Note that $\supp$ is induced by composition with the functor ${\cSet}_{\kappa} \to \mathbf{2}$ mapping every non-empty set
to $1$ and the empty set to $0$.  For $\alpha \co X \to \UV$, we write $\supp(\alpha) \co X \to \iprop$ and call it the \strong{support} of $\alpha$. Pasting the factorisation on the left and the pullback diagram of \cref{diag:iprop-vs-V}  on the right, we obtain
\begin{equation}
\label{diag:iotasigma}
\begin{tikzcd}
{\UVptd}   \ar[d,swap,"\pi"]  \ar[r,twoheadrightarrow] & \mathsf{im}(\dspu)  \ar[d,tail] \ar[r]  &  1 \ar[d, tail,"\ttrue"] \ar[r] & {\UVptd}  \ar[d,"\dspu"] \\
\UV \ar[r, equal] &  \UV \ar[r, swap, "\supp"] & \iprop \ar[r,swap,"\pat{-}"] &  \UV.
\end{tikzcd}
\end{equation}
We define the \strong{truncation} operation $\trunc \co \UV \to \UV$ to be the composite of $\supp \co \UV \to \iprop$ and~$\pat{-}\co \iprop \mono \UV$. For $\alpha \co X \to \UV$, we write $\trunc(\alpha) \co X \to \UV$ for the composite of $\alpha$ and $\trunc$, so that 
\begin{equation}
\label{equ:ptrun-definition}
\trunc(\alpha) \defeq \pat{ \supp(\alpha) }  \mathrlap{.} 
\end{equation}
This operation of \strong{propositional truncation} behaves like the bracket types of \cite{Awodey-Bauer}, as stated in the following proposition

\begin{prop} \label{thm:prop-trunc}
Let $p \co A \to X$ be a small map with classifying map $\alpha \co X \to \UV$. Then 
$\ptrun{\alpha} \co X \to~\UV$ is a classifying map for the image factorisation of $p$, in
the sense that there is a pullback diagram
\[
\begin{tikzcd}
\mathsf{im}(p) \ar[r] \ar[d, tail] & \UVptd \ar[d, "\dspu"] \\
X \ar[r,  swap, "\ptrun{\alpha}"] & \UV \mathrlap{.} 
\end{tikzcd}
\]
Moreover, the composite $\supp\circ \pat{-} \co \iprop \to \iprop$ is the identity map.
\end{prop}

\begin{proof}
Consider the image factorisation on the left face below, which is a pullback of the right face, since image factorisations are stable under pullback:
\begin{equation}\label{diag:trunctypeimage}
\begin{tikzcd}[column sep=small, row sep = small]
A \ar[rr] \ar[dd, swap, "p"]  \ar[rd, twoheadrightarrow, swap] 
  &[10pt] & \, \, \, {\UVptd}  \, \, \,  \ar[dr,twoheadrightarrow] \ar[dd,swap] \\
 & {\mathsf{im}(p)}  \ar[rr, crossing over] & &  \mathsf{im}(\dspu) \ar[dd,tail]  \\
X \ar[rr, swap, pos= 0.7, "\alpha"] \ar[dr, equal]  
      && \UV \ar[dr,equal] & \\
&  X  \ar[from = uu, crossing over, tail] \ar[rr,swap, "\alpha"]  
  && \UV \mathrlap{.} 
\end{tikzcd}
\end{equation}
Adjoining diagram \eqref{diag:iotasigma} on the right, we obtain
\[
\begin{tikzcd}[column sep=small, row sep = small]
 {A} \ar[rr] \ar[dd]  \ar[rd, twoheadrightarrow, swap] 
  && \, \, \, {\UVptd}  \, \, \,  \ar[dr,twoheadrightarrow] \ar[dd,swap] \\
 & {\mathsf{im}(p)}  \ar[rr, crossing over] & & \mathsf{im}(\dspu)  \ar[dd,tail]  \ar[rr]  
  &&  1 \ar[dd, tail] \ar[rr] && {\UVptd}  \ar[dd] \\
 X \ar[rr, swap, pos= 0.7, "\alpha"] \ar[dr, equal]  
      && \UV \ar[dr,equal] & \\
&  X  \ar[from = uu, crossing over, tail] \ar[rr,swap, "\alpha"]  
  && \UV \ar[rr, swap, "\supp"] && \iprop \ar[rr,swap,"\pat{-}"] &&  \UV \mathrlap{,}
\end{tikzcd}
\] 
which shows that the three-fold composite across the front face below is also a pullback.
Thus 
\[
\mathsf{im}(p) \cong \typecomp{X}{\pat{\supp(\alpha)}} = \typecomp{X}{\ptrun{\alpha}} \mathrlap{.}
\]
Finally, the left hand square in 
\[
\begin{tikzcd}
1  \ar[d,swap,"\ttrue",tail]  \ar[rr] 
  &&  \mathsf{im}(\dspu)  \ar[d,tail] \ar[rr]  &&  1 \ar[d,"\ttrue", tail] \\
\iprop \ar[rr,swap,"\pat{-}"] &&  \UV \ar[rr, swap, "\supp"] && \iprop 
\end{tikzcd}
\] 
is a pullback, as can be seen by putting $\pat{-}$ for $\alpha$ in \eqref{diag:trunctypeimage} and using $\mathsf{im}(\typecomp{\Omega}{\pat{-}}) = \mathsf{im}(\ttrue) = \ttrue$, since $\ttrue \co 1\to\iprop$ is monic.
The right hand square is a pullback by definition.  Thus the outer rectangle is a pullback, and so the composite $\supp \circ \pat{-}$ must be the identity, since it is the characteristic map of
 $\ttrue \co 1 \mono \iprop$.
\end{proof}

\subsection*{Propositions and types}\label{subsec:PaT}

Subobjects and small maps can be combined, but some care is required in working with their characteristic maps and classifying maps, respectively. For example, fix $\alpha \co X \to \UV$ and let $p \co \typecomp{X}{\alpha} \to X$ be the associated display map. Then consider 
 $\sigma \co  \typecomp{X}{\alpha}  \to \Omega$ and let $S =  \propcomp{ (x,a) \oftype  \typecomp{X}{\alpha}}{\sigma(x,a)}$ 
 be the associated subobject,
\begin{equation}\label{diag:sigmaandalpha}
\begin{tikzcd} 
S \ar[r, tail] &   \typecomp{X}{\alpha} \ar[r, "p"] &  X \mathrlap{.}
\end{tikzcd} 
 \end{equation}
By Corollary \ref{cor:compUvOmega}, $S = \typecomp{X}{\alpha.\pat{\sigma}}$ as subobjects, and so the composite $S \to X$ is classified  as a small map by $\Sigma_{\alpha}\pat{\sigma} \co X\to\UV$.  Thus $S \to X$ is isomorphic over $X$ to the asssociated display map,
\[
\begin{tikzcd}
S \ar[r,"\cong"] \ar[rd] & \typecomp{X}{\Sigma_{\alpha}\pat{\sigma}} \ar[r] \ar[d]  & \UVptd \ar[d, "{\dspu}"] \\
& X \ar[r, swap, "\Sigma_{\alpha}\pat{\sigma}"] & \UV \mathrlap{.}
  \end{tikzcd}
 \]
 
Now consider the image factorisation $S \twoheadrightarrow \exists_p(S) \rightarrowtail X$ and the characteristic map ${\exists_\alpha(\sigma)} \co X\to \iprop$ of the subobject $\exists_p(S) \mono X$, as in  \eqref{diag:imageclassify}.   
Composing with the inclusion $\pat{-}\co \iprop\mono\UV$, we can compare the map $\pat{\exists_\alpha(\sigma)} \co X\to\UV$
to the propositional truncation $\ptrun{\Sigma_{\alpha}\pat{\sigma}} \co X\to \UV$, 
which also classifies the image $\exists_p(S)\mono X$ (as a small map), by Proposition \ref{thm:prop-trunc}.  
Indeed, since $\trunc = \pat{-}\circ\supp$, and $\pat{-}$ is monic, it suffices to compare the maps 
$\exists_\alpha(\sigma)  \co X\to \iprop$ and $\supp(\Sigma_{\alpha}\pat{\sigma})  \co X\to \iprop$.
Since these are characteristic maps into $\iprop$, they are equal if and only if the associated subobjects of $X$ are the same, which is true since both are $\mathsf{im}(S)\mono X$.  Thus
 we have $\pat{\exists_\alpha(\sigma)} = \ptrun{\Sigma_{\alpha}\pat{\sigma}}$ as maps from~$X$ to~$\UV$. 

Now consider the universal quantification, and compare $\pat{\forall_\alpha(\sigma)} $ with $ \Pi_{\alpha}\pat{\sigma}$.  
Again we have the subobject 
$S \mono \typecomp{X}{\alpha}$ of \eqref{diag:sigmaandalpha} with its characteristic map $\sigma \co \typecomp{X}{\alpha} \to \iprop$, as well as its classifying map as a small map $\pat{\sigma} \co \typecomp{X}{\alpha} \to \UV$. Applying $\Pi_p$ along $\typecomp{X}{\alpha}\to X$ gives a small map $q\co\typecomp{X}{\Pi_{\alpha}\pat{\sigma}} \to X$ classified by $\Pi_{\alpha}\pat{\sigma} \co X\to \UV$.  Since $\Pi_p$ preserves monomorphisms, $q\co\typecomp{X}{\Pi_{\alpha}\pat{\sigma}} \to X$ is monic and so determines a subobject of $X$, which by definition is
$\forall_p(S) \mono X$.
The characteristic map of this subobject is then $\forall_{\alpha}(\sigma) \co X\to \iprop$. 
Finally, composing again with the inclusion $\pat{-}\co \iprop\mono\UV$ we can compare the maps
\begin{equation}
 \label{eq:forallvPi}
\pat{\forall_\alpha(\sigma)} \co X\to \UV  \mathrlap{,}  \qquad \Pi_{\alpha}\pat{\sigma} \co X\to \UV \mathrlap{.}
\end{equation}
But now there is no reason why these maps should be equal.  What we can say is that their supports agree, because they are the characteristic maps of the same subobject of $X$, and therefore
\[
\forall_\alpha(\sigma) = \supp\pat{\forall_\alpha(\sigma)} = \supp(\Pi_{\alpha}\pat{\sigma}).
\]

Bearing this lesson in mind, we have the following relationship between the operations on propositions $\sigma\co X\to\Omega$ and on types $\alpha\co X\to\UV$, stated in terms of equations between maps into $\iprop$.  
This is one variant of the usual propositions-as-types translation (see  \ref{prop:prop-as-types-isos} below for another). 

\begin{prop} \label{prop:prop-as-types} The following equalities hold for maps into $\iprop$, where $\sigma \, , \tau \co X \to \iprop$.
\begin{enumerate} 
\item $\ttrue = \supp(\unit) $,
\item $\tfalse = \supp(\emptytp)$,
\item $\sigma \land \tau =  \supp(\pat{\sigma}  \times \pat{\tau})$,
\item $\sigma \lor \tau  =  \supp(\pat{\sigma}  + \pat{\tau} )$,
\item $\sigma \Rightarrow \tau  =  \supp(\pat{\sigma}  \to \pat{\tau})$.
\end{enumerate} 
For $\alpha \co X \to \UV$ and $\sigma \co \typecomp{X}{\alpha} \to \Omega$,
\begin{enumerate} \setcounter{enumi}{5} 
\item \label{item:patA} ${\forall_\alpha (\sigma) }  =  \supp(\Pi_\alpha \pat{ \sigma} )$,
\item\label{item:patE} ${ \exists_\alpha(\sigma) }   = \supp(\Sigma_\alpha \pat{\sigma} ) $.
\end{enumerate} 
\end{prop}

\begin{rmk}  \label{diag:patuquantifier} The equations of \cref{prop:prop-as-types} can be expressed 
as diagrams in $\tE$ by representing the operations on subobjects and on small maps internally as maps involving $\Omega$ 
 and $\UV$ (as shown in~\cite{Awo18-naturalModels}).
We illustrate this only for the quantifiers. 
Let $P \co \tE \to \tE$ be the polynomial
functor associated to the map $\dspu \co \UVptd \to \UV$, as defined in~\eqref{equ:poly-dspu}.
Dependent sums and existential quantifiers then give rise to maps
$\Sigma \colon  P(\UV) \to \UV$ and $\exists \colon P(\Omega) \to \Omega$, respectively.
The following square 
\begin{equation*}
\begin{tikzcd}
P(\iprop) \ar[r, "{\exists}"] \ar[d, swap,  "P\pat{-}"] 
& \iprop \\
P(\UV) \ar[r, swap, "\Sigma"] 
& \UV  \ar[u,swap, "\supp"]
\end{tikzcd}
\end{equation*}
then commutes, since the equation \ref{item:patE} holds for all $\alpha \co X \to  \UV$ and $\sigma \co \typecomp{X}{\alpha} \fun \iprop$. 
The direct analogue for the universal quantifier also holds, of course, by \ref{item:patA}; 
but consider instead the following diagram, corresponding to the two maps in~\eqref{eq:forallvPi}:
\begin{equation}
\label{diag:patequantifier}
\begin{tikzcd}
{P(\iprop)} \ar[r, "{\forall}"] \ar[d, swap,  "P\pat{-}"] 
& \iprop \ar[d, "\pat{-} "] \\
{P(\UV)} \ar[r, swap, "\Pi"]  
& \UV \mathrlap{.}
\end{tikzcd}
\end{equation}
Since the maps in this diagram land in $\UV$, rather than $\iprop$, we cannot use the universal property of $\iprop$ to conclude that it commutes.  Nonetheless, it can be shown to commute ``up to isomorphism'', which we make precise next. 
\end{rmk}

\subsection*{The internal category of types}\label{subsec:categoryoftypes}

We have emphasized the analogy between the small map classifier $\UV$ and the subobject classifier $\iprop$.
There are, of course, also some important differences; in particular, the universal property of $\iprop$ is stronger than that of $\UV$, in that a subobject $S \mono X$ has a \strong{unique} characteristic map from $X$ to $\iprop$, while for a small map $p \co A \to X$, a classifying map from $X$ to $\UV$ will usually not be unique.  As was shown in~\eqref{equ:char-map-of-small-map}, however, there is a canonical choice of the classifying map of a given small map; moreover, this map is in fact \emph{unique up to isomorphism}, in a sense that we now make precise.

As was noted in (\cref{subsec:SOC}), for each $X \in \tE$, the poset of subobjects $\Sub(X)$ has a stable Heyting algebra structure, which is mirrored by operations on morphisms from $X$ to $\iprop$.  Since the isomorphism $\Sub(X)\ \cong\ \tE(X, \iprop)$ is natural in $X$, by the Yoneda lemma there are associated operations on $\Omega$ making it an internal Heyting algebra in $\tE$.  In particular, $\Omega$ is an internal poset, with an ordering relation 
$\{ (x, y) \oftype \iprop\times\iprop \ | \ x \leq y \} \mono \iprop\times\iprop$, 
satisfying the condition that for subobjects $S, T \mono X$ with characteristic maps $\sigma, \tau \co X \to \iprop$, one has 
$S\leq T$  if and only if there is a (necessarily unique) lift as indicated in the following diagram.
\begin{equation*}
\begin{tikzcd}
& \{ (x, y) \ | \ x \leq y \} \ar[d, tail] \\
X \ar[r,swap, "{(\sigma, \tau)}"] \ar[ur, dotted] & \iprop\times\iprop \rlap{.}
\end{tikzcd}
\end{equation*}

In much the same way, there is an internal \emph{category} structure on the object $\UV$,
\[
\catUmor \to \UV\times \UV,
\]
such that for any small maps $A \to X$ and $B \to X$, with classifying maps $\alpha, \beta \co X \to \UV$, the maps $f \co A\to B$ in the slice category over $X$ correspond (uniquely!)~to lifts~$\vartheta$ of~$(\alpha, \beta) \co X \to \UV\times\UV$, as indicated in the following diagram.
\begin{equation*}
\begin{tikzcd}
& \catUmor \ar[d] \\
X \ar[r,swap, "{(\alpha, \beta)}"] \ar[ur, dotted, "\vartheta"] & \UV\times\UV \rlap{.}
\end{tikzcd}
\end{equation*}
Moreover, since by \eqref{diag:context-extension-ess-surj} the assignment of $p_\alpha\co \typecomp{X}{\alpha} \to X$ to $\alpha \co X\to\UV$ is essentially surjective onto the objects in the category $\tS/_{X}$ of small maps, we therefore have an equivalence of categories,
\begin{equation}\label{eq:catoftypes}
\tS/_{X}\ \simeq\ \underline{\tE}(X, \catU),
\end{equation}
where $\catU = (\UV, \catUmor)$ may be called the \emph{internal category of types}, and the category $\underline{\tE}(X, \catU)$ consists of objects  $\tE(X, \UV)$ and morphisms $\tE(X, \catUmor)$.  
This can be done internally in any locally cartesian closed category using only the dependent product structure, via the so-called internal full subcategory construction. In our context, it is possible to give an alternative 
construction, which we limit ourselves to sketching since we shall make no essential use of \eqref{eq:catoftypes}.

First recall the definition \eqref{equ:hofmann-streicher-universe} of the set $\UV(c)$, for $c \in \cC$, and consider the corresponding set of morphisms,
\begin{equation}
\label{eq:catoftypesmorphisms}
\catUmor(c) \defeq \mathrm{Mor} \big[ \op{(\cC/_{c})} \, ,  {\cSet}_\kappa \big]  \, .
\end{equation}
Explicitly, an element $\vartheta \in \catUmor(c)$ is a natural transformation $\vartheta \co \alpha\to \beta$ of presheaves $\alpha,\beta \in \UV(c)$ on the slice category $\cC/_{c}$ whose values are small sets.  The domain and codomain maps are just those of the functor category $\big[ \op{(\cC/_{c})} \, ,  {\cSet}_\kappa \big]$, as are the identities and composition. Naturality with respect to $d\to c$ in $\cC$ is given by ``whiskering'' (i.e.\ precomposition) with the composition functor $\cC/d \to \cC/c$. 

Now let $c \in \cC$ and assume $\alpha,\beta \co \yon c \to \UV$, with associated display maps $p_\alpha\co\typecomp{\yon c}{\alpha} \to \yon c$ and $p_\beta\co\typecomp{\yon c}{\beta} \to \yon c$.  A map $h \co p_\alpha \to p_\beta$ in $\tS/_{\yon c}$ is a morphism in
\begin{equation}
\label{eq:catoftypesequivalence}
\big({{\cSet}_\kappa}^{\op{\cC}}\big){/\yon c}\ \simeq\ {{\cSet}_\kappa}^{\op{(\cC{/c})}},
\end{equation}
and therefore corresponds to a unique element $\vartheta$ of $\catUmor(c)$ by \eqref{eq:catoftypesmorphisms}.  One then checks that the equivalence of categories \eqref{eq:catoftypesequivalence} indeed relates $\alpha \co \yon c \to \UV$ to the display map $p_\alpha\co\typecomp{\yon c}{\alpha} \to \yon c$ in a functorial way (see \eqref{diag:context-extension-repr}).

Finally, two maps $\alpha,\beta \co X \to \UV$ are said to be \emph{isomorphic} if they are so as objects in the category $\underline{\tE}(X, \catU)$ of types over $X$ \eqref{eq:catoftypes}.  In particular, it is in this sense that the diagram~\eqref{diag:patequantifier} commutes up to isomorphism.
Indeed, using this notion, we can give a more precise statement than Proposition \ref{prop:prop-as-types} of the relationship between propositions $\sigma\co X\to\Omega$ and types $\alpha\co X\to\UV$. 

\begin{prop} \label{prop:prop-as-types-isos} The following equalities and isomorphisms hold for maps into $\UV$.
\begin{enumerate} 
\item $\pat{\ttrue} \cong \unit$,
\item $\pat{\tfalse} \cong \emptytp$,
\item $\pat{\sigma \land \tau}  \cong  \pat{\sigma}  \times \pat{\tau}$,
\item $\pat{\sigma \lor \tau}  = \trunc \big(  \pat{\sigma}  + \pat{\tau} \big)$,
\item $\pat{\sigma \Rightarrow \tau}  \cong  \pat{\sigma}  \to \pat{\tau}$.
\end{enumerate} 
For $\alpha \co X \to \UV$ and $\sigma \co \typecomp{X}{\alpha} \to \Omega$,
\begin{enumerate} \setcounter{enumi}{5} 
\item $\pat{ \forall_\alpha (\sigma) }  \cong  \Pi_\alpha \pat{ \sigma} $,
\item $\pat{ \exists_\alpha(\sigma) }   =  \trunc (\Sigma_\alpha \pat{\sigma} )$.
\end{enumerate} 
\end{prop}

\begin{proof}
We have an adjunction $\supp \dashv \{-\}$, so that for all $\sigma \co X\to \Omega$ and $\alpha \co X\to \UV$, there is a natural isomorphism
$\Omega^X ( \supp(\alpha) , \sigma)\ \cong\ \catU^X (\alpha , \pat{\sigma})$.
Now use the fact that $\supp(\pat{\sigma}) = \sigma$  and $\pat{\supp(\alpha )} = \trunc (\alpha)$.
For the two equalities, use~\cref{prop:prop-as-types}.
\end{proof}


\section{The type theory of a presheaf category} 
\label{sec:types}

\subsection*{The category with families} 
The aim of this section is to introduce precisely a dependent type theory associated to the presheaf category $\tE$. In order to do so, we begin by saying how $\tE$ determines a category with families via the universe $\pi\co E\to\UV$ (as in \cite{Awo18-naturalModels}). We then consider what forms of type are supported by this category with families. We will not specify all the deduction rules that are valid in $\tE$, but rather focus on those that will be most important for our applications, namely those concerning dependent sums, dependent products and the subobject classifier. 

In \cref{sec:cof,sec:unifib}, we will extend this type theory with additional forms of type, corresponding to additional structure that may be assumed on $\tE$. 
We begin by introducing some terminology.

\begin{defn}[The category with families $\mathcal{T}_{\tE}$] \leavevmode
\begin{itemize} 
\item A \strong{context} is an object of $\tE$. We use letters $\Gamma, \Delta, \ldots$ to denote contexts.
\item For a context $\Gamma$, a \strong{type in context} $\Gamma$ is a morphism $\alpha \maps \Gamma \to \UV$ in $\tE$. In this case, we write 
$\Gamma \vdash \alpha \colon \UV$.
\item Given two types $\alpha_1$ and $\alpha_2$ in context $\Gamma$, we say that $\alpha_1$ and $\alpha_2$ are
\emph{judgementally equal} if they are equal as morphisms from $\Gamma$ to $U$. In this case, we write $\Gamma \vdash \alpha_1 = \alpha_2 \colon \UV$.
\item For a context $\Gamma$ and a type $\alpha$ in context $\Gamma$, an \strong{element of} $\alpha$ in context $\Gamma$ is a map $a \co \Gamma \to \UVptd$ such that
 \begin{equation}
\label{equ:tmGammaA}
\begin{tikzcd}
\Gamma \ar[d, equal] \ar[r, "a"] & \UVptd \ar[d, "{\dspu}"] \\
\Gamma \ar[r, swap, "\alpha"] & \UV \mathrlap{.} 
\end{tikzcd}
\end{equation} 
 commutes. In this case, we write $\Gamma \vdash a \colon \alpha$.
 \item For two terms $a_1$ and $a_2$ of type $\alpha$ in context $\Gamma$, we say that $a_1$ and $a_2$ are \emph{judgementally equal} elements of $\alpha$ if they are equal as morphisms from $\Gamma$ to $E$. We then write $\Gamma \vdash a_1 = a_2 \colon \alpha$.
\end{itemize}
\end{defn}

For a context $\Gamma$ and a type $\alpha$ in context $\Gamma$, we have a new
context $\typecomp{\Gamma}{\alpha}$, obtained by the pullback in \eqref{diag:context-extension}. This is the operation
of \strong{context extension}.  We sometimes refer to a map $a$ with codomain $\UVptd$ simply as an \emph{element}, leaving
implicit that it is an element of the type $\dspu \circ a$.

\begin{rmk}\label{rmk:sections} Thus elements of a type $\alpha$ in context $\Gamma$ are in bijective correspondence with sections of its display map
$\Gamma.\alpha \to \Gamma$.  Indeed, an element $a$ of $\alpha$ in context $\Gamma$, as defined above, determines the diagram
  \[
  \begin{tikzcd}
\Gamma \ar[drr, bend left = 30, "{a}"] \ar[ddr, swap, "{\id_\Gamma}", bend right = 30] \ar[dr, dotted, "{(\id_\Gamma,a)}"] & & \\
 & \typecomp{\Gamma}{\alpha} \ar[r] \ar[d] & \UVptd \ar[d, "\dspu"] \\
 & \Gamma \ar[r, swap, "{\alpha}"] & \UV \mathrlap{.}
 \end{tikzcd}
\]
\end{rmk}

The \strong{empty context} is the terminal object $1$ of $\tE$. A type in the empty context, which is just a map~$\alpha \co 1 \to \UV$ is called a \strong{closed type}. In this case, we simply write
$\vdash \alpha \co \UV$. As a special case of~\eqref{diag:context-extension}, for a closed type $\alpha$, we obtain 
\[
\begin{tikzcd}
1.\alpha  \ar[r] \ar[d] & \UVptd \ar[d, "\dspu"] \\
1 \ar[r, swap, "\alpha"] & \UV \mathrlap{.}
\end{tikzcd}
\]
Such closed types correspond up to isomorphism to small presheaves $A\co \op{\cC} \to \cSet_\kappa$.

\begin{defn} Let $\Gamma$ and $\Delta$ be contexts. A \strong{context morphism}
from $\Delta$ to $\Gamma$ is a map 
$t \co \Delta \to \Gamma$ in $\tE$.
\end{defn} 

Context morphisms, which may be regarded as (tuples of) \emph{terms}, act on types and elements via the operation of substitution, which we now define.

\begin{nom}[Substitution] \label{nom:substitution} \label{lem:subst}
Fix a context morphism $t \maps \Delta\to\Gamma$. For a type~$\alpha$ in context~$\Gamma$, we define a type $\alpha(t)$ in context $\Delta$,
obtained by \strong{substitution} of $t$ in $\alpha$, by letting
\begin{equation}
\label{equ:subst-terms-into-types}
\alpha(t) \defeq \alpha \circ t \mathrlap{.}
\end{equation}
 We then have the following diagram of the associated context extensions, in which the left square is also a pullback.
\[
\begin{tikzcd}
{\Delta. \alpha(t) } \ar[d] \ar[r] &  \Gamma.\alpha \ar[d] \ar[r] & \UVptd \ar[d, "{\dspu}"] \\
\Delta \ar[r, swap, "t"] & \Gamma \ar[r, swap, "\alpha"] & \UV \mathrlap{.} 
\end{tikzcd}
\]
If $a$ is term of type $\alpha$ in context $\Gamma$, we define  the \strong{substitution} of $t$ in $a$ to be the element $a(t)$ of type $\alpha(t)$ in context $\Delta$ defined by letting $a(t) \defeq a \circ t$.  Note that $a(t)$ is indeed an element of  $\alpha(t)$ since
\[
\begin{tikzcd}
\Delta \ar[d, equal]  \ar[r, "t"] & \Gamma \ar[r, "a"]  \ar[d, equal] & \UVptd \ar[d, "{\dspu}"] \\
\Delta \ar[r, swap, "t"]  & \Gamma  \ar[r, swap, "\alpha"] & \UV 
\end{tikzcd}
\]
commutes if the right hand square does. Note that substitution satisfies the expected laws $\alpha(t)(s) = \alpha(ts)$, etc.

We introduce some auxiliary notation for substitutions of a context morphism of the form $(\id_\Gamma, a) \co \Gamma \to \Gamma.\alpha$ as in \cref{rmk:sections}. In this case, given a type $\beta$ in context $\typecomp{\Gamma}{\alpha}$, we write the substitution of $(\id_\Gamma, a)$ in $\beta$ as $\beta(a)$, rather than the more precise $\beta(\id_\Gamma, a)$. Similarly, if $b$ is a term of type $\beta$
in context $\Gamma.\alpha$, we write $b(a)$ instead of $b(\id_\Gamma, a)$. The interpretation of these terms and types is displayed in the diagram
\[
\begin{tikzcd}
\Gamma \ar[d, equal]  \ar[r, "{(\id_\Gamma, a)}"] &[10pt] \Gamma.\alpha \ar[r, "b"]  \ar[d, equal] &[10pt] \UVptd \ar[d, "{\dspu}"] \\
\Gamma \ar[r, swap, "{(\id_\Gamma, a)}"]  & \Gamma.\alpha \ar[r, swap, "\beta"] & \UV \mathrlap{.} 
\end{tikzcd}
\]
This notation allows us to derive the substitution rule
\[
\begin{prooftree}
\Gamma \vdash a \oftype \alpha \qquad \Gamma, \alpha \vdash b \oftype \beta
\justifies
\Gamma \vdash b(a) \oftype \beta(a) \mathrlap{.}
\end{prooftree}
\]
\end{nom}

\begin{nom}[Weakening]  \label{nom:weakening}
When we perform a substitution along a map of the form $\dsp{\beta} \co \Gamma.\beta \to \Gamma$
for some $\beta \co \Gamma \to \UV$, we obtain the diagram
\[
\begin{tikzcd}
\Gamma.\beta.\alpha \ar[d] \ar[r] &  \Gamma.\alpha \ar[d] \ar[r] & \UVptd \ar[d, "{\dspu}"] \\
\Gamma.\beta \ar[r, swap, "\dsp{\beta}"] & \Gamma \ar[r, swap, "\alpha"] & \UV \mathrlap{.} 
\end{tikzcd}
\]
This corresponds to the rule of \strong{weakening}, which we write as
\[
\begin{prooftree}
\Gamma \vdash \alpha \co \UV \qquad
\Gamma \vdash \beta \co \UV
\justifies
\typecomp{\Gamma}{\alpha} \vdash \beta(\dsp{\alpha}) \co \UV \mathrlap{.}
\end{prooftree}
\]
Note that we write $\beta(\dsp{\alpha})$ rather than the more familiar $\beta$ in the conclusion, as this will help us keep track of the context in which we are working. Analogously, we have weakening for elements of types, which we write as
\[
\begin{prooftree}
\Gamma \vdash \alpha \co \UV \qquad
\Gamma \vdash b \co \beta 
\justifies
\typecomp{\Gamma}{\alpha} \vdash b(\dsp{\alpha}) \co \beta(\dsp{\alpha}) \mathrlap{.}
\end{prooftree}
\]
\end{nom}

\begin{nom} Observe that in the context extension diagram \eqref{diag:context-extension}, viz.\
\begin{equation}
 \begin{tikzcd}
      \typecomp{\Gamma}{\alpha} \ar[r, "q_\alpha"] \ar[d, swap, "\dsp{\alpha}"]  & \UVptd
      \ar[d, "\dspu"] \\
       \Gamma \ar[r, swap, "\alpha"] & \UV \mathrlap{,}
    \end{tikzcd}   
\end{equation}
the top arrow $q_\alpha \co \typecomp{\Gamma}{\alpha} \to E$  gives a canonical element $q_\alpha \co \alpha(\dsp{\alpha})$ in the weakened context~$\typecomp{\Gamma}{\alpha}$.  This corresponds to the syntactic rule of introducing a variable, usually expressed in dependent type theories by a judgement of the form $\Gamma, x\co\alpha \vdash x \co \alpha$.
\end{nom}

We recall a fundamental fact from~\cite{Awo18-naturalModels}.

\begin{thm} Contexts, context morphisms, types, elements, and the context extension operation, as defined above, form a category with families, denoted $\mathcal{T}_{\tE}$. \qed
\end{thm}

\begin{rmk}\label{rem:sliceCwF}
By \cref{rmk:universe-in-slice}, the slice category $\tE/_{X} \cong \Psh(\int\! X)$ also has a universal small map $\pi_X \maps E_X \to U_X$, and so we can apply the foregoing construction to obtain a category with families $\mathcal{T}_{\tE/_{\!X}}$.
Using \eqref{eq:small-map-classifier-in-slice-sections}, and similar reasoning, one sees that the contexts, types, elements, etc., of $\mathcal{T}_{\tE/_{\!X}}$ correspond to those of $\mathcal{T}_{\tE}$ in which the context has the form $X.(-)$, as in e.g.\  $X.\Gamma \vdash a\co \alpha$.
We shall not distinguish between these isomorphic type theories.
\end{rmk}

\subsection*{Dependent sum and dependent product types} 
The category-theoretic structure of $\tE$ and the closure properties of the class $\tE$ of small maps discussed in~\cref{sec:presheaves} correspond to
well-known type-theoretic constructions. In particular, the  type theory of $\tE$ has empty type $\emptytp$, a unit type $\unit$, sum types $\alpha + \beta$,
dependent sum types $\Sigma_\alpha(\beta)$, dependent
product types $\Pi_\alpha(\beta)$, a type of propositions $\Omega$ and subset types~
$\propcomp{x \oftype \alpha}{\sigma}$. 
With these product types $\alpha \times \beta$ and function types $\alpha \rightarrow \beta$ can be defined, as explained 
in~\cref{nom:fun-and-prod-types} below. It also has a type of natural numbers $\mathsf{N}$, as well as many other inductively defined types.  We shall not list the rules for all these types, limiting ourselves to those that will be most relevant and illustrative, namely dependent sums, dependent products and propositions. 

\cref{thm:rules-sigma} and \cref{thm:rules-pi} and their proofs, for which we refer to~~\cite{Awo18-naturalModels}, make use of the convention introduced in \cref{nom:substitution}. Even if we speak of `rules' in their statements, it should be pointed out that these are claims about semantic objects, not generating rules of a dependent type theory.

\begin{prop}[Rules for dependent sum types] \label{thm:rules-sigma} 
The following rules are valid in~$\tE$.
\begin{itemize} 
\item Formation: if $\Gamma \vdash \alpha \oftype \UV$ and $\typecomp{\Gamma}{\alpha} \vdash \beta \oftype \UV$, then $\Gamma \vdash \Sigma_\alpha(\beta) \oftype \UV$.
\item Introduction: if $\Gamma \vdash a \oftype \alpha$ and $\Gamma \vdash b \oftype \beta(a)$ then
$\Gamma \vdash \tpair(a,b) \oftype \Sigma_\alpha(\beta)$.
\item Elimination: if $\Gamma \vdash t \co \Sigma_\alpha(\beta)$ then 
$\Gamma \vdash \tproj_1(t) \co \alpha$ and $\Gamma \vdash \tproj_2(t) \co \beta(\tproj_1(t))$.
\item Computation: if $\Gamma \vdash a \oftype \alpha$ and $\Gamma \vdash b \oftype \beta(a)$
then $\Gamma \vdash \tproj_1(\tpair(a,b)) = a  \oftype \alpha$ and~$\Gamma \vdash \tproj_2(\tpair(a,b)) = b  \oftype \beta(a)$.
\item Expansion: if $\Gamma \vdash t \co \Sigma_\alpha(\beta)$ then 
$\Gamma \vdash t = \tpair( \tproj_1(t), \tproj_2(t)) \co \Sigma_\alpha(\beta)$. \qed
\end{itemize}
\end{prop}

\begin{nom} \label{nom:simplified-lambda} 
In the statement of the next proposition and below we make use not only of the convention introduced in \cref{nom:substitution}, but also write
$\lambda(b)$ instead of $\lambda( (\id_\Gamma, b^\sharp) )$ for simplicity.
\end{nom}

\begin{prop}[Rules for dependent product types] \label{thm:rules-pi} The following rules are valid in~$\tE$.
\begin{itemize}
\item Formation: if $\Gamma \vdash \alpha \oftype \UV$ and $\typecomp{\Gamma}{\alpha} \vdash \beta \oftype \UV$
then $\Gamma \vdash \Pi_\alpha \beta \oftype \UV$.
\item Introduction: if $\typecomp{\Gamma}{\alpha} \vdash  b \co \beta$ then $\Gamma \vdash \lambda(b) \co \Pi_\alpha(\beta)$.
\item Elimination: if $\Gamma \vdash t \oftype \Pi_\alpha(\beta)$ and $\Gamma \vdash a \oftype \alpha$
then $\Gamma  \vdash \funapp(t,a) \co \beta(a)$.
\item Computation: if $\typecomp{\Gamma}{\alpha} \vdash  b \co \beta$ and $\Gamma \vdash a \oftype \alpha$ then $\Gamma \vdash \funapp \big( \lambda(b), a \big) = b(a)  \co \beta(a)$.
\item Expansion: if $\Gamma \vdash t \oftype \Pi_\alpha(\beta)$ then $\Gamma \vdash 
\lambda( \funapp(t(p_\alpha), q_\alpha) = t 
 \co \Pi_\alpha(\beta)$. \qed
\end{itemize}
\end{prop}

\begin{nom}[Product and function types] \label{nom:fun-and-prod-types} The rules for dependent sum and dependent product types of~\cref{thm:rules-sigma,thm:rules-pi}, in combination with the weakening operation of~\cref{nom:weakening}, allow us to define product and function types and derive 
deduction rules for them. For $\alpha \co \Gamma \to \UV$ and $\beta \co \Gamma \to \UV$, we define the product type $\alpha \times \beta \co
\Gamma \to \UV$ and the function type $\alpha \to \beta \co \Gamma \to \UV$ by letting
\[
\alpha \times \beta \defeq \Sigma_\alpha \beta(\dsp{\alpha}) \mathrlap{,} \qquad
\alpha \to \beta \defeq \Pi_\alpha \beta(\dsp{\alpha}) \mathrlap{,} 
\]
where we used the notation for weakening introduced in~\eqref{nom:weakening}. We can then derive the deduction rules 
\[
\begin{prooftree}
\Gamma \vdash \alpha \co \UV \quad
\Gamma \vdash \beta \co \UV
\justifies
\Gamma \vdash \alpha \times \beta \co \UV \mathrlap{,}
\end{prooftree} \qquad
\begin{prooftree}
\Gamma \vdash a \co \alpha  \quad
\Gamma \vdash b \co \beta
\justifies
\Gamma \vdash \tpair(a,b) \co \alpha \times \beta 
\end{prooftree} 
\]
for product types, and the rules
\[
\begin{prooftree}
\Gamma \vdash \alpha \co \UV \quad
\Gamma \vdash \beta \co \UV
\justifies
\Gamma \vdash \alpha \rightarrow \beta \co \UV \mathrlap{,}
\end{prooftree} 
\qquad
\begin{prooftree}
\typecomp{\Gamma}{\alpha} \vdash b \co \beta(\dsp{\alpha})
\justifies
\Gamma \vdash \lambda(b) \co \alpha \rightarrow \beta  
\end{prooftree} 
\]
for function types, where we made use of the simplification introduced in~\cref{nom:simplified-lambda}.
\end{nom}

\begin{nom}[Constant functions] \label{not:constant-functions}
 In order to be able to manipulate constant functions conveniently, we introduce a further abuse of notation on top of the
one discussed in~\cref{nom:simplified-lambda}. If $\Gamma \vdash b \co \beta$, we write $\Gamma \vdash \lambda(b) \co \alpha \to \beta$ instead of
 $\Gamma \vdash \lambda ( b(\dsp{\alpha})) \co \alpha \to \beta$, thereby omitting mention of the weakening in this particular case. This convention will be used frequently in \cref{sec:cof}.
\end{nom} 

\begin{rmk}\label{rem:up-to-iso-types} The rules for dependent sum and dependent product types stated in \cref{thm:rules-sigma} and \cref{thm:rules-pi}  imply also many familiar properties of these types.  In particular, if~$\Gamma \vdash 
\alpha \colon \UV$ and $\typecomp{\Gamma}{\alpha}\vdash \beta \co \UV$ and $\typecomp{\Gamma.\alpha}{\beta} \vdash  
\rho \co \UV$, we can prove  that there are isomorphisms of types,
\begin{align}
\label{prop:dsums-fubini}
\textstyle
\dsm{x \co \alpha} \dsm{y \oftype \beta}  \rho 
&  \textstyle
\ \iso\ \dsm{z \co \dsm{x \oftype \alpha} \beta}  
\rho(\tproj_1(z), \tproj_2(z))   
\mathrlap{,} \\ 
\label{prop:forcing.AC}
\textstyle
\dprd{x \oftype \alpha} \dsm{y  \oftype \beta} \rho & 
\textstyle
\ \iso\  
\dsm{u \oftype \dprd{x \oftype \alpha} \beta}  \dprd{x \oftype \alpha} \rho(x, \funapp(u,x)) \mathrlap{.}
\end{align}
Also, if $\Gamma \vdash 
\alpha \colon \UV$ and $\typecomp{\Gamma}{\alpha} \vdash \beta \co \UV$, then there is an isomorphism, 
\begin{equation}
\label{prop:maps-out-of-dsm}
\textstyle
\big(\dsm{x \oftype \alpha} \beta \big) \fun \rho\ \iso\ \dprd{x\oftype \alpha} \big( \beta \fun \rho \big) \mathrlap{.}
\end{equation}

The soundness of the rules in $\tE$, \cref{thm:rules-sigma}
and \cref{thm:rules-pi}, then implies that there are corresponding isomorphisms in the category of types over $\Gamma$, 
\[
\tS/_{\Gamma}\ \simeq\ \underline{\tE}(\Gamma, \catU),
\]
introduced in \eqref{eq:catoftypes} of section \ref{subsec:categoryoftypes}.  
In particular, it follows that the internal category of types~$\catU$ is locally cartesian closed, much as the subobject classifier~$\iprop$ is an internal Heyting algebra.
\end{rmk}

\subsection*{The type of propositions} 
We conclude this section by fixing some notation for the closed type associated to the subobject classifier of $\tE$, which we shall also write $\Omega$. The elements of this type, to be regarded as propositions, will be written using the standard logical notation,
\[
\ttrue \mathrlap{,} \quad \tfalse \mathrlap{,} \quad  a =_\alpha b \mathrlap{,}  \quad
\sigma \land \tau \mathrlap{,}  \quad
\sigma \lor \tau \mathrlap{,}  \quad 
\sigma \implies \tau \mathrlap{,}  \quad
 (\forall x \co \alpha) \tau(x) \mathrlap{,}  \quad
 (\exists x \co \alpha) \tau(x) \mathrlap{.}
\]
Note that the proposition $a =_\alpha b$ is a well-formed element of type $\Omega$ in context $\Gamma$ when
$a$ and $b$ are well-formed elements of type $\alpha$ in context $\Gamma$. We refer to this as \strong{propositional equality}; it is to be distinguished from judgemental equality $a = b \co \alpha$, relating two elements of type $\alpha$ (and from $p = q \co \Omega$, relating two elements of type $\Omega$). In spite of this, we will sometimes drop the subscript from propositional equality, when this does not create confusion (see also \cref{rmk:Equality} below).

\begin{defn}  \label{defn:valid}
We say that a proposition $\sigma$ in context $\Gamma$ is \strong{valid} if 
$\sigma \co \Gamma \to \iprop$ factors through $\ttrue \co 1 \to \iprop$.  In that case, we write $\Gamma \vdash \sigma$.
Similarly, we say that $\Gamma, \varphi_1, \ldots, \varphi_n$ \emph{implies} $\varphi$,
if the subobject $\propcomp{\Gamma}{\varphi_1 \land \ldots \land \varphi_n}$ factors through $\propcomp{\Gamma}{\varphi}$.
In this case, we write  $\Gamma, \varphi_1, \ldots, \varphi_n \vdash \varphi$.
 \end{defn}

Note that $\sigma$ is valid exactly when the associated subobject $\propcomp{\Gamma}{\sigma} \mono \Gamma$
admits a section, and hence is isomorphic to $\Gamma$.

\begin{rmk}\label{rmk:Equality} Propositional equality $a=_\alpha b \co \Omega$ serves the purpose of the \emph{equality types} $\mathsf{Eq}_{\alpha}(a,b)$ sometimes used in a system of type theory without a type of propositions.  Here judgemental and propositional equality are equivalent, in the sense that the inference rules
\[
\begin{prooftree}
\Gamma \vdash a_1 \co \alpha \quad
\Gamma \vdash a_2 \co \alpha \quad
\Gamma \vdash a_1 =_\alpha a_2
\justifies
\Gamma \vdash a_1 = a_2 \co \alpha
\end{prooftree}
\qquad\qquad
\begin{prooftree}
\Gamma \vdash a_1 \co \alpha \quad
\Gamma \vdash a_2 \co \alpha \quad
\Gamma \vdash a_1 = a_2 \co \alpha
\justifies
\Gamma \vdash a_1 =_\alpha a_2  
\end{prooftree}
\]
are also valid. This is sometimes expressed by saying that the type theory of $\tE$ is \strong{extensional},
in contrast with formulations of Martin-L\"of type theories that are \strong{intensional}. 
\end{rmk}


\section{Kripke-Joyal forcing}
\label{sec:kjshott-definition}

\subsection*{Definition and basic properties} 
The aim of this section is to give the definition of the Kripke-Joyal semantics for the type theory $\mathcal{T}_{\tE}$ of the category $\tE$ of presheaves, introduced in \cref{sec:types}, and establish some of its key properties. In particular,
we unfold the general definition for each of the main type forming operations, so as
to facilitate the work of subsequent sections.

\begin{defn}  \label{defn:forcing.types} \label{defn:forcing.kjshott} 
Let $\alpha \colon X \to \UV$ and $x \maps \yon{c} \to X$.
\begin{enumerate}
\item For $a \maps \yon{c} \to \UVptd$, we say that 
$c$ \strong{forces} $a \colon \alpha(x)$, written $$c \forces a \colon \alpha(x)\,,$$ 
if the following diagram commutes:
\[
\begin{tikzcd}
\yon{c} \ar[r,  "{a}"] \ar[d, swap, "{x}"]  & \UVptd \ar[d, "\dspu"] \\
  X \ar[r, swap, "{\alpha}"] & \UV \mathrlap{.} 
 \end{tikzcd}
\]
 \item For $a, b \maps \yon{c} \to \UVptd$ such that $c \forces a \colon \alpha(x)$ and $c \forces b \colon  \alpha(x)$,  we say that 
$c$ \strong{forces} $a = b \colon \alpha(x)$, written $$c \forces a = b \colon \alpha(x)\,,$$ 
if $a$ and $b$ are equal maps in $\tE$.
 \end{enumerate}
 \end{defn}

 \begin{rmk} \label{thm:forcing-alternative}
 We provide some equivalent ways of regarding the forcing condition.
 \begin{enumerate}
 \item  \label{item:forcing-vs-entail} A map $a \co \yon{c} \to \UVptd$ such that $c \forces a \co \alpha(x)$ is the same thing as  an element of $\alpha(x)$ in context $\yon{c}$. Indeed, we can rewrite the diagram in \cref{defn:forcing.types} as
\[
\begin{tikzcd} 
\yon{c} \ar[rr, "a"] \ar[d, equal] & & \UVptd \ar[d, "\dspu"] \\
\yon{c} \ar[r, swap, "x"] & X \ar[r, swap, "\alpha"] & \UV \mathrlap{.}
\end{tikzcd}
\]
Consulting \eqref{equ:tmGammaA}, we see that $c \forces a \co \alpha(x)$ is therefore simply another notation for $\yon{c} \vdash a \co \alpha(x)$. Even if it is just an abbreviation, the forcing notation provides an alternate tool for working with the type theory of $\tE$,
when unfolded for the various type formers, as we shall see later.

\item A map $a \co \yon{c} \to \UVptd$ such that $c \forces a \co \alpha(x)$ is the same thing
 as a dotted map making the following diagram commute:
 \[
  \begin{tikzcd}
	 & \typecomp{X}{\alpha} \ar[d] \ar[r] 
	 & \UVptd \ar[d] \\
\yon{c} \ar[r, swap, "{x}"]  \ar[ur, dotted, bend left = 30]  
	& X \ar[r, swap, "{\alpha}"] 
	& \UV \mathrlap{.}
 \end{tikzcd}
 \]
 As we will see in \cref{prop:forcing.props-in-ctx}, our definition subsumes the Kripke-Joyal forcing for propositions, which 
 defines $c \forces \sigma(x)$ to mean that there is a (now necessarily unique) map in
  \[
  \begin{tikzcd}
	 & \propcomp{X}{\sigma} \ar[d, tail] \ar[r] 
	 & 1 \ar[d] \\
\yon{c} \ar[r, swap, "{x}"]  \ar[ur, dotted, bend left = 30]  
	& X \ar[r, swap, "{\sigma}"] 
	& \iprop\diadot{.}
 \end{tikzcd}
 \]
One of the key differences between forcing for types and forcing for propositions is that in
the forcing condition for types we need to record the map $a$, as it is not unique  in
general. 

\item\label{item:forcing-vs-section}  A map $a \co \yon{c} \to \UVptd$ such that $c \forces a \co \alpha(x)$ also corresponds uniquely to a section of the display map of ${\alpha(x)}\co \yon{c} \to \UV$,
\[
\begin{tikzcd}
\typecomp{\yon{c}}{\alpha(x)} \ar[d] \ar[rr] && \UVptd \ar[d] \\
\yon{c} \ar[r, swap, "{x}"]  \ar[u, dotted, bend left = 40]  & X \ar[r, swap, "{\alpha}"] & \UV\diadot{.}
 \end{tikzcd}
 \]
 \end{enumerate}
 \end{rmk}
 
The following statements gather some immediate consequences of \cref{defn:forcing.kjshott}.  

\begin{lemma} \label{thm:equality-pointwise}
Let $\alpha, \beta \co X \to \UV$. Then $\alpha = \beta$ if and only if  $c \forces \alpha(x) = \beta(x)$ for every~$c\in\cC$ and~$x \co \yon c \to X$. \qed
\end{lemma}

For the next lemma, recall that the \strong{monotonicity} of conventional Kripke semantics states that if $k \forces p$ and $j>k$ then $j\forces p$. Something similar holds in our setting as well.

\begin{lemma}[Monotonicity] \label{prop:forcing.monotonicity}  Let $\alpha \co X \to \UV$, 
$x \co \yon{c} \to X$ and $a \maps \yon{c} \to \UVptd$. If  \(c \forces a \oftype \alpha(x) \)  then  \(d \forces a(f)  \oftype \alpha (x f) \)
for every  \(f \maps d \to c\).
\end{lemma} 

\begin{proof} Just precompose the diagram \cref{defn:forcing.kjshott} with $f \co d \to c$. 
\end{proof}

A \strong{converse to monotonicity} might be stated as follows. Suppose that for every  \(f \maps d \to c\), we have $a_f \co \yon{d} \to \UVptd$ with $d \forces a_f  \oftype \alpha (x f)$, and moreover the uniformity condition $e \forces a_f(  g) = a_{fg} \oftype \alpha(x  fg)$ holds for all \(g \maps e \to d\).
Then there is $a \co \yon{c} \to \UVptd$ with $c \forces a \oftype \alpha(x)$, and moreover $d \forces a(f) = a_f \oftype \alpha(x  f)$
for all  \(f \maps d \to c\).  By the Yoneda lemma, this is trivial, however, because for~$a$ we can just take $a_{\id_{c}}$.  
What is not so trivial is the analogous statement  for a judgement of the form $X\vdash a \co \alpha$ when $X$ is not  representable. Since the presheaf $X$ is a colimit of representables, the validity of such judgements in the internal type theory $\mathcal{T}_\tE$ can be related to forcing semantics as follows.

\begin{prop} \label{prop:validity-and-forcing} 
Let $X \vdash \alpha \co \UV$. Suppose that for each ${x \maps \yon{c} \to X}$ there is a given element $a_x\co \yon{c} \to \UVptd$ such that
 \begin{equation}\label{eq:SandCforces}
c \forces a_x \oftype \alpha(x),
\end{equation}
and moreover, the following \strong{uniformity condition} holds for all $f \maps d \to c$,
\[
d \forces a_x(  f) = a_{x  f} \oftype \alpha(x  f)\,.
\]
Then there is a unique element $a \co X\to\UVptd$ such that 
\begin{equation}\label{eq:SandCvalid}
X\vdash a \co \alpha
\end{equation}
and, for all ${x \maps \yon{c} \to X}$,
\[
c \forces a(  x) = a_{x  } \oftype \alpha(x  )\,.
\]
\end{prop}

\begin{proof} 
Given any $X \vdash a \co \alpha$ we can set $a_x \defeq ax$ for each ${x \maps \yon{c} \to X}$, and we then have $c \forces a_x \oftype \alpha(x)$  by precomposing \eqref{equ:tmGammaA} with ${x \maps \yon{c} \to X}$.  The uniformity condition $d \forces a_x(  f) = a_{x  f} \oftype \alpha(x  f)$ holds because $a_x(  f) = (ax)f = a(x f) = a_{x  f}$. 

To see that any uniform family of elements $(a_x\co \yon{c} \to \UVptd)_{x \maps \yon{c} \to X}$ arises in this way from a unique  $X \vdash a \co \alpha$, note that, for each $c$, the family $(a_x\co \yon{c} \to \UVptd)_{x \maps \yon{c} \to X}$ determines a function $\mathbf{a}_c\co X(c) \to E(c)$, while the  uniformity condition means that the various functions~$\mathbf{a}_c$ are natural in $c$, so we indeed have a natural transformation $a \co X \to E$.  Recall that $X \vdash a \co \alpha$ means that this $a \co X \to E$ makes the diagram \eqref{equ:tmGammaA} commute.  This follows from the individual conditions~\eqref{eq:SandCforces}. 
 \end{proof} 

We can summarise \cref{prop:validity-and-forcing} as follows.

\begin{cor} \label{cor:forcing.dependent-elements} 
Let $\alpha \co X\to \UV$. Then the following data are in bijective correspondence:
\begin{enumerate}
\item\label{cor:forcing.dependent-elements(i)}  elements $a \maps  X \to \UVptd$ such that $X\vdash a\co\alpha$,  
\item\label{cor:forcing.dependent-elements(ii)} 
uniform families of elements $a_x \maps  \yon{c} \to \UVptd$  such that $c \forces a_x \co \alpha(x)$.
\end{enumerate}
Specifically, the bijection is defined by letting $a_x \defeq a(x)\co \alpha(x)$. \qed
\end{cor}

\begin{prop} \label{defn:KJforcing-in-E}
Let $\alpha \maps 1 \to \UV$ be a  closed type. 
Then the following data are in bijective correspondence:
\begin{enumerate}
\item elements $a \co 1 \to \UVptd$ such that $\vdash a \co \alpha$,
\item uniform families of elements $(a_c)_{c\in\cC}$, where $c \forces a_c \oftype \alpha$,  {i.e.}~such that
$d \forces {a_c}( f) = a_{d} \oftype \alpha$,  for all $f \maps d \to c$.
\end{enumerate}
\end{prop}

\begin{proof} Immediate from \cref{cor:forcing.dependent-elements}.
\end{proof}

When the equivalent conditions of \cref{defn:KJforcing-in-E}  hold, we may write $\tE \forces a \co \alpha$ and say that \emph{$\tE$ forces $a \co \alpha$} (or even just that \emph{$\tE$ forces $\alpha$}). In these terms \cref{cor:forcing.dependent-elements} implies the following soundness and completeness theorem for the $\vdash$ relation of the internal type theory $\mathcal{T}_{\tE}$ with respect to the Kripke-Joyal forcing relation $\forces$ (but also see \cref{Kripke-Joyal-completeness}).

\begin{thm}[Soundness and completeness of forcing] \label{cor:SandC}
Let  $\alpha \maps 1 \to \UV$ be a closed type and $a \co 1 \to \UVptd$. Then $\mathcal{T}_{\tE}\vdash a \co\alpha$ if and only if $\tE\forces a\co\alpha$. \qed
\end{thm}

\subsection*{Forcing in slice categories}\label{susec:slice-forcing}
  
Since any slice category of a presheaf topos  is also a presheaf topos, as $\Psh(\cC)/_{X} \cong\Psh(\int\!X)$, \cref{defn:forcing.kjshott} also determines a notion of forcing there (for the type theory $\mathcal{T}_{\tE/_{\!X}}$, see \cref{rem:sliceCwF}), which is related to forcing in~$\tE$ by the following, \cf also \cref{rmk:universe-in-slice}.

\begin{lemma} \label{lemma:forcing-in-slices}
Let $\alpha \maps X \to \UV$ and $x \maps \yon c \to X$. Then there is a bijection between 
For any  the following conditions are equivalent.
\begin{enumerate}
\item \label{item:total} maps $a \co \yon c \to \UVptd$ such that $c \forces  a \oftype \alpha(x)$ in  $\Psh(\cC)$.
\item \label{item:slice} maps $a \co \yon (c,x) \to E_X$ such that $(c, x) \forces a \oftype  \alpha$ in $\Psh(\int\!X)$. 
\end{enumerate}
\end{lemma}
   
\begin{proof} 
Under the identification of \eqref{eq:small-map-classifier-in-slice-sections} (and observing \cref{rem:sliceCwF}), the forcing condition 
\((c, x) \forces a \oftype \alpha \)
states that the diagram on the left below commutes in $\Psh(\int\!X)$, where~$\pi_X$ is the small map classifier of $\Psh(\int\!X)$, as in \cref{rmk:universe-in-slice}.  But by the specification of~$\pi_X$, this is equivalent to the commutativity in $\Psh(\cC)$ of the diagram on the right, which is what the condition $c \forces  a \oftype \alpha(x)$ means.  
\[
 \begin{tikzcd}
    \yon (c, x) \ar[r, "a"]   \ar[d, swap]  & \UVptd_X   \ar[d, "\dspu_X"] \\
  1 \ar[r, swap, "\alpha"] & \UV_X
\end{tikzcd}   
\qquad \qquad
\begin{tikzcd}
    \yon (c) \ar[r, "a"]   \ar[d, swap, "x"]  & \UVptd   \ar[d, "\dspu"] \\
  X \ar[r, swap, "\alpha"] & \UV
\end{tikzcd}   \vspace{-2em}
 \]
\end{proof}

\begin{rmk} Let $\alpha \co X\to \UV$ and reconsider the bijection of \cref{cor:forcing.dependent-elements} between
\begin{quote}
\begin{enumerate}
\item elements $X\vdash a\co\alpha$,
\item uniform families of elements $c \forces a_x \co \alpha(x)$, for all $x \maps \yon{c} \to X $.
\end{enumerate}
\end{quote}
In light of \cref{lemma:forcing-in-slices}, when we move to the slice category $\tE/_X$ where $\alpha \co X\to \UV$ becomes a \emph{closed type}, the bijection  is then between:
\begin{quote}
\begin{enumerate}
\item elements $\vdash a\co\alpha$,
\item uniform families of elements $(c,x)\forces a_{(c,x)} \co \alpha$, for all $(c,x) \in \int{X}$.
\end{enumerate}
\end{quote}
The latter statement is  \cref{defn:KJforcing-in-E}.  
\end{rmk} 

In conclusion, forcing a \emph{type in context} $X\vdash a \co\alpha$ in $\mathcal{T}_\tE$ is equivalent to forcing the corresponding \emph{closed type} $\vdash a \co\alpha$ in $\mathcal{T}_{\tE/_{X}}$.  Indeed, there is a bijection between the witnessing elements $a$.  We can summarise this as follows.

\begin{prop}\label{prop:forcing-eq-closedsliceforcing} 
Let $\alpha \co X\to \UV$ be a type. Then  the following data are in bijective correspondence:
\begin{enumerate}
\item elements $a \co X \to \UVptd$ such that $\tE/_X \forces a\co \alpha$,
\item uniform families of elements $(a_x)_{(c,x) \in \int X}$, where $c\forces a_x\co\alpha(x)$. \qed
\end{enumerate}
\end{prop}

\subsection*{Forcing for types}  We now unfold the definition of forcing with respect to each of the main type-forming operations of $\mathcal{T}_{\tE}$.  We make use thereby of the term constructors $\mathsf{pair}$, $\funapp$, etc., of $\mathcal{T}_{\tE}$ that were introduced in Section \ref{sec:types}.

\begin{prop}[Forcing for empty and unit type]  Let $c \in \mathbb{C}$.
\label{prop:forcing.empty-type} \label{prop:forcing.unit-type}
\begin{enumerate}
\item \label{item:forcing-empty} $c \notforces a \co \emptytp$ for all $a \co \yon{c} \to \UVptd$.
\item \label{item:forcing-unit}  $c \forces a \co \unit$ for a unique $a=* \co \yon{c} \to \UVptd$.
\end{enumerate}
\end{prop}

\begin{proof} 
For \cref{item:forcing-empty}, if $c \forces a \co \emptytp$, then by \cref{item:forcing-vs-section} of \cref{thm:forcing-alternative} we would have a section of the display map
$\yon{c}.\emptytp  \to \yon{c}$, but $\yon{c}.\emptytp = \{\yon{c}\ |\ \bot\} \mono \yon{c}$ is the initial object of $\tE$, so this is impossible, since $\yon{c}$ always has at least the element $\id \co\yon{c}\to\yon{c}$. 

For \cref{item:forcing-unit}, $c \forces a \co \unit$ for every section $a$ of the display map $\yon{c}.\unit \to \yon{c}$, but 
$\yon{c}.\unit = \{\yon{c}\ |\ \top \} \mono \yon{c}$ is the total subobject, so there is always exactly one such section.
\end{proof} 

\begin{prop}[Forcing for sum types]
\label{prop:forcing.sum}
Let \(\alpha, \beta \oftype X \to \UV \). For every $x \co \yon{c} \to X$, the following conditions hold.
\begin{enumerate} 
\item \label{item:forcing.sum-i} If \( c \forces a\co\alpha(x)  \) then \( c \forces \TypeFont{inl}(a) \colon \alpha(x) + \beta(x)$, 
	and if \( c \forces b\co\beta(x)  \) then \( c \forces  \TypeFont{inr}(b) \colon \alpha(x) + \beta(x)$.    
\item \label{item:forcing.sum-iii} If \(c \forces  t \co \alpha(x) + \beta(x) \) 
	then either \( c \forces t =  \TypeFont{inl}(a)\co\alpha(x)  \TypeFont{+} \beta(x) \)  for \( c \forces a\co\alpha(x)  \)  
	or  \( c \forces t =  \TypeFont{inr}(b)\co\alpha(x)  \TypeFont{+} \beta(x) \) for \( c \forces b\co\beta(x)  \).
\end{enumerate}
\end{prop}

\begin{proof}
The sum \(\alpha + \beta\colon X \to \UV\) 
classifies the canonical map from the coproduct $$[p_\alpha, p_\beta] \co X.\alpha + X.\beta \to X.$$ 
Now if \( c \forces a\co\alpha(x)  \) then we have a lift $a \co \yon{c} \to X.\alpha$ of $x \co \yon{c} \to X$ across the display map $p_\alpha \co X.\alpha \to X$. Composing with the coproduct inclusion $ \TypeFont{inl} \co X.\alpha  \to X.\alpha + X.\beta$ we obtain \( c \forces  \TypeFont{inl}(a) \colon \alpha(x) + \beta(x)$.  If instead \( c \forces b\co\beta(x)  \) then in the same way we obtain \( c \forces  \TypeFont{inr}(b) \colon \alpha(x) + \beta(x)$. Thus \ref{item:forcing.sum-i}.

For \ref{item:forcing.sum-iii}, if 
\(c \forces  t \co \alpha(x) + \beta(x) \), then there is a lift
\(t \maps \yon{c} \to X.\alpha + X.\beta\)
of $x \co \yon{c} \to X$ across $[p_\alpha, p_\beta]$. 
Since the representable $\yon{c}$ is indecomposable, the map $t$ must factor through one of the coproduct inclusions.  Thus we either have  \( c \forces t = \mathsf{inl}(a)\co\alpha(x) + \beta(x)\) where \( c \forces a\co\alpha(x)  \)  or we have \( c \forces t = \mathsf{inr}(b)\co\alpha(x) + \beta(x) \) where \( c \forces b\co\beta(x)  \), as required.
\end{proof}

\begin{prop}[Forcing for dependent sum types]
 \label{prop:forcing.dsum}
Let $\alpha \co X \to \UV$ and $\beta \co \typecomp{X}{\alpha}  \to \UV$. For every
$x \co \yon{c} \to X$, the following conditions hold.
\begin{enumerate} 
\item \label{eq:forcing.dsum-1} If $c \forces a \colon \alpha(x)$ and $c \forces b \co \beta(a)$, then $c \forces 
\mathsf{pair}( a, b) \colon  (\Sigma_\alpha \beta)(x)$.
\item \label{eq:forcing.dsum-2}
If $ c \forces t \oftype (\Sigma_\alpha \beta)(x)$, then $c \forces \tproj_1(t) \colon \alpha(x)$
and $c \forces \tproj_2(t) \colon \beta(\tproj_1(t))$.
\end{enumerate}
\end{prop}

\begin{proof} The statements (i) and (ii) follow from~\cref{thm:rules-sigma} (introduction and elimination) and \cref{item:forcing-vs-entail} of \cref{thm:forcing-alternative}.
\end{proof}

\begin{cor}[Forcing for product types]
 \label{prop:forcing.product-types}
For $\alpha, \beta \co X \to \UV$, the following hold.
\begin{enumerate} 
\item \label{item:forcing.product-i} If $c \forces a \colon \alpha(x)$ and $c \forces b \co  \beta(x)$, then $c \forces 
\mathsf{pair}(a, b) \colon (\alpha \times \beta) (x)$.
\item \label{item:forcing.product-ii} If $ c \forces t \oftype  (\alpha \times \beta)(x)$, then $c \forces \tproj_1(t) \colon  \alpha(x)$
and $c \forces \tproj_2(t) \colon \beta(x)$.
\end{enumerate}
\end{cor}

\begin{proof} Special case of \cref{prop:forcing.dsum} since product types are defined as dependent sum types (see \cref{nom:fun-and-prod-types}).
\end{proof}

The next proposition unfolds the forcing condition for elements of the dependent product type. We see another appearance of the uniformity condition, which will play a role in our study of uniform fibrations in \cref{sec:tfib}.  As the proof shows, this condition arises from the non-pointwise character of dependent products in presheaves, much like the consideration of ``future possible worlds'' in the familiar Kripke semantics of implication $p \Rightarrow q$.

\begin{prop}[Forcing for dependent product types]
\label{prop:forcing.dprd}
\leavevmode
Let $\alpha \co X \to \UV$ and $\beta \co \typecomp{X}{\alpha} \to \UV$. Then for every $x \co \yon{c} \to X$ we have the  following.
\begin{enumerate}

\item \label{item:forcing-lambda-i}
Suppose for every $f \maps d \to c$ and every $d \forces a \oftype \alpha(x f )$ we are given ${d \forces b_{(f,a)} \oftype \beta(x f, a)}$, 
satisfying the following uniformity condition for all $g \maps e \to d$, 
$$e \forces b_{(f,a)} (g) = b_{(fg,ag)} \co \beta(x fg, ag).$$
Then there is a map $b \co \typecomp{\yon{c}}{\alpha} \to \UVptd$ for which 
\[
c \forces  \lambda b \oftype ( \Pi_\alpha \beta  ) (x) \mathrlap{,}
\]
and for all $f \maps d \to c$ and $d \forces a \oftype \alpha(x f )$ we have
\[
{d \forces   \funapp ( (\lambda b) (f) , a ) =  b_{(f,a)} \oftype \beta(x f, a)} \mathrlap{.}
\]

\item \label{item:forcing-app} 
Suppose that $c \forces t \co  (\Pi_{\alpha} \beta)(x)$.  Then for every $f \co d \to c$ and every $d \forces a \co \alpha(x f)$, we have
\[
d \forces \funapp (  t (f), a) \co \beta(x f, a)\,,
\] 
and moreover the following uniformity condition holds  for every  $g \maps e \to d$,
$$e \forces \funapp ( t (f) , a) (g) =  \funapp ( t (f g) , a (g) ) \co \beta(x fg, a).$$

\end{enumerate}
\end{prop}

\begin{proof}
We work in the slice category $\tE/{\typecomp{\yon{c}}{\alpha}}$ where the representables are pairs of the form $(f, a)$ with $f \co d \to c$ and $d \forces a \co \alpha(x f)$.  Applying \cref{prop:forcing-eq-closedsliceforcing} to the condition in~(i) we obtain
$\tE/{\typecomp{\yon{c}}{\alpha}} \forces b\co \beta(x)$, 
which by \cref{cor:SandC} implies $\mathcal{T}_{\tE/{\typecomp{\yon{c}}{\alpha}}} \vdash b \co\beta(x)$.
Thus in $\mathcal{T}$ we have $\typecomp{\yon{c}}{\alpha} \vdash b \co \beta(x)$
for some $b \co \typecomp{\yon{c}}{\alpha} \to \UVptd$.
By \cref{thm:rules-pi}, therefore, $\yon{c} \vdash  \lambda b \co  \Pi_\alpha(\beta)(x)$, 
and hence  $c \forces  \lambda b \co  \Pi_\alpha(\beta)(x)$ by \cref{thm:forcing-alternative}(i). 
Then for all $f \maps d \to c$ and $d \forces a \oftype \alpha(x f )$ we indeed have 
\[
 \funapp ( (\lambda b) (f) , a ) = b(f,a) =  b_{(f,a)} \oftype \beta(x f, a)
 \] 
 by the computation rule of \cref{thm:rules-pi} and 
\cref{cor:forcing.dependent-elements}.
Thus (i), as required.

For (ii), take $c \forces t \co  (\Pi_{\alpha} \beta)(x)$ and $f \co d \to c$ so that 
$d \forces t(f) \co  (\Pi_{\alpha} \beta)(xf)$
by \cref{prop:forcing.monotonicity}.  We then have $\yon{d} \vdash t(f) \co  (\Pi_{\alpha} \beta)(xf)$, 
again by (i) of \cref{thm:forcing-alternative}.  And similarly for~$d \forces a \co \alpha(x f)$, we have 
$\yon{d} \vdash a \co \alpha(x f)$.  Applying the elimination rule from \cref{thm:rules-pi} then gives
$\yon{d} \vdash \funapp (  t (f), a) \co \beta(x f, a)$, whence $d \forces \funapp (  t (f), a) \co \beta(x f, a)$.
The uniformity with respect to $g \maps e \to d$ follows from the corresponding statement in the type theory, which is proved as usual from the Computation, Expansion, and Substitution rules.
\end{proof}

\begin{cor}[Forcing for function types]
\label{prop:forcing.function}
Let $\alpha \co X \to \UV$ and $\beta \co X \to \UV$. Then for every $x \co \yon{c} \to X$ we have the  following.
\begin{enumerate}
\item 
Suppose for every $f \maps d \to c$ and every $d \forces a \oftype \alpha(x f )$
we are given  $d \forces b_{(f,a)} \oftype \beta(x f)$, satisfying the following uniformity condition for all $g \maps e \to d$, 
$$e \forces b_{(f,a)} (g) = b_{(fg,a(g))} \co \beta(x fg).$$
Then there is a map $b \co \typecomp{\yon{c}}{\alpha} \to \UVptd$ for which 
\[
c \forces  \lambda(b) \oftype  (\alpha \to \beta)(x) \mathrlap{,}
\]
and for all $f \maps d \to c$ and $d \forces a \oftype \alpha(x f )$ we have
\[
{d \forces   \funapp ( (\lambda b) (f) , a ) =  b_{(f,a)} \oftype \beta(x f)} \mathrlap{.}
\]

\item  Suppose that $c \forces t \co  (\alpha \to \beta)(x)$.
Then for every $f \co d \to c$, 
and every $d \forces a \co \alpha(x f)$, we have
\[
d \forces \funapp (  t (f), a) \co \beta(x f)\,,
\] 
and moreover  the following uniformity condition holds for every  $g \maps e \to d$,
\[
\funapp ( t (f) , a) (g) =  \funapp ( t (fg) , a( g))\co \beta(x fg).
\]
\end{enumerate}
\end{cor}

\begin{proof} This is the special case of \cref{prop:forcing.dprd} since function types are defined as dependent product types (see \cref{nom:fun-and-prod-types}).
\end{proof}

The foregoing forcing statements included detailed information about which specific terms are forced, on the basis of those assumed as given.  While this information is sometimes useful, it can also be convenient to use the following simplified \emph{forcing conditions} which elide this information (to the extent possible).  Except where otherwise indicated, the individual forcing statements of the form $c \forces  t \co  \alpha$ then assert that \emph{there merely exists} a term $t$ such that $c \forces  t \co  \alpha$.

\begin{cor}[Forcing conditions]\label{cor:Forcingiff}
Let the types $\alpha$, $\beta$, etc., be as in the corresponding propositions above. 
Then for every $x \co \yon{c} \to X$ we have the following equivalences.
\begin{enumerate}
\item  $c \forces a \co \emptytp$ if and only if $a\neq a$.
\item $c \forces a \co \unit$ if and only if $a= *$.
\item $c \forces  t \co  (\alpha + \beta)(x)$ if and only if $c \forces a\co\alpha(x)$ or $c \forces b\co\beta(x)$,
\item $c \forces t \co  (\alpha \times\beta)(x)$ if and only if $c \forces a \colon \alpha(x)$ and $c \forces b \colon \beta(x)$.
\item $c \forces t \co  (\Sigma_\alpha \beta)(x)$ if and only if $c \forces a \colon \alpha(x)$ and $c \forces b \colon \beta(a)$.
\item $c \forces t \co  ({\alpha} \to {\beta})(x)$ if and only if $d \forces b_{(f,a)} \oftype \beta(x f)$ for all $f \maps d \to c$ and $d \forces a \oftype \alpha(x f )$, and $e \forces b_{(f,a)} (g) = b_{(fg,ag)} \co \beta(x fg)$ for all $g \maps e \to d$.
\item $c \forces t \co  (\Pi_{\alpha} \beta)(x)$ if and only if $d \forces b_{(f,a)} \oftype \beta(x f, a)$
for all $f \maps d \to c$ and $d \forces a \oftype \alpha(x f )$, and $e \forces b_{(f,a)} (g) = b_{(fg,ag)} \co \beta(x fg, ag)$
for all $g \maps e \to d$. \qed
\end{enumerate}
\end{cor}

\subsection*{Kripke-Joyal forcing for propositions} 

We conclude this section by specialising our Kripke-Joyal forcing for type theory to propositions in $\tE$ and relating it to
the standard Kripke-Joyal forcing for first-order logic. For this, we make use of the relation between the small map classifier $\UV$ and the subobject classifier $\iprop$ described in \cref{sec:small-vs-subobject}. 
Recall that the standard Kripke-Joyal semantics is given by the specification, for 
$\sigma \colon X \to
\Omega$ and $x \colon \yon{c} \to X$, that $c\forces\sigma(x)$ if there is a (necessarily
unique) dotted map fitting in the diagram
\begin{equation}\label{diag:standardKJforcing}
\begin{tikzcd}
& \propcomp{x \oftype X}{\sigma(x)} \ar[r] \ar[d, tail] 
&1\ar[d, tail, "\ttrue"] \\
\yon{c} \ar[r, swap, "x"] \ar[ru, bend left = 20, dotted] &  X  \ar[r, swap, "\sigma"] 
& \iprop \, .
\end{tikzcd}
\end{equation}
The following shows that our Kripke-Joyal semantics for type theory coincides with the standard one when restricted to propositions.

\begin{thm}
\label{prop:forcing.props-in-ctx} \label{prop:forcing.props}
Let \(\sigma \maps X \to \iprop\) and $x \maps \yon{c} \to X$. Then the following 
conditions are equivalent.
\begin{enumerate}
\item $c \forces \sigma(x)$. 
\item There exists a (necessarily unique) $s \colon \yon{c} \to \UVptd$ such that $c \forces s \colon 
\pat{\sigma(x)}$.
\end{enumerate}
\end{thm}

 \begin{proof} Immediate from the definitions, observing that we have the two-pullback diagram
\[ 
\begin{gathered}
\begin{tikzcd}
& \typecomp{X}{\pat{\sigma}} \ar[r] \ar[d, tail]
&1\arrow[d, tail, "\ttrue"]  \arrow[r]& \UVptd \arrow[d, "\pi"]  \\
 \yon{c} \ar[r, swap, "x"] \ar[ru, bend left = 20, dotted] &  X  \ar[r, swap, "\sigma"] 
& \iprop \ar[r, "\pat{-}", swap]  & \UV \mathrlap{.}
\end{tikzcd} 
\end{gathered} \vspace{-2em}
\]
\end{proof}

\begin{nom}
In light of \cref{prop:forcing.props-in-ctx}, we shall often write \(c \forces  \sigma(x) \) to mean 
that there exists an $s$ such that \(c \forces s \colon \pat{\sigma(x)} \). 
\end{nom}

By  \cref{prop:forcing.props-in-ctx}, we obtain forcing conditions for all  types 
$\sigma \colon \Gamma \to \UV$ that factor through the monomorphism $\{-\}\co\iprop \to \UV$
of~\eqref{equ:pat-Omega-to-U}. 
For such types, the uniformity conditions required for function types
and $\Pi$-types become vacuous and indeed give back the standard Kripke-Joyal conditions.
We illustrate this in a few cases.

\begin{prop}[Forcing for propositional equality]
\label{prop:forcing.exteq}
Let $\alpha \co X \to \UV$  and suppose that $X \vdash a \co \alpha$ and  $X \vdash b \co \alpha$.
Then for any $x \maps \yon{c} \to X$, the following conditions are equivalent.
\begin{enumerate}
    \item \( c \forces a(x) = b(x) \).
    \item $a(x)$ and $b(x)$ are equal as maps $\yon{c} \to \UVptd$.
\end{enumerate}
\end{prop}

\begin{proof} For each $x \maps \yon{c} \to X$, the element $a\co \alpha$ determines a map $(x,a(x))$  as follows (and similarly for $b\co \alpha$).
\[
\begin{tikzcd}
& \typecomp{X}{\alpha} \ar[d] \ar[r] & \UVptd \ar[d] \\
\yon{c} \ar[r, swap, "{x}"] \ar[ru, dotted, bend left = 20,  "{(x,a(x))}"]  & X \ar[ru, "a"] \ar[r,swap, "{\alpha}"] & \UV\diadot{.}
 \end{tikzcd}
 \]
Now consider the following diagram, in which both squares are pullbacks. 
\[
\begin{tikzcd}
       && \typecomp{X}{\alpha} \ar[d, swap, tail, "{\delta_\alpha}"] \ar[r] &  1 \ar[d] \ar[r] & \UVptd \ar[d] \\
 \yon{c} \ar[rr, swap, "{(x,a(x), b(x))}"]  \ar[urr, bend left = 20, dotted] 
 	&& \typecomp{X}{\alpha \times \alpha} \ar[r,  swap, "{=_\alpha}"] & \iprop \ar[r] & \UV  \mathrlap{.}
 \end{tikzcd}
 \]
By definition, we have \( c \forces a(x) = b(x)\) \ifaif there is a dotted map as indicated.  But this holds  if and only if 
 $(x,a(x)) = (x,b(x))$, which in turn is equivalent to $a(x) = b(x)$ as maps $\yon{c} \to \UVptd$.  
\end{proof}

For disjunction, even though $\pat{\sigma \lor \tau}$  is not the same as $\pat{\sigma}+ \pat{\tau}$, we still
have the following equivalence due to special properties of presheaf categories.

\begin{prop}[Forcing for disjunction]
\label{prop:forcing.disjunction}
Let \(\sigma, \tau \oftype X \to \iprop \). For $x \co \yon{c} \to X$,
 the following conditions are equivalent.
\begin{enumerate}
    \item \label{item:kj-disj-i} \(c \forces  \sigma(x) \disj \tau(x)  \). 
    \item \label{item:kj-disj-iii} \( c \forces \sigma(x) \) or  \( c \forces \tau(x) \).
\end{enumerate}
\end{prop}

\begin{proof}
As stated in \cref{prop:prop-as-types}, the disjunction \(\sigma \disj \tau \colon X \to \iprop\)
classifies the subobject $\pat{\sigma} \cup \pat{\tau} \mono X$, which is the image of the canonical map $\pat{\sigma} + \pat{\tau} \to X$.  So if
 \(c \forces  \sigma(x) \disj \tau(x) \),
then there is a map 
\(v \maps \yon{c} \to \pat{\sigma} \cup \pat{\tau}\)
lifting $x$. 
Since the representable $\yon{c}$ is projective, there is then a further lift $u \colon \yon{c} \to \pat{\sigma} + \pat{\tau}$, and so $c \forces u \colon \pat{\sigma(x)} + \pat{\tau(x)}$. 
Applying \cref{prop:forcing.sum} we obtain either
 \(c \forces  a \colon \pat{\sigma(x)} \) or  \(c \forces  b \colon \pat{\tau(x)} \), whence either \( c \forces \sigma(x) \) or  \( c \forces \tau(x) \).
The converse is even more direct, using the factorisations 
\[
\begin{tikzcd}
 \pat{\sigma} \ar[rd] \ar[r] & \pat{\sigma} \cup \pat{\tau} \ar[d] & \ar[l] \ar[ld] \pat{\tau} \\
& X \mathrlap{.} &
 \end{tikzcd}\vspace{-2em}
 \]
 \end{proof}

A similar fact also holds for existential quantifiers.

\begin{prop}
\label{prop:forcing.exists}
Let $\alpha \co X  \to \UV$ and  $\sigma \co  \typecomp{X}{\alpha} \to \Omega$. For
$x \co \yon{c} \to X$,  the following conditions are equivalent.
\begin{enumerate}
\item \( c \forces   (\exists_\alpha \sigma ) (x) \).
\item There exists $a \co \yon{c} \to \UVptd$ such that \(c \forces a \oftype \alpha(x) \), 
and \(c \forces \sigma(x, a) \).
\end{enumerate}
\end{prop}

\begin{proof} First assume that $c \forces (\exists_\alpha \sigma)(x)$, which means that we
have a lift
\[
\begin{tikzcd}
 & \propcomp{x \oftype X}{ \exists_\alpha \sigma (x)} \ar[d, tail] \\
 \yon{c} \ar[ur, dotted, "u"]  \ar[r, swap, "x"] & X \mathrlap{.}
 \end{tikzcd}
 \]
Since $\yon{c}$ is projective, there is then a further lift as indicated in the diagram below, which also recalls the construction of $\propcomp{X}{\exists_\alpha \sigma}$.
\[
\begin{tikzcd} 
& \typecomp{X}{\alpha.\pat{\sigma}} \ar[r, tail, "{p_\sigma}"] \ar[d, twoheadrightarrow] 
	& \typecomp{X}{\alpha}  \ar[r, "{q_\alpha}"] \ar[d, "{p_\alpha}"] 
	& \UVptd \ar[d, "{\dspu}"] \\
\yon{c}\ar[ur, dotted, "v"] \ar[r, swap, "u"] & \propcomp{X}{\exists_\alpha \sigma} \ar[r, tail] 
	& X \ar[r, swap, "\alpha"] & \UV
 \end{tikzcd} 
\]
Let $a \defeq q_\alpha\circ p_\sigma \circ v$, so that \(c \forces a \oftype \alpha(x) \), 
and \(c \forces \sigma(x, a) \), by construction.
The converse is direct, using the above diagram.
\end{proof}

The general characterisation of forcing for subset types follows the same pattern. The proof is left to the reader.

\begin{prop}[Forcing for subset types] \label{prop:forcing.formula}
Let $\alpha \co X \to \UV$ and $\sigma \colon \typecomp{X}{\alpha} \to \Omega$.
For  \(x \maps  \yon{c} \to X\), the following conditions 
are equivalent.
\begin{enumerate}
\item There exists $s$ such that $c \forces s \oftype \cmpset{y \oftype \alpha(x)}{\sigma(x,y)} $
\item There exists $a$ such that $c \forces a \oftype \alpha(x)$ and $c \forces  \sigma (x, a)$. \qed
\end{enumerate}
\end{prop}

\begin{prop}[Forcing for propositional truncation] 
Let $\alpha \co X \to \UV$. For  \(x \maps  \yon{c} \to X\), the following conditions 
are equivalent.
\begin{enumerate}
\item \label{item:trunc} $c\forces \supp(\alpha)(x)$.
\item \label{item:exists-witness} There exists $a$ such that $c\forces a \co \alpha(x)$.
\end{enumerate}
\end{prop}       
        
\begin{proof}  Recall from \cref{thm:prop-trunc} that the propositional truncation $\supp(\alpha) \co X \to \Omega$ is the pullback along $\alpha$ of the image factorisation of $\pi \co \UVptd \to \UV$,
\[
\begin{tikzcd}
&[20pt] \typecomp{X}{\alpha} \ar[d,twoheadrightarrow]  \ar[r,"{q_\alpha}"] & \UVptd \ar[d,twoheadrightarrow] \\
& \typecomp{X}{\{\supp(\alpha)\}} \ar[d,tail]  \ar[r] & \im(\pi) \ar[d,tail]  \\
\yon{c} \ar[r,swap, "x"] \ar[ru, dotted,pos= 0.65] \ar[ruu, dotted, bend left = 20, pos= 0.6, "u"] 
	& X  \ar[r, swap, "\alpha"]  & \UV \mathrlap{.}
\end{tikzcd}
\] 
By definition, $c\forces \supp(\alpha)(x)$ if and only if there is a factorisation of $x \co \yon{c} \to X$ through the subobject $\typecomp{X}{\{\supp(\alpha)\}}\mono X$.
But since the representable $\yon{c}$ is projective, this is the case if and only if there is a further lift $u\co \yon{c} \to\typecomp{X}{\alpha}$ as indicated.  Composing to get~$a= q_\alpha\circ u$ gives $c\forces a \co \alpha(x)$. 
The converse is  immediate.
\end{proof}

\begin{rmk}\label{Kripke-Joyal-completeness}
Our Kripke-Joyal forcing semantics can be specialised to the system of dependent type theory without a type $\Omega$ of propositions, but with the type formers
\begin{equation}\label{eq:EDTT}
  \quad 1 \mathrlap{,}   \quad   \alpha \times \beta \mathrlap{,}   \quad  \alpha \to \beta \mathrlap{,}  \quad \Sigma_{x\co \alpha }\beta(x) \mathrlap{,}  
\quad \Pi_{x \co \alpha} \beta(x) \mathrlap{,}  \quad  a =_\alpha b \mathrlap{,}  
\end{equation}
by treating propositional equality $a =_\alpha b$ equivalently as an (extensional) equality type $\mathsf{Eq}_{\alpha}(a,b)$, as mentioned in \cref{rmk:Equality}.  This can then be seen as a proof-relevant version of the standard Kripke semantics for intuitionistic first-order logic~\cite{troelstra-van-dalen}. The standard deduction relation $\vdash_\mathsf{IFOL}$ for that system can be shown to be \emph{sound and complete} with respect to standard Kripke semantics, in the sense that for any closed formula $\sigma$, one has 
$ \vdash_\mathsf{IFOL} \sigma$ if and only if $K\forces \sigma$
for all Kripke frames~$K$, which are essentially interpretations of the signature of $\sigma$ into categories of the form~$\mathsf{Set}^K$ for arbitrary posets~$K$.

Similarly, one can infer from the results in \cite{Awodey-Rabe2011} that the system of type theory specified by \eqref{eq:EDTT},  with the customary inference rules, as stated e.g.\ in \emph{ibid.}, is deductively complete with respect to our Kripke-Joyal forcing semantics, in the following sense. Given a closed type $\alpha$ (over a basic signature of closed and dependent types and terms), one can construct a closed term $\vdash a \co \alpha$ if and only if, for every interpretation of the signature of $\alpha$ into a category of the form $\mathsf{Set}^P$ for an arbitrary \emph{poset} $P$, one has $\mathsf{Set}^P \forces a \co \alpha$
in the sense of \cref{defn:KJforcing-in-E}.  Briefly, extensional dependent type theory is complete with respect to Kripke-Joyal forcing semantics, something that is not true for the system of higher-order logic, with a type of propositions.   We leave to the reader the pleasant task of simplifying the forcing conditions given in this section for the special case where the index category~$\cC$ is a poset.
\end{rmk}


\section{Cofibrations}
\label{sec:cof}

\subsection*{Basic definitions}

This section starts the second part of the paper, in which we  use our forcing semantics to describe models of homotopy type theory. In order
to do so, we assume certain auxiliary structures on our fixed presheaf category, as in~\cite{CoquandT:cubttc}, which give rise to (algebraic) weak factorisation systems~\cite{awodey-warren:homotopy-idtype,van-den-berg-garner,gambino-garner:idtypewfs}. As we shall see, forcing allows us to relate precisely  category-theoretic concepts, as defined in~\cite{CoquandT:cubttc,gambino-larrea,gambino2017frobenius}, with their type-theoretic counterparts, as defined in~\cite{orton-pitts}. Below, we shall take the categorical point of view as our starting point. In \cref{defn:cofibrations}, we write $\tE^{\to}_{\mathrm{cart}}$ for the category of arrows and pullback
squares in $\tE$.

\begin{defn} \label{defn:cofibrations}
A {\em class of cofibrations in $\tE = \Psh(\cC)$} is a class of maps $\awfsCof$  that satisfies the following conditions:
\begin{enumerate}[label=(C\arabic*)]
\item\label{item:ax-cof-mono} The elements of $\awfsCof$ are monomorphisms.
\item\label{item:ax-cof-false}  The unique map $0 \to 1$ is in $\awfsCof$.
\item\label{item:ax-cof-true}  The identity $\id_1 \co 1 \to 1$ is in $\awfsCof$.
\item\label{item:ax-cof-pullback} Cofibrations are stable under pullback along all maps.
\item\label{item:ax-cof-composition} Cofibrations are closed under composition.
\item\label{item:ax-cof-colimit} The full subcategory of $\tE^{\to}_{\mathrm{cart}}$ spanned by $\awfsCof$ has a terminal object.
\end{enumerate} 
\end{defn}

Let us now assume that $\awfsCof$ is a class of cofibrations in $\tE$. We say that an object $X \in \tE$ is {\em cofibrant} if the unique map $!_X \co 0 \to X$ is a cofibration.

\begin{rmk}\label{remark:cofibrations}
We note without proof a few elementary consequences of these axioms.
\begin{enumerate}
\item  In light of \ref{item:ax-cof-pullback}, axioms \ref{item:ax-cof-false} and \ref{item:ax-cof-true} imply that, for every $X \in \tE$,
the unique map $!_X \co 0 \to X$
and the identity $\id_X \co X \to X$ are cofibrations . Thus, every object is cofibrant. See~\cite{gambino-henry,henry-kan-constructive} for a context where this is not the case.
\item The terminal cofibration posited in \ref{item:ax-cof-colimit} will be denoted $\tcof \co 1 \mono \Phi$, note that its domain can be shown to be $1$.  In the pullback square
\begin{equation}\label{defn:Cof}
\begin{tikzcd}
 1 \ar[r] \ar[d, swap, tail, "\tcof"] & 1 \ar[d,"\ttrue", tail] \\
 \Cof \ar[r,swap] & \iprop \mathrlap{,}
\end{tikzcd}
\end{equation}
the lower horizontal characteristic map can be shown to be monic.
The map $\tcof \co 1 \mono \Phi$ classifies cofibrations, in the sense that a monomorphism $S \mono X$ is in $\awfsCof$ if and only if its characteristic map \(\sigma \maps X \to \iprop\)  factors through \( \Cof \mono \iprop\), and thus if and only if it is a pullback of $\tcof$ along a (necessarily unique) classifying map.

\item The full subcategory of $\tE^{\to}_{\mathrm{cart}}$ spanned by $\awfsCof$ has all colimits, and these are preserved by the inclusion into $\tE^{\to}$.  In particular, cofibrations are closed under coproducts $A + B \mono X + Y$, descent along epimorphisms, and joins $S \vee T \mono X$.

\item The cofibrations with representable codomain $S \mono\yon c$ are called \strong{generating cofibrations}. Note that these form a small set. Every cofibration is a colimit in $\tE^{\to}_{\mathrm{cart}}$ of generating cofibrations in a canonical way (determined by the category of elements of its codomain).  As shown in~\cite[Theorem~9.1]{gambino2017frobenius} using Garner's small object argument~\cite{garner:small-object-argument}, the subcategory of generating cofibrations determines an algebraic weak factorisation system~$(\awfsCof, \awfsTrivFib)$, with right class the \strong{uniform trivial fibrations} to be considered in \cref{sec:tfib}.

\end{enumerate}
\end{rmk}

Since $\Cof$ is a subobject of $\Omega$, it admits a characteristic map $\cof \maps \iprop \to \iprop$,
\begin{equation*} 
\begin{tikzcd} 
\Cof \ar[d, tail] \ar[r] & 1 \ar[d,tail, "\ttrue"] \\
\iprop \ar[r, swap, "\cof"] & \iprop  \mathrlap{.}
\end{tikzcd}
\end{equation*}
In the internal type theory, we thus have 
\begin{equation}
    \Cof =  \cmpset{\varphi \co \iprop}{\cof{\varphi}} \mathrlap{.} 
\end{equation}
We will call a proposition~$\varphi \co X \to \Omega$ \strong{cofibrant} if it factors through $\Cof \mono \iprop$, for which we also write $\varphi \co X \to \Phi$.

Let $\varphi \co X \to \Cof$ be a cofibrant proposition. The \strong{comprehension} of $\varphi$ with respect to $\Cof$ is defined as the pullback of $\tcof$ along $\varphi$.   Thus the cofibrant comprehension agrees with the usual (subobject) comprehension,
 \[ 
\begin{tikzcd}
 \pat{X\,|\,\varphi} \ar[d, tail]  \ar[r] & 1 \ar[r] \ar[d,"\tcof", tail] 
 & 1 \ar[d,"\ttrue", tail] \\
X  \ar[r, swap, "\varphi"]  & \Cof \ar[r,swap,tail]
& \iprop \mathrlap{.}
\end{tikzcd}
\]   
Thus in particular, every generating cofibration is of the form $\pat{\yon c\,|\,\varphi} \mono \yon c$ for a unique $\varphi \co \yon c \to \Cof$.

The next series of propositions shows some of the consequences of our axioms for cofibrations in the internal type theory of $\tE$, thus relating them
to those in~\cite[Figure~1]{orton-pitts}. In particular, \cref{prop:closure-properties-of-cof-under-composition} below 
is~\cite[Lemma~5.4~(i)]{orton-pitts}, but we give a forcing proof.

\begin{lem} Let $\varphi \co X \to \Omega$ be a proposition. For every $x \co \yon c \to X$, the following conditions are equivalent.
\begin{enumerate}
\item $c \forces  \cof \varphi(x)$.
\item  $\varphi(x)  \maps \yon c \to \Omega$ factors through  $\Cof \mono \Omega$.
\end{enumerate}
\end{lem}

\begin{proof} Consider the diagram
\begin{equation*}
\label{diag:forcing.cof}
\begin{tikzcd}[scale=1.5]
 & \pat{X\,|\,\cof \varphi}  \ar[d, tail]  \ar[r] & \Cof \ar[r] \ar[d, tail] 
& 1 \ar[d,"\ttrue", tail] \\
\yon c \ar[ur, dotted,bend left = 20] \ar[r, "x",swap] & X   \ar[r, swap, "\varphi"]  & \iprop  \ar[r,swap,"\cof"]
& \iprop \mathrlap{.}
\end{tikzcd}\vspace{-2em}
\end{equation*}
\end{proof}

\begin{lemma}
Let $\varphi \co X \to \Omega$. Then  \(\tE \forces \varphi \imp \cof \varphi \).
\end{lemma}

\begin{proof} We show that \(c \forces \varphi(x) \imp \cof \varphi(x)\) for every $x \maps \yon c \to X$.  So
let $f \maps d \to c$ and assume \(d \forces \varphi(x  f ) \). Then there is a dotted map in 
\[
\begin{tikzcd}
  & & \pat{\varphi} \ar[r] \ar[d, tail] \pbmark & 1 \ar[d, "\ttrue"] \\
\yon d \ar[r, swap, "f"] \ar[urr, dotted, bend left = 20] & \yon c \ar[r, swap, "x"] & X \ar[r, swap, "\varphi"] & \iprop \mathrlap{.}
\end{tikzcd}
\] 
We then obtain the diagram
\begin{equation*}
\begin{tikzcd}
 & & \pat{\varphi} \ar[r] \ar[d, tail] \pbmark & 1 \ar[d,tail, "\tcof"] & \\ 
 & & { \pat{\cof \varphi}} \ar[d, tail] \ar[r] \pbmark & \Cof \ar[r] \ar[d, tail] \pbmark & 1 \ar[d,"\ttrue"] \\
\yon d \ar[r, "f", swap]   \ar[rru, dotted, bend left = 20] \ar[rruu, dotted, bend left = 20] & \yon c  \ar[r, swap,"x"] & X  \ar[r, swap,  "\varphi"]  & \iprop \ar[r,swap,"\cof"] & \iprop \mathrlap{,}
\end{tikzcd}
\end{equation*}
which shows $d \forces \cof \varphi (xf )$, as required.
\end{proof}

\begin{cor} Let     \( \varphi \co X  \to \iprop \).
\label{cor:basic-properties-of-cofibrations}
\label{lem:properties-of-cofibrations}
\begin{enumerate}
    \item 
    The map \(\pat{\varphi} \mono \pat{\cof \varphi} \)   is a cofibration. 
    \item 
    $\cof \varphi =_\iprop \ttrue$ if and only if $\pat{\varphi}  \mono X$ is a cofibration 
    if and only if  \( \pat{\cof \varphi}  \mono X \) is an isomorphism. \qed 
\end{enumerate}
\end{cor}

\begin{prop}[Forcing characterisation of cofibrations]
\label{prop:forcing.cofibrations}
Let \( m \maps S \mono X \) be a monomorphism. Then the following conditions are equivalent.
\begin{enumerate}
\item The map $m$ is a cofibration. 
\item \(\tE \forces (\forall x \co X) \cof\!\big( (\exists s \co S)( m(s) = x) \big) \). 
\end{enumerate}
\end{prop}

\begin{proof} We first show that, for $\varphi \maps X \to \iprop$, $\tE \forces  (\forall x \co X) \cof \varphi(x)$ if and only if $\varphi$ is
cofibrant. Indeed, $\tE \forces  (\forall x \co X) \cof \varphi(x)$  if and only if $d \forces \cof \varphi(x  f)$ for all $x \co \yon c \to X$ and $f \co d \to c$.
This holds if and only if there is a dotted map in
\[
\begin{tikzcd}
    						& 				& {\pat{\cof \varphi}} \ar[d, tail] \\
\yon d \ar[r, swap, "f"]  \ar[urr, dotted, bend left = 20] &  \yon c \ar[r, swap, "x"] & X    \mathrlap{.}
 \end{tikzcd} 
\]
By the Yoneda lemma, this means there is a factorisation of $\varphi \co X \to\Omega$ through $\Phi$,
\[
\begin{tikzcd}
 & &  {\pat{\cof \varphi}} \ar[d, tail] \ar[r] & \Cof \ar[r] \ar[d, tail] & 1 \ar[d,"\ttrue", tail] \\
\yon d  \ar[r, "f", swap]  \ar[rru, dotted, bend left = 20] & \yon c  \ar[r, swap,"x"] & X \ar[ru, dotted, swap] \ar[r, swap, "\varphi"]  & \iprop \ar[r,swap,"\cof"] & \iprop  \mathrlap{.}
\end{tikzcd}
\]
To prove the required claim, it suffices to define $\varphi \co X \to \Omega$ to be the characteristic map of~\( m \maps S \mono X \), {i.e.}~$\varphi(x) \defeq (\exists s \co S) (m(s) = x)$.
\end{proof}

\subsection*{Dominance}

We show that Axiom \ref{item:ax-cof-composition} corresponds to the basic law of a \strong{dominance}, introduced in \cite{Rosolini:thesis}.  

\begin{lemma}
\label{lemma:forcing-dependent-cofibrant-prop}
Suppose $\varphi  \co X \to  \iprop$ and $\psi \co X \to \iprop$, and consider the pullback squares
\begin{equation*}
\begin{tikzcd}
\pat{\varphi \wedge \psi} 
\arrow[r,tail] \arrow[d,tail,swap] & \arrow[d,tail] \pat{\psi} \ar[r] & 1 \ar[d, "\ttrue"]  \\
\pat{\varphi}  \arrow[r,swap,tail]  & X \ar[r, swap, "\psi"] & \iprop  \mathrlap{.}
\end{tikzcd}
\end{equation*}
The following conditions are equivalent. 
\begin{enumerate} 
\item $\pat{\varphi \land \psi} \rightarrowtail \pat{\varphi}$ is a cofibration.
\item \( \tE \forces \varphi \imp \cof \psi \).
\end{enumerate}
\end{lemma}

\begin{proof} Assume first that $\pat{\varphi \land \psi} \rightarrowtail \pat{\varphi}$ is a cofibration.
Let $c \in \cC$ and $x \co \yon c \to X$. We want to show \( c \forces \varphi(x) \imp \cof \psi(x)  \),
so let $f \co d \to c$ and assume that $d \forces \varphi(xf)$, \ie that we have a factorisation
\[
\begin{tikzcd} 
   & & \pat{\varphi} \ar[d, tail] \\
 \yon d \ar[r, swap, "f"] \ar[rru, dotted, bend left =20] &   \yon c \ar[r, swap, "x"] & X \mathrlap{,}
\end{tikzcd}
\]
and we claim that $d \forces \cof \psi(xf)$. That holds if and only if $\psi(x f) \co \yon d \to \Omega$ factors through $\Cof$ or, equivalently, just if  $\pat{ \psi(xf) } \rightarrowtail \yon d$ is a cofibration. But $\pat{ \psi(xf) } \rightarrowtail \yon d$ is also the pullback of $\pat{\varphi \land \psi} \rightarrowtail \pat{\varphi}$ along the dotted arrow above. The claim follows by the closure of cofibrations under pullback, Axiom~\ref{item:ax-cof-pullback}.

Conversely, as we just saw, $\tE\forces\varphi \imp \cof \psi$ if and only if the pullback of $\pat{\varphi \land \psi} \rightarrowtail \pat{\varphi}$
along any map $\yon c \to \pat{ \varphi }$ is a cofibration. But this implies that $\pat{\varphi \land \psi} \rightarrowtail \pat{\varphi}$ itself is a cofibration. 
\end{proof}

\begin{prop} Assuming only axioms (C1--4) and (C6), axiom (C5) can be expressed equivalently as follows. 
\label{prop:closure-properties-of-cof-under-composition}
\begin{enumerate} 
\item \label{item:cof-composition}  Cofibrations are closed under composition.
\item \label{item:dominance} $\tE \forces ( \forall \varphi, \psi \co \iprop)  \cof\varphi \imp
  \big((\varphi\imp\cof\psi) \imp \cof(\varphi \conj \psi)\big)$.
  \end{enumerate}
\end{prop}

\begin{proof} 
Suppose that \cref{item:cof-composition} holds.  We need to show that  for every $c \in \cC$, we have $c \forces ( \forall \varphi, \psi \co \iprop)  \cof\varphi \imp 
	\big((\varphi\imp\cof\psi) \imp \cof(\varphi \conj \psi)\big)$\,.
Thus take any $x : d\to c$ and $\varphi, \psi : \yon{d} \to \Omega$ and any $y\co e\to d$ such that \(e \forces \cof\varphi(xy) \), and let furthermore $z\co f\to e$ with \(f \forces \varphi(xyz) \imp \cof(\psi(xyz)) \).
Then $\pat{\varphi(xy)}\rightarrowtail \yon e$ is a cofibration and therefore so is $\pat{\varphi(xyz)}\rightarrowtail \yon f$ and, by \cref{lemma:forcing-dependent-cofibrant-prop},
$\pat{\varphi(xyz) \land \psi(xyz)} \rightarrowtail \pat{\varphi(xyz)}$ is a cofibration as well.
By the assumption, the composite $\pat{\varphi(xyz) \land \psi(xyz)} \rightarrowtail \yon f$ is then a cofibration, and hence 
\(f \forces \cof(\varphi(xyz) \land \psi(xyz))\), as required.

To see that \ref{item:dominance} implies \ref{item:cof-composition}, note first that arbitrary cofibrations compose just if the same is true of cofibrations of the form $\pat{\varphi} \rightarrowtail \yon c$ and $\pat{\varphi \land \psi} \rightarrowtail \pat{\varphi}$, where $\varphi, \psi \co \yon c \to \iprop$. (Indeed, given arbitrary cofibrations $\pat{\varphi} \rightarrowtail  X$ and $\pat{\vartheta} \mono \pat{\varphi}$, the composite is a cofibration just if its pullback along every $\yon c \to X$ is one; then take $\pat{\psi} \rightarrowtail \yon c$ to be the composite $\pat{\vartheta} \rightarrowtail \pat{\varphi} \rightarrowtail \yon c$.)  
So let $\pat{\varphi} \rightarrowtail \yon c$ and $\pat{\varphi \land \psi} \rightarrowtail \pat{\varphi}$ be cofibrations, and assume \cref{item:dominance}, so that for all $d \to c $ and all $\varphi, \psi: d \to  \Omega$, we have
\[
d \forces  \cof\varphi \imp
  ((\varphi \imp\cof\psi) \imp \cof(\varphi\conj \psi)) \mathrlap{.}
\]
Thus, in particular, if $c \forces \cof\varphi$ and $c \forces \varphi \imp\cof \psi$ then 
$c \forces \cof(\varphi \conj \psi)$. But 
the first premise holds since $\pat{\varphi} \rightarrowtail \yon c$ is assumed to be a cofibration, and the second premise holds by \cref{lemma:forcing-dependent-cofibrant-prop}, since $\pat{\varphi \land \psi} \rightarrowtail \pat{\varphi}$ is also assumed to be a cofibration. Thus we have $c \forces \cof(\varphi \conj \psi)$, and so $\pat{\varphi \land \psi} \rightarrowtail \yon c$ is a cofibration, as required.
\end{proof}

\begin{rmk}\label{rmk:cof-composition}
Let $\varphi  \co X \to  \Cof$ and $\psi \co \pat{\varphi} \to \Cof$, and consider the pullback squares
\begin{equation*}
\begin{tikzcd}
\pat{\varphi}  \arrow[r] \arrow[d,tail]  & 1  \ar[d,"\tcof"] \\
X \ar[r, swap, "\varphi"] & \Cof\mathrlap{,} 
\end{tikzcd} \qquad
\begin{tikzcd} 
\pat{\varphi \,|\, \psi} \arrow[r] \arrow[d,tail] & 1 \ar[d, "\tcof"]  \\
 \pat{\varphi}   \ar[r, swap, "\psi"] & \Cof \mathrlap{.} 
\end{tikzcd}
\end{equation*}
We write $\Sigma_{\varphi} \psi \co X \to \Phi$ for the classifer of  the composite cofibration $\pat{\varphi \,|\, \psi} \mono \pat{\varphi} \mono X$. Thus, $\pat{\varphi \,|\, \psi} = \pat{\Sigma_{\varphi} \psi}$ as subobjects of $X$.
\end{rmk}


\section{Partial elements}
\label{subsec:partiality}

Uniform trivial fibrations will be defined in \cref{sec:tfib} as maps equipped with a lifting structure against cofibrations. In the type theory of $\tE$, this is expressed in terms of the notion of a cofibrant partial element of a type~\cite[Section~5.1]{orton-pitts}. In preparation for the material on uniform fibrations, we discuss cofibrant partial elements, leading up to \cref{prop:forcing.partial}, which provides the equivalence between the categorical and type-theoretical
definitions. As before, we start with a diagrammatic definition. We consider a fixed class of cofibrations as in \cref{defn:cofibrations}.

\begin{defn}
\label{defn:sem-partial-element}
Let $p \co A \to X$ be a map in $\tE$.
\begin{enumerate}
\item A \strong{(cofibrant) partial element} of  $p \co A \to X$ is a pair $(m, u)$, where $m \co S \mono X$ is a cofibration
and $u \co S \to A$ makes the following diagram commute:
\begin{equation}
\label{diag:exto-sem}
\begin{tikzcd}
S  \arrow[r,  "{u}"]	\arrow[d,tail, swap, "m"]  & A \arrow[d, "p"] \\
X \ar[r, equal]	& X	 \mathrlap{.} 
\end{tikzcd}
\end{equation}
We call $S$ the \strong{extent} of the partial element $(m, u)$. 
\item A partial element $(m, u)$ of $A \to X$ is \strong{extensible} if there exists $a \co X \to A$  such that the following diagram commutes:
\[
\begin{tikzcd}
S \arrow[r,  "{u}"]	\arrow[d,tail, swap, "m"]  & A \arrow[d] \\
X \ar[r, equal]	\ar[ur, swap,dotted, "a"] & X \mathrlap{.} 
\end{tikzcd}
\]
\end{enumerate}
\end{defn}

Regarding $A$ as an object in $\tE/_{\!X}$, we wish to define the object $\T_X(A) \in \tE/_{\!X}$ of all such partial elements of $A$, or as we shall say, the \emph{classifier} for partial elements.  
Consider the case $X=1$.  A partial element can then be described equivalently as a pair $(\varphi, u)$ where $\varphi\co1\to \Phi$,  and $u \co \pat{\varphi}\to A$.  We can therefore define the object $\T(A)$ as follows.  Let $\T \co \tE \to \tE$ be the polynomial endofunctor associated to the map $\tcof \maps 1 \to \Cof$, namely the composite
\begin{equation}
\label{defn:typeof-partial-elements}
\begin{tikzcd} 
 \tE \ar[r, "\radj{\tcof}"] &[10pt] \tE/_{\Cof} \ar[r, "\ladj{\Cof}"]  &[10pt] \tE \mathrlap{.} 
\end{tikzcd} 
\end{equation}

The global elements $t \co 1\to \T(A)$ are pairs  \(t=\pair{\varphi}{v}\) with \(\varphi \oftype 1\to \Cof\) and~\(v \co 1\to A^{\pat{\varphi}} \). This data determines a unique span
\begin{equation}
\label{diag:partialelement1}
\begin{tikzcd}
\pat{\varphi}  \arrow[r,  "{u}"]	\arrow[d,tail, swap]  & A \arrow[d] \\
1 \ar[r, equal]	& 1	 \mathrlap{,}
\end{tikzcd}
\end{equation}
with $\lambda u = v$, which is exactly a partial element of $A$ in the sense of \cref{defn:sem-partial-element}.
For example, since \(\cof(\tcof) = \ttrue\) by axiom~\ref{item:ax-cof-true}, for any $a \co 1\to A$ there is the totally defined  partial element~\(\pair{\tcof}{a }\co 1\to \T(A)\). 
Similarly, since \(\cof(\fcof)= \ttrue \) by axiom \ref{item:ax-cof-false}, we have the totally undefined partial element~\(\pair{\fcof}{0_A} \co 1\to \T(A)\).

Finally, if $p\co A\to X$, then let $\T_X(A)\to X$ be the result of applying the corresponding endofunctor $\T_X \co \tE/_{\!X} \to \tE/_{\!X}$, where in place of $\tcof \maps 1 \to \Cof$ we use the cofibration classifier for $\tE/_{\!X}$, namely $X^*\tcof \maps X^*1 \to X^*\Cof$, as in~\cite[Section~9]{gambino2017frobenius}.  We then recover the description in~\cite[Definition~5.1]{orton-pitts} using the internal type theory of $\tE$, as follows. If $p \co A\to X$  is classified by $\alpha \co X\to \UV$, then $\T_X(A) \to X$ is classified by the type
\begin{equation}\label{eq:T-type}
\textstyle
\T(\alpha) \defeq\ \dsm{\varphi \oftype \Cof}  \alpha^{\pat{\varphi}} \mathrlap{,}
\end{equation}
in the sense that $\T_X(A) \cong \typecomp{X}{\T(\alpha)}$ over $X$.  Hereafter, we will  use the type-theoretic description \eqref{eq:T-type} interchangeably with the relative point of view, keeping in mind that the function type $\alpha^{\pat{\varphi}}$ refers to the relative exponential over $X$. Sections $t \co X\to \typecomp{X}{\T(\alpha)}$ over $X$ then correspond in exactly the same way to diagrams
\begin{equation}
\label{diag:partialelement2}
\begin{tikzcd}
\pat{\varphi}  \arrow[r,  "{u}"]	\arrow[d,tail, swap]  & \typecomp{X}{\alpha} \arrow[d, "{\dsp{\alpha}}"] \\
X \ar[r, equal]	& X	 \mathrlap{,}
\end{tikzcd}
\end{equation}
i.e.\ partial elements of $\dsp{\alpha} \co \typecomp{X}{\alpha}\to X$, with $t = (\varphi , \lambda u)$ for $\varphi \co X\to\Cof$ and $u \co \pat{\varphi} \to X.\alpha$ over $X$.

The functor $\T \co \tE \to \tE$ defined in~\eqref{defn:typeof-partial-elements}, as well as its relative versions $\T_X \co \tE/_{\!X} \to \tE/_{\!X}$, has a canonical pointing $\eta \co \id_{\tE} \to \T$, given by the composite
\[
\begin{tikzcd} 
1 \ar[r, equal] 
 & \Sigma_{{\Cof}}  \circ \Sigma_{\tcof}  \ar[r] 
 & \Sigma_{{\Cof}}  \circ \Pi_{\tcof} \circ {\tcof}^* \circ \Sigma_{\tcof}  \ar[r, "\cong"]
 &  \Sigma_{{\Cof}}\circ  \Pi_{\tcof} \ar[r, equal]
 & \T \mathrlap{,}
 \end{tikzcd}
 \]
where the second map uses the unit of the adjunction ${\tcof}^* \dashv \Pi_\tcof$, and the third map uses that the unit of the adjunction $\Sigma_\tcof \dashv {\tcof}^*$ is an isomorphism $\id_{\tE} \cong {\tcof}^* \circ \Sigma_\tcof$ since $\tcof$ is monic.
Observe that $\eta_A \co A \to \T(A)$ can be defined  in the type theory of $\tE$ as the function taking $a \co A$ to the totally defined partial element \(\pair{\tcof}{a } \co T(A)\).  \cref{prop:etaclassifies} below gives the sense in which $\T(A)$ classifies cofibrant partial elements of $A$. For this, we 
require a lemma.

\begin{lem}\label{lem:tiscof} For every $A \in \tE$, the commutative diagram
\[
\begin{tikzcd}[scale=1.5]
A  \arrow[r, "\eta_A"] \arrow[d, swap, "!_A"] 
& \T(A) \arrow[d,"\fst"] \\
1 \arrow[r,swap,"\tcof"]
& \Cof
\end{tikzcd}
\]
is a pullback. Thus $\eta_A \co A \to \T(A)$ is a cofibration.
\end{lem}

\begin{proof} 
As already noted, since $\tcof$ is monic the unit of $\Sigma_\tcof \dashv {\tcof}^*$ is an isomorphism $\id_{\tE} \cong {\tcof}^* \circ \Sigma_\tcof$.  It follows by a standard argument that the counit of ${\tcof}^* \dashv \Pi_\tcof$ is also an isomorphism ${\tcof}^* \circ \Pi_\tcof \cong \id_{\tE}$.  Since $T(A) =  \Pi_\tcof (A)$, we therefore have $\tcof^*(TA) \cong A$.
\end{proof}

\begin{cor}
\label{cor:eta-is-cartesian}
The natural transformation $\eta \co \id_{\tE} \to \T$ is  cartesian.
\end{cor}

\begin{proof}
Let $f \maps A \to B$ be a morphism in $\tE$ and consider the naturality diagram
\[ 
\begin{tikzcd}[ampersand replacement=\&]
A \arrow[r,  "\eta_A"]  \arrow[d, swap, "f"] \&[10pt]  \T(A)  \arrow[d, "\T{f}"]  \\
B  \ar[d] \arrow[r, "\eta_B"]  \&  \T(B) \ar[d]   \\
 1 \ar[r, swap, "\tcof"] \& \Cof \mathrlap{,} 
\end{tikzcd}
\]
where $\tcof \co 1 \to \Cof$ coincides with $\eta_1 \co 1 \to \T(1)$. The bottom square and the outer rectangle are cartesian by the previous lemma, so the top square is also cartesian. 
\end{proof}

The next proposition gives the universal property of $\eta_A \co A \to \T(A)$.

\begin{prop}\label{prop:etaclassifies} For every $A \in \tE$, the map $\eta_A \maps A\mono \T(A) $ classifies cofibrant partial maps to $A$ in the following sense: for any $T \in \tE$ and cofibration $m \co S \mono T$ and map $u \co S \to A$ as indicated below,  the span $(m, u)$ 
can be uniquely completed to a pullback square as follows. 
\begin{equation}
\label{diag:cofibrant-partial-map-classifier}
\begin{tikzcd}
S \arrow[d, tail, swap, "m"] 
\arrow[r,"u"] &[10pt] A  \arrow[d, "\eta_A", tail] \\
T  \arrow[r, swap, dotted, ""] & \T(A) \mathrlap{.} 
\end{tikzcd}
\end{equation}
\end{prop}

\begin{proof} The cofibration $m \co S \mono T$ has a unique classifying map $\varphi \co T \to\Cof$. Using  \cref{lem:tiscof}, we can complete \eqref{diag:cofibrant-partial-map-classifier} to form the solid part of the following diagram, in which the outer (distorted) rectangle is a pullback.
\begin{equation}
\label{diag:proof-cofibrant-partial-map-classifier}
\begin{tikzcd}
S \arrow[d, tail, swap, "m"] 
\arrow[r,"u"] & A  \arrow[d,  tail] 
\arrow[r] 
& 1 \arrow[d, "\tcof"]\\
T  \arrow[r, swap, dotted] \arrow[rr, bend right=25, swap, "\varphi"] & \T(A) \arrow[r,"\fst"] & \Cof \mathrlap{.} 
\end{tikzcd}
\end{equation}

Recalling that~$\fst \co \T(A)\to \Phi$ is $\Pi_\tcof(A)$, we have the indicated dotted map as $u^\sharp \co T \to \T(A)$ over $\Cof$, the adjoint transpose of $u \co T \cong \pbk{\tcof}(T) \to A$ under the adjunction~$\tcof^* \dashv \radj{\tcof}$.  The left hand square then commutes, and is a pullback, because~$\tcof^*(u^\sharp) = u$.
 \end{proof}

Letting $T=1$ in the foregoing shows that $\T(A)$ classifies partial elements of $A$ in $\tE$, and doing the same for $p \co A\to X$ in the slice category $\tE/_{\!X}$ shows that $\T_X(A)$ classifies partial elements in the sense of \cref{defn:sem-partial-element}.
In light of \cref{prop:etaclassifies}, we call $\eta_A \oftype A \to \T(A)$ the \strong{cofibrant partial map classifier} of $A$  (cf.~\cite[Proposition~10]{awodey-cubical-git}), generalising the usual partial map classifier from topos theory~\cite[Proposition~A.2.4.7]{Johnstone:TT}, which occurs when $\cof = \ttrue\oo{!_\Omega}$ and so $\Cof = \Omega$.  

We now unfold explicitly the forcing conditions for $\T(A)$, which are determined by those for dependent sums and dependent products.

\begin{prop}[Forcing for partial elements] 
\label{prop:forcing.partial}
Let $\alpha \co X \to \UV$. Then for any $x \co \yon{c} \to X$ the following conditions are equivalent.
\begin{enumerate}
\item \label{rmk:forcingTiscomsquare} \label{diag:corparforce} There is a cofibration $m \co \pat{\varphi} \rightarrowtail \yon c$ and a 
map $u \co \pat{\varphi}  \to \UVptd$ such that the following diagram commutes:
\begin{equation*}
\begin{tikzcd}[column sep=large]
 \pat{\varphi} \arrow[d, tail, swap, "m"] \arrow[r,"{(xm,u)}"] & X.\alpha \arrow[d, "\dsp{\alpha}"] \\
 \yon{c} \arrow[r, swap,"x"]  &  X \mathrlap{.}
 \end{tikzcd}
\end{equation*}
\item \label{item:forcing-partial} \(c \forces \varphi \oftype \Cof\)  and there is a map $u \co \pat{\varphi}  \to \UVptd $ for which  
 \(c \forces \lambda u \oftype \pat{\varphi}\to \alpha(x)\).
\item \label{item:forcing-monad} There is a map $t \co \yon c \to \UVptd$ such that \( c \forces t \oftype \T( \alpha)(x) \).
\end{enumerate}
\end{prop}

\begin{proof}  For the equivalence between \cref{rmk:forcingTiscomsquare} and \cref{item:forcing-partial}, a map $u \co \pat{\varphi} \to \UVptd$ determines a unique map $(xm,u) \co \pat{\varphi}  \to \typecomp{X}{\alpha}$ making the diagram in \cref{rmk:forcingTiscomsquare} 
commute if and only if the judgement $\yon{c},\pat{\varphi}\vdash u \co \alpha(xm)$ holds.  But the latter is equivalent to $c\forces \lambda u \co \pat{\varphi}\to \alpha(x)$. 

For the equivalence between \cref{item:forcing-partial} and \cref{item:forcing-monad}, recall the definition of  $\T(\alpha)$ from 
from \eqref{eq:T-type}. So by \cref{prop:forcing.dsum}, 
\( c \forces t \oftype \T(\alpha)(x) \)
if and only if $t = (\varphi,v)$ with 
\(c \forces \varphi \oftype \Cof \) and 
\(c \forces v \oftype (\pat{\varphi}\to \alpha)(x) \). 
We might as well assume without loss of generality that $v = \lambda u$ for suitable $u \co \pat{\varphi}  \to \UVptd$, as required.
\end{proof}

\begin{rmk} The forcing conditions in \cref{prop:forcing.partial} may be unfolded even further using~\cref{prop:forcing.function}.
Indeed, if $t = (\varphi,v)$ then \(c \forces \varphi \oftype \Cof \) and \(c \forces v \oftype (\pat{\varphi}\to \alpha)(x) \) 
if and only if for all  \(f \maps d \to c \)  and \(d \forces p \oftype \pat{\varphi( f)} \) we have 
\(d \forces \funapp(v(f), p) \oftype \alpha(x f)\), and for all \(g \maps e \to d\), we have
\( e\forces \funapp(v(f), p)  (g) = \funapp(v({f g}), p(g))\co \alpha(xfg) \).  
In particular, for~$v = \lambda{u}$, the computation rule for function types reduces the above to
\[
d \forces p \oftype  \pat{\varphi (f)} \, , \qquad d \forces u(f)p \oftype \alpha(x  f)  
\]
for all $f\maps d \to c$ and, furthermore, the following uniformity condition holds
\[
e \forces (u(f)p) ( g) = u({f g})p(g) \co\alpha(xfg) 
\]
for all  \(g \maps e \to d\).	 
\end{rmk}

The next result also admits both diagrammatic and internal type-theoretic proofs (cf.~\cite[Proposition~12]{awodey-cubical-git}).  The proof via forcing given here serves mainly to illustrate the method.

\begin{thm} \label{thm:monad-structure-partial-elts}
Let $\alpha \co 1 \to U$. Then there is a map \(\mu_\alpha \maps \T^2(\alpha) \to \T(\alpha) \), natural in $\alpha$, making the pointed endofunctor $(\T, \eta)$ into a monad on $\tE$.  Moreover, the monad $(\T, \eta, \mu)$ extends to the slice categories $\tE/_{X}$ making it an $\tE$-indexed monad. 
\end{thm}

\begin{proof}
We will force an element $\mu_\alpha$ of the closed function type $\T^2(\alpha) \to \T(\alpha)$.
Thus we need to show:
\begin{enumerate}[(1)]
\item for all $c\in\cC$, there is given $\mu_\alpha c$ such that\ $c \forces \mu_\alpha c \oftype \T^2(\alpha) \to \T(\alpha)$, 
\item for all $f \maps d \to c$,
\begin{equation}\label{eq:force-monad-mult} 
d \forces {\mu_\alpha c}\circ f\, =\, \mu_\alpha d \oftype  \T^2(\alpha) \to \T(\alpha)\,.
\end{equation}
\end{enumerate}

By \cref{prop:forcing.function}, the forcing condition (1) for $ \mu_\alpha c$ of function type $ \T^2(\alpha) \to \T(\alpha)$ at $c$ requires that for every $f \maps d \to c$ and for all $a$ such that $d \forces a \oftype \T^2(\alpha)(f)$,
there is given $\mu_\alpha c_{(f,a)}$ such that\ $d \forces \mu_\alpha c_{(f,a)}\oftype \T(\alpha)(f)$, and moreover 
$\mu_\alpha c_{(f,a)}\circ g = \mu_\alpha c_{(fg,ag)}$ for every  \(g \maps e \to d\).  

Since the types are closed, the indicated substitution of $f$ is trivial, and so the  requirements (1) and (2) simplify to the following (with some relettering): 
\begin{itemize}
\item[(1')]  for all $f \maps d \to c$, and for all $a$ such that $d \forces a \oftype \T^2(\alpha)$,
there is given $\mu_\alpha c_{(f,a)}$ such that\ $d \forces \mu_\alpha c_{(f,a)}\oftype \T(\alpha)$, and
\item[(2')] for all $g\maps e \to d$,\ we have $\mu_\alpha c_{(f,a)}\circ g = \mu_\alpha c_{(fg,ag)}$, and
\item[(3)] for all $a$ such that\ $e \forces a \oftype \T^2(\alpha)$, we have
			$$e\forces  \mu_\alpha c_{(fg,a)}\, =\,  \mu_\alpha d_{(g,a)}  \oftype \T(\alpha).$$
 \end{itemize}
Note that this also gives the uniformity in $c$ required in (2). 

Thus let  $f \maps d \to c$, and suppose that $d \forces a \oftype \T^2\alpha$. 
By \cref{prop:forcing.partial}, for the type of partial elements $\T(\T\alpha)= \dsm{\varphi \oftype \Cof}  (\T\alpha)^{\pat{\varphi}}$, we have $a = (\varphi,\lambda u)$ with:
\begin{enumerate}
	\item[(a)] \(d \forces \varphi \oftype \Cof\) and \(d \forces \lambda u \oftype \pat{\varphi}\to \T\alpha\), the latter meaning that:
	
	\item[(a1)] for all $g\maps e \to d$ and all \(e \forces p \oftype  \pat{\varphi g}\), 
		we have \(e \forces u(g)p \oftype \T\alpha  \), and 
	
	\item[(a2)] for all  \(h \maps e' \to e\),\  we have $(u(g)p)h = u(g h) p(h)$.
 	\end{enumerate}
	
Unfolding the clause \(e \forces u(g)p \oftype \T\alpha  \) in (a1) yields $u(g)p = (\psi, \lambda v)$ with:
\begin{enumerate}
	\item[(b)] \(e \forces \psi \oftype \Cof\) and \(e \forces \lambda v  \oftype \pat{\psi} \to \alpha\), the latter meaning that:
		
	\item[(b1)] for all $h\maps e' \to e$ and all \(e' \forces q \oftype  \pat{\psi h}\),  
		 we have \(e' \forces v(h) q \oftype \alpha \), and 
		
	\item[(b2)] for all  \(i \maps e'' \to e'\),\  we have $(v(h)q)i = v(hi)q(i)$.
 	\end{enumerate}

To satisfy (1'), we need to construct an element  $\mu_\alpha c_{(f,a)} = (\vartheta, \lambda w) \oftype \T\alpha$  (based at $\yon{d}$) such that, among other things,  
\[
 d \forces \vartheta \oftype \Cof\quad \text{and}\quad d \forces \lambda w \oftype \pat{\vartheta}\to \alpha\,.
\]	

From (a) above, we have a span,
\begin{equation}
\label{diag:Tmonad1}
\begin{tikzcd}
 \pat{\varphi} \arrow[d, tail, swap] \arrow[r,"{u}"] & \T\alpha  \\
 \yon{d} \mathrlap{.} &
\end{tikzcd}
\end{equation}
To see this, take any $p \co \yon{e} \to  \pat{\varphi}$ and compose with $\pat{\varphi} \mono \yon{d}$ to get a map $g\maps e \to d$ with  $e \forces p \oftype  \pat{\varphi g}$.  By (a1) we then get $ u(g)p \co \yon{e} \to \T\alpha $.  Since the assignment $p \mapsto u(g)p$ is natural in $e$ by (a2), by Yoneda there is indeed a map $u \co \pat{\varphi} \to \T\alpha$, as claimed. 

Now let $u = \pair{\psi}{\lambda v}$, where $\psi \co  \pat{\varphi}  \to \Phi$, and there is another span 
\begin{equation}
\label{diag:Tmonad2}
\begin{tikzcd}
 \pat{\psi} \arrow[d, tail, swap] \arrow[r,"v"] & \alpha  \\
 \pat{\varphi} \mathrlap{,} &
\end{tikzcd}
\end{equation}
as can be deduced from (b) in the same way that $(\varphi, \lambda u)$ followed from (a).

Let $\vartheta \co \yon{d}\to \Phi$ classify the composite $\pat{\psi} \mono \pat{\varphi} \mono \yon{d}$, which is a cofibration since each factor is one.  We can then take $w = v$ to get the desired pair $(\vartheta, \lambda w)$ with $d\forces (\vartheta, \lambda w)\oftype \T(\alpha)$. Thus with $\mu_\alpha c_{(f,a)} = (\vartheta, \lambda w)$ we have (1').  The uniformity conditions (2') and (3) then follow from the uniformity conditions in (a) and (b).

Thus we have shown \eqref{eq:force-monad-mult}, and so there is a closed term  $\mu_\alpha\co \T^2\alpha \to \T\alpha$.  That $\mu_\alpha$ is natural in $\alpha$ can be seen to follow from the universal property of $(\T, \eta)$ given in \cref{prop:etaclassifies}.  The monad laws for $(\T, \eta,\mu)$ can be seen to follow from the unit and associativity laws for composition of (general maps and) cofibrations. Finally, the entire monad $(\T, \eta, \mu)$ is indexed over $\tE$, because all of the type-theoretic constructions involved are stable under pullback, as is forcing, by \cref{prop:forcing-eq-closedsliceforcing}.
\end{proof}


\section{Uniform trivial fibrations}
\label{sec:tfib}

We again assume a fixed class $\mathsf{Cof}$ of cofibrations in $\tE$, as in \cref{defn:cofibrations}.  The subcategory of $\tE^{\rightarrow}$ consisting of cofibrations and cartesian squares between them determines an algebraic weak factorisation system ~$(\awfsCof, \awfsTrivFib)$ on~$\tE$~\cite[Theorem~9.1]{gambino2017frobenius}.  The maps in the right class $\awfsTrivFib$ are the uniform trivial fibrations
which appear in the semantics of Homotopy Type Theory (see~\cite[Section~5.1]{CoquandT:cubttc}, \cite{awodey-cubical-git}, and~\cite{gambino-larrea}).  In this section, we show how these structured maps 
can be described equivalently in the internal type theory of~$\tE$. In particular, we see how the uniformity condition arises
naturally from Kripke-Joyal forcing.

\begin{defn}\label{defn:uniformtrivialfibration} Let $p \co A \to X$ be a map in $\tE$.
A \emph{uniform trivial fibration structure} on $p$ consists of 
a function $j$ that assigns a diagonal filler $j(m,u,v)$ to every commutative diagram of solid arrows of the form
\[
\begin{tikzcd}
 S \arrow[r, "u"] \arrow[d, swap,  "m", tail]   &[25pt]  A \arrow[d, "p"] \\ 
 T \arrow[r, swap, "v"]  \arrow[ru, dotted]   &  X\diadot{,}
\end{tikzcd}
\]
where  $m$  is a cofibration, subject to the following
\emph{uniformity} condition: for every $t \co T' \to T$, giving rise to the following diagram with a pullback square on the left,
\begin{equation}
\label{diag:uniformcofibration}
\begin{tikzcd}[scale=2]
\pbk{t}(S) \arrow[r, "s"] \arrow[d, tail,swap, "{\pbk{t}(m)}"]  &[25pt] S \arrow[r, "u"] \arrow[d, tail,  pos = 0.65, "m"] 
 &[25pt] A \arrow[d, "p"] \\ 
 T'  \arrow[r, swap, "t"] \arrow[rru,dotted] 
  &  T \arrow[r, swap, "v"]  \arrow[ru, dotted,swap] 
  & X \mathrlap{,}
\end{tikzcd}
\end{equation}
we have $j(\pbk{t}(m),  us, vt) = j(m, u, v) \circ t$.
\end{defn}
\smallskip

The next result recalls from~\cite[Theorem~9.1]{gambino2017frobenius} that a trivial fibration structure is determined by its values for the generating cofibrations, i.e.~those with representable codomains, since a general diagonal filling problem can be expressed as a colimit of filling problems for generating cofibrations. 

\begin{lemma} \label{thm:triv-cof-againts-generating} 
Let $p \co A \to X$ be a map in $\tE$. Then $p$ admits a uniform trivial fibration structure if and only if there is a 
function $j$ that assigns a diagonal filler $a = j(m, u, x)$ to every commutative diagram of solid arrows
\begin{equation}
\label{equ:lift-problem-gen-cof-triv-fib}
\begin{tikzcd}
S  \arrow[r, "u"] \arrow[d, swap, "m", tail]   &[25pt] A  \arrow[d, "p"] \\ 
 \yon{c}  \arrow[r, swap, "x"]  \arrow[ru, dotted, swap, "a"]   &  X \diadot{,}
\end{tikzcd}
\end{equation}
where  $m$ is a generating cofibration, subject to
the uniformity condition indicated in the following diagram:
\begin{equation*}
\begin{tikzcd}[scale=2]
\pbk{f}(S) \arrow[r, ""] \arrow[d, tail]  &[25pt] {S} \arrow[r, "{u}"] \arrow[d, tail] 
 &[25pt] A \arrow[d, "p"] \\ 
 \yon{d}  \arrow[r, swap, "f"] \arrow[rru,dotted] 
  &  \yon{c} \arrow[r, swap, "x"]  \arrow[ru, dotted,swap] 
  & X \mathrlap{.}
\end{tikzcd} \vspace{-1em}
\end{equation*} \qed
\end{lemma}

We now wish to express the diagonal filling property in the internal type theory. So let $p \co A \to X$ be a small map, $m \co S \mono \yon{c}$ be a cofibration and consider their classifying maps $\alpha \co X \to \UV$ and $\varphi \co \yon{c} \to \Cof$, respectively. The outer diagram in \eqref{equ:lift-problem-gen-cof-triv-fib} can then be expressed equivalently in the form,
\begin{equation}\label{diag:classified-filling-problem}
\begin{tikzcd}
\pat{\varphi}  \ar[rr, "{(xi,u)}"] \ar[d, swap, tail, "i"] && \typecomp{X}{\alpha} \ar[d, "p_{\alpha}"] \\
{\yon{c}} \ar[rr, swap, "x"] && X \diadot{,}
\end{tikzcd}
\end{equation} 
for a suitable $u \co \pat{\varphi} \to E$ with $\yon{c}.\pat{\varphi}\vdash u \co \alpha(xi)$.  Recall thereby that $i\co\pat{\varphi}\mono \yon{c}$ abbreviates the comprehension,
\begin{equation}\label{diag:classified-filling-problem2}
\begin{tikzcd}
\pat{\varphi} \ar[d, swap, tail, "i"]  \ar[r, "="] & \typecomp{\yon{c}}{\pat{\varphi}} \ar[d, "p_{\pat{\varphi}}"]  \\
{\yon{c}} \ar[r, swap, "="] & {\yon{c}} \, \diadot{.}
\end{tikzcd}
\end{equation} 

The commutative square \eqref{diag:classified-filling-problem} represents a partial element $(\varphi, \lambda{u})$ of $\alpha$ at $x \co \yon{c}\to X$, 
in the sense of \cref{defn:sem-partial-element}, where $c\forces \varphi\co \Phi$ and $c\forces \lambda{u} \co \pat{\varphi}\to\alpha(x)$.
We show that also the existence of a diagonal filler for~\eqref{diag:classified-filling-problem} can be written as a forcing condition. 
First recall that for any $x \co \yon{c} \to X$ and $a\co\yon{c} \to \UVptd$, the condition
$c \forces a \oftype \alpha(x)$ means that there is a (necessarily unique)  lift of the form,
\begin{equation}\label{diag:classified-filling-problem3}
 \begin{tikzcd}  
  &[20pt] \typecomp{X}{\alpha} \arrow[d]  \\
\yon{c} \arrow[r, swap, "{x}"] \ar[ur, dotted, "{(x,a)}"] & X\diadot{.}
\end{tikzcd}
\end{equation}
Then for maps $\varphi\co \yon{c} \to \Cof$ and $u \co \pat{\varphi} \to \UVptd$ 
as in \eqref{diag:classified-filling-problem} above, we construct the type 
\[
(u\req{\varphi}a) \co \yon{c} \too \UV
\]
as follows.  The maps $\varphi$ and $\alpha$ give rise to types over $\yon{c}$ as indicated below. 
\begin{equation}\label{diag:fillingtype1}
\begin{tikzcd}
\UVptd \ar[d] & \ar[l] \typecomp{\yon{c}}{\pat{\varphi}}  \ar[rd, swap, tail,swap, "i"] && \typecomp{\yon{c}}{\alpha(x)} \ar[ld,swap, "p_{\alpha}"] \ar[r] & \ar[d] \UVptd \\
\UV & \ar[l,"\pat{-}"] \Cof & \ar[l, "\varphi"] \yon{c} \ar[r, swap, "x"] & X \ar[r,swap, "\alpha"] & \UV \mathrlap{.} 
\end{tikzcd}
\end{equation} 
The pullback of these two display maps may be regarded as weakenings of each context by the other; we choose the one indicated in the following diagram:
\begin{equation}\label{diag:fillingtype2}
\begin{tikzcd}
 && \ar[ld, "p_{\alpha(xi)}"] \yon{c}.\pat{\varphi}.\alpha(xi) \ar[rd] && \\
\UVptd \ar[d] & \ar[l] \typecomp{\yon{c}}{\pat{\varphi}}  \ar[ru, dotted, bend left = 20] \ar[rd, tail, "i"] 
		&& \typecomp{\yon{c}}{\alpha(x)} \ar[ld, swap, "p_{\alpha(x)}"] \ar[r] & \ar[d] \UVptd \\
\UV & \ar[l,"\pat{-}"] \Cof & \ar[l, "\varphi"] \yon{c} \ar[r,swap, "x"] & X \ar[r,swap, "\alpha"] & \UV \mathrlap{.} 
\end{tikzcd}
\end{equation} 
Maps $(xi,u) \co \pat{\varphi} \to \typecomp{X}{\alpha}$ and $(x,a) \co \yon{c}\to \typecomp{X}{\alpha}$ as in \eqref{diag:classified-filling-problem} and \eqref{diag:classified-filling-problem3}  then determine two sections of the indicated display map $p_{\alpha(xi)}$. Expressing the same thing in the type theory of $\tE$, we have $\yon{c}.\pat{\varphi} \vdash u \co \alpha(xi)$ and $\yon{c}.\pat{\varphi} \vdash a(i) \co \alpha(xi)$.
Note that the weakened term and type $\yon{c}.\pat{\varphi} \vdash a(i) \co \alpha(xi)$ do not depend on $\pat{\varphi}$.
Over $\yon{c}$ we can then define the type
\begin{equation}
\label{eq:exto-def1}
\textstyle
(u\req{\varphi}a)\ \defeq\  \dprd{\pat{\varphi}}  u  =_{\alpha(xi)} a(i)   \mathrlap{.}
\end{equation} 
As defined $(u \req{\varphi}a) \co \yon c \to U$, \ie we have a type. However, $(u \req{\varphi}a)$ can be viewed as a proposition since it factors through $\iprop$, which expresses the condition that ``$u$ is the restriction of $a$ to $\pat{\varphi}$'' or, equivalently, ``$u$ extends along $\pat{\varphi}\mono \yon{c}$ to $a$", as in~\cite[Definition~5.1]{orton-pitts}. Because of this, we shall sometimes
merely focus on existence of elements of $(u \req{\varphi}a)$ rather than on their particular form. Diagonal filling of \eqref{diag:classified-filling-problem} can now be expressed in terms of forcing as in \cref{lemma:restriction.forcing} below, where we make use of \cref{nom:fun-and-prod-types}.

\begin{lemma}
\label{lemma:restriction.forcing}
Let $\alpha \co X \to \UV$, $x \co \yon{c} \to X$, $\varphi \co \yon{c} \to \Cof$ and $u \co \pat{\varphi} \to E$ so that the diagram of solid arrows
\begin{equation*}
 \begin{tikzcd}
{\pat{\varphi}} \arrow[d, tail, swap, "i"] \arrow[r, "{u'}"]    &[30pt] \typecomp{X}{\alpha} \arrow[d, "p_{\alpha}"]  \\
\yon{c} \arrow[r, swap, "{x}"] \ar[ur, swap, dotted, "{a'}"] & X 
\end{tikzcd}
\end{equation*}
commutes, where $u' =(x i, u)$. Then there is a bijection between:
\begin{enumerate} 
 \item  diagonal fillers $a' \co \yon c \to \typecomp{X}{\alpha}$ as indicated with the dotted arrow in the diagram, 
 \item maps $a \co \yon{c} \to \UVptd$ such that 
 \[
 c  \forces a \co \alpha(x) \mathrlap{,} \quad 
 c \forces u\req{\varphi}a \mathrlap{.}
\]
\end{enumerate}
\end{lemma}

\begin{proof} 
We have already observed that maps $a \co \yon{c} \to \UVptd$ such that $c \forces a \co \alpha(x)$ correspond uniquely to maps $a' \co \yon{c} \to \typecomp{X}{\alpha}$ making the bottom triangle in the diagram commute.   
It therefore suffices to show that, for such $a$ and $a'$, $c \forces u \req{\varphi} a$ if and only if the top triangle also commutes.
Unfolding the definition of $u \req{\varphi}a$, assume that
\[
\textstyle c \forces \dprd{\pat{\varphi}}  u  =_{\alpha(xi)} a(i) \,. 
\]
Although the type $\dprd{\pat{\varphi}}  u  =_{\alpha(xi)} a(i)$ is a proposition and therefore requires no explicit forcing term, let us assume one  in order to apply the rule of forcing for dependent products (\cref{prop:forcing.dprd}) literally. Thus suppose we have
\[
\textstyle c \forces t\co \dprd{\pat{\varphi}}  u  =_{\alpha(xi)} a(i) \,. 
\]
Then by \cref{prop:forcing.dprd}, for every $f \co d \to c$ and every $d \forces e \co \pat{\varphi(f)}$, we have
\begin{equation}
\label{equ:subst-type}
d \forces \funapp (  t (f), e) \co ( u  =_{\alpha(xi)} a(i) ) (f, e)\,,
\end{equation}
where we applied part (ii) of \cref{prop:forcing.dprd}.\footnote{We applied it with $\pat{\varphi}$ for $\alpha$, $u  =_{\alpha(xi)} a(i)$ for $\beta$ and $\yon{c}$ for $X$; and $\id_{\yon{c}}$ for $x \co \yon{c} \to X$.}
The substitution in the type in~\eqref{equ:subst-type} can be calculated as follows:
\[
 \big( u  =_{\alpha(xi)} a(i) \big) (f, e) \ =\ \big( u(f, e)  =_{\alpha(xi(f, e))} a(i(f, e)) \big)\,.
\]
Thus, equivalently, for every $f \co d \to c$ and every lift $(f,e) \co \yon{d}\to \pat{\varphi}$, we have 
\begin{equation}
\label{eq:exto-def3}
d \forces u(f,e)  =_{\alpha(xi(f,e))}  a(i(f,e))\,,
\end{equation}
as indicated in the following diagram:
\begin{equation}\label{diag:lemmapartialfill2}
 \begin{tikzcd}
& {\pat{\varphi}} \arrow[d, tail, swap, "i"] \arrow[r, "{u'}"]    & \typecomp{X}{\alpha} \arrow[d, "p_{\alpha}"]  \\
\yon{d} \arrow[ur, "{(f, e)}"] \arrow[r, swap, "{f}"] & \yon{c} \arrow[r, swap, "{x}"] \ar[ur, swap, dotted, "{a'}"] & X \mathrlap{.} 
\end{tikzcd}
\end{equation}
But $a(i(f,e)) = a(f)$ and $\alpha(xi(f,e)) = \alpha(xf)$, so from \eqref{eq:exto-def3} we have
\[
d \forces u(f,e) =_{\alpha(xf)}  a(f) \,.
\]
Thus
\begin{multline}
u' (f,e)
= (xi,u) (f,e) 
= ( xi(f,e) ,u(f,e) ) 
=  (xf , u(f,e) )\\
= (xf , af)
= (x , a)f 
= a'f 
=  a' i (f,e).
\end{multline}
Since this holds for all $(f,e)$, we have $u' = a'i$, as required, by \cref{thm:equality-pointwise}.

Note that the uniformity condition for forcing the dependent product is trivial because we are forcing a proposition.
\end{proof}

\begin{rmk} \label{eq:restrictiontopartial}
Observe that by function extensionality~\cite[Axiom~2.9.3]{hottbook} the type $(u\req{\varphi}a)$ defined in \eqref{eq:exto-def1} is equivalent to the type $\lambda u = \lambda a$, where the equality is on the function type $\alpha(x)^{\pat{\varphi}} = (\pat{\varphi}\to \alpha(x))$ over $\yon{c}$ (making use of \cref{rmk:forcingTiscomsquare}).  The term $\lambda{a}$ here (see \cref{not:constant-functions}) therefore represents the restriction to $\pat{\varphi}$ of the constant $a$-valued function
and we have $(u\req{\varphi}a)\ =\ (\lambda u = \lambda a)$. 
\end{rmk}

We can now construct a classifier $\TFib(\alpha)$ for trivial fibration structures on~$\alpha$  (cf.~\cite[Section 6]{awodey-cubical-git}).
\begin{equation}
\label{def:TFib}
\textstyle
 \TFib(\alpha)\ \defeq\ \dprd{\varphi \oftype \Cof} \dprd{v \oftype   \alpha^{\pat{\varphi}} } \dsm{a \oftype  \alpha} v = \lambda a \mathrlap{.}
 \end{equation}
Note that under the propositions as types reading, this type expresses that every partial element of $\alpha$ extends to a total one.  Indeed, using the partial elements endofunctor $P$ from~\eqref{defn:typeof-partial-elements}, we can write this in the form
\begin{equation}
\label{eq:collapstriv-fib-type}
\textstyle 
\TFib(\alpha)  = \dprd{(\varphi, v) \oftype \T\alpha} \dsm{a \oftype  \alpha} v = \lambda a \mathrlap{.}	
 \end{equation} 

The notion of a uniform trivial fibration structure now arises inexorably from forcing of the type $\TFib(\alpha)$ of trivial fibration structures on $\alpha$.

\begin{thm} Let $\alpha \co X \to \UV$. The following conditions are equivalent.
\label{prop:ext-implies-tfib}
\label{prop:tfib-to-uniform-triv-fib}
\begin{enumerate}
\item There is a term $t$ such that $X \vdash t \oftype \TFib(\alpha)$.
\item For each $x\co \yon{c}\to X$ we have \( c \forces t_x \oftype  \TFib(\alpha)(x)\), uniformly in $x$.
\item There is a uniform trivial fibration  structure on  $\dsp{\alpha} \co \typecomp{X}{\alpha} \to X$.
\end{enumerate}
\end{thm}

\begin{proof} Observe that (i) and (ii) are equivalent by \cref{prop:validity-and-forcing}.  
 
Now suppose (ii). For each $x \co \yon{c} \to X$, we have \( c \forces t_x \oftype  \TFib(\alpha)(x)  \), uniformly in $x$.  
By definition \eqref{eq:collapstriv-fib-type}  this means  
\[ 
\textstyle
c \forces  t_x \co   \dprd{(\varphi, v) \oftype \T \alpha(x) } \dsm{a \oftype \alpha(x)}  v = \lambda a  \mathrlap{.} 
\]
Unwinding the forcing conditions using \cref{prop:forcing.dprd} and \cref{prop:forcing.dsum}, we see that for every  $f \maps d \to c$ and every
\begin{equation}
\label{eq:coffill0}
\textstyle
d \forces (\varphi, v) \co  \dsm{\varphi \oftype \Phi}  \alpha(xf)^{\pat{\varphi}},
\end{equation} 
we have $t_{xf} (\varphi, v) = \big( t_{x f} (\varphi, v)_1 , t_{x f} (\varphi, v)_2 \big)$ with 
\begin{align}
\label{eq:coffill1} 
d &\forces    t_{x f} (\varphi, v)_1 \co \alpha (x f)\\ 
\label{eq:coffill2} 
d &\forces t_{x f} (\varphi, v)_2 \co  v = \lambda( t_{x f} (\varphi, v)_1 ) . 
\end{align}
Without loss of generality, we may assume that $v= \lambda u$, so that by \eqref{eq:restrictiontopartial} the second condition becomes
\begin{equation}
\label{eq:coffill3}
d \forces t_{x f} (\varphi, \lambda u)_2  \co  u \req{\varphi} t_{x f} (\varphi, \lambda u)_1 .
\end{equation}
Moreover, for any \(g \maps e  \to d\), we have
\begin{equation}
\label{eq:coffill4}
(t_{x f}(\varphi, \lambda u) )g = t_{x fg} (\varphi g , \lambda u g)\,.
\end{equation}
Unfolding the condition \eqref{eq:coffill0} using \cref{prop:forcing.partial} yields a commutative diagram 
\begin{equation}
\label{diag:forcing.cpf}
\begin{tikzcd}
 {\pat{\varphi}} \arrow[r, "{u}"] \arrow[d, tail]  &[20pt] \typecomp{X}{\alpha} \arrow[d] \\
 \yon{d}  \arrow[r, swap, "{x f}"] & X \diadot{,}
\end{tikzcd}
\end{equation}
and the uniformity condition  yields, for any  \(g \maps e  \to d\), a further pullback diagram on the left,
\begin{equation}
\label{diag:forcing.uniform.problems}
\begin{tikzcd}
{\pat{\varphi g}} \arrow[r] \arrow[d, tail] & {\pat{\varphi}} \arrow[r, "{u}"] \arrow[d, tail]  &[20pt] \typecomp{X}{\alpha} \arrow[d] \\
\yon{e}  \arrow[r, swap, "{g}"] & \yon{d}  \arrow[r, swap, "{x f}"] & X \diadot{.}
\end{tikzcd}
\end{equation}
Now we can apply \cref{lemma:restriction.forcing}  to \eqref{eq:coffill1} and \eqref{eq:coffill3} to conclude that 
$t_{x f}(\varphi,\lambda u)_1 $ is a diagonal filler in
\[
\begin{tikzcd}
{\pat{\varphi}} \arrow[d, tail] \arrow[r, "{u}"]  &[20pt]   \typecomp{X}{\alpha} \arrow[d] \\
{\yon{d}} \arrow[ru, dotted] \arrow[r, swap, "{x f}"] &[20pt]  X \diadot{,}
\end{tikzcd}
\]
while $t_{x fg} (\varphi g , \lambda u g)_1$ is similarly a diagonal filler for the outer rectangle in \eqref{diag:forcing.uniform.problems}.

In this way, forcing $\TFib(\alpha)$ produces diagonal fillers 
\begin{equation}\label{defeq:trivfibstru}
j(\varphi, xf, u)\ \defeq\  t_{x f}(\varphi,u)_1
\end{equation}
against all generating cofibrations $\pat{\varphi} \mono \yon{d}$. Finally, the uniformity of these fillers $j(\varphi, xf, u)$ with respect to all \(g \maps e  \to d\) is provided by the uniformity condition \eqref{eq:coffill4}.  

The inference from (iii) to (ii) is essentially the same argument backwards.
\end{proof}

\begin{cor}\label{triv-fib-force}
Let $\alpha \co X \to \UV$. The  display map $p_\alpha \co X.\alpha\to X$ is a uniform trivial fibration if and only if $\tE\forces t\co\TFib(\alpha)$ for some $t$.  Indeed, $t$ itself may be regarded as a trivial fibration structure on $p_\alpha$ in virtue of the specification in \eqref{defeq:trivfibstru}. \qed
\end{cor}

\subsection*{Uniform trivial fibrations as algebras}

For $X \in \tE$, we defined the pointed endofunctor $\T_X \maps \tE/_{X} \to \tE/_{X}$  on the slice topos $\tE/_{X}$ in \eqref{defn:typeof-partial-elements}.  An algebra for $\T_X$ consists of an object 
    $p \co A \to X$ of $\tE/_{X}$ and  a retraction of $\eta_A  \co A \to \T_X(A)$ over $X$, \ie 
    a map $s \maps \T_X(A) \to A $ over $X$ such that
    \[
    \begin{tikzcd}
    A \ar[r, "\eta_A"] \ar[dr, bend right = 20, swap, "\id_A"] & \T_X(A) \ar[d, "s"] \\
     & A 
     \end{tikzcd}
     \]
 commutes.  Such an algebra is the same thing as a uniform fibration structure on $p \co A \to X$ (\cite[Theorem~9.8]{gambino2017frobenius}).
 We give a new simple proof of this fact using forcing, which in particular avoids any reference to Garner's small object argument~\cite{garner:small-object-argument}  and left Kan extensions (cf.~\cite[Proposition 14]{awodey-cubical-git}).

 For this, let us first consider the object~$\Talg(A)$ of algebra structures on $p$ by taking the following pullback in~$\tE/_{X}$, where we write $A$ for $p \co A \to X$ as an object in $\tE/_{X}$ and $\T$ for $\T_X$:
\begin{equation}
\label{diag:algebra-strutures-of-partial-classifier-polynomial}
 \begin{tikzcd}
     \Talg(A) \ar[r] \ar[d, swap]  & {A^{\T(A)}}
      \ar[d, "A^{\eta_A}"] \\
     1 \ar[r, swap, "{\id_A^\sharp}"] & {A^{A}} \mathrlap{.}
    \end{tikzcd}   
\end{equation}
If $p \co A \to X$ is classified by $\alpha \co X \to \UV$, then $\Talg(A) \to X$ is classified by the map $\Talg(\alpha) \co X \to \UV$ defined by setting
\begin{equation} 
 \label{eq:plusalgintt}
 \textstyle
\Talg(\alpha)  =  \dsm{s \oftype \alpha^{\T\alpha}} s\circ\eta = \id_A \mathrlap{.} 
 \end{equation}
We can then strengthen \cref{prop:ext-implies-tfib} to the following isomorphism.

\begin{thm}\label{prop:plusistfib} Let $\alpha \co X \to \UV$. Then there is an isomorphism $\Talg(\alpha)  \cong \TFib(\alpha)$.
\end{thm}
\begin{proof} Recall the description of $\TFib(\alpha)$ in \eqref{eq:collapstriv-fib-type}.
Applying to it the type-theoretic axiom of choice~\eqref{prop:forcing.AC} yields a natural isomorphism
 \begin{equation}
 \label{eq:tfibtoplusalg}
 \textstyle
   \left( \dprd{(\varphi, v) \oftype \T\alpha}  \dsm{a \oftype \alpha} v = \lambda a   \right) \cong 
     \left(\dsm{s \oftype  \alpha^{\T\alpha}} \dprd{(\varphi, v) \oftype \T\alpha}  v = \lambda\,\funapp(s,(\varphi, v)) \right).
 \end{equation} 
We claim that, up to natural isomorphism, the right hand side is the object $\Talg(\alpha)$ of algebra structures on $\alpha$ from \eqref{eq:plusalgintt}.  
It  suffices to show that (in the  context extended by $s \co \alpha^{\T\alpha}$) there is a natural isomorphism
\begin{equation}\label{eq:forcingTalgisoTfib3} 
\textstyle
 \dprd{(\varphi, v) \oftype \T\alpha} \big( v = \lambda\,\funapp(s,(\varphi, v)) \big) \
   \cong\  
   \dprd{a \oftype \alpha} \funapp( s\circ\eta, a) = a )\,.
\end{equation}
Moreover, since both sides are now propositions, it suffices to show that they are logically equivalent, which we will do by forcing.  

Thus take any $x \co \yon{c} \to X$ and any $s \co \yon{c} \to \UVptd$ with $c\forces s \co ({{\T\alpha}\to\alpha})(x)$ and suppose that
\begin{equation}\label{eq:forcingTalgisoTfib1} \textstyle
c \forces \dprd{(\varphi, v) \oftype \T\alpha(x)}  v = \lambda\, \funapp(s,(\varphi, v)) \mathrlap{.}
\end{equation}
Then for all $f \co d \to c$ and all $(\varphi, v)$ such that
\begin{equation}
\label{forces:partialelementtoiso1}
d \forces (\varphi, v) \oftype \T\alpha(xf)   \mathrlap{,}
\end{equation}
we have
\begin{equation}
\label{forces:partialelementtoiso2}
d \forces  v =\lambda\, \funapp(s(f), (\varphi, v))  \mathrlap{.} 
\end{equation}
To see what this means, consider the diagram on the left in \eqref{diag:pluseqtfib} below:
\begin{equation}
\label{diag:pluseqtfib}
\begin{tikzcd}
 \pat{\varphi} \arrow[rr, "{u}"] \arrow[d, tail,swap, "i"] 
&& {\alpha(x)}  \arrow[d, tail, swap, "\eta"] \\
 \yon{d}  \arrow[rr,swap, "\pair{\varphi}{\lambda u}"]
  \arrow[rru, swap, "j", dotted]
&& {\T\alpha(x)} \mathrlap{,}  \arrow[u, bend right=30, swap, dotted, "\sigma_f "]   
\end{tikzcd}
\qquad \qquad 
\begin{tikzcd}
 \pat{\tcof} \arrow[rr, "{u}"] \arrow[d, equals] 
&& \alpha(x)  \arrow[d, tail, swap, "\eta"] \\
 \yon{d}  \arrow[rr,swap, "\pair{\tcof}{a(f)}"]
 \arrow[rru, swap, "a(f)", dotted]
&& \T\alpha(x) \mathrlap{.}  \arrow[u, bend right=30, swap, dotted, "\sigma_f"] 
\end{tikzcd}
\end{equation}
The solid part is the pullback in $\tE/\yon{c}$ arising from \eqref{forces:partialelementtoiso1} and the forcing rules for $\T\alpha(x)$ in \cref{prop:forcing.partial}, with $\lambda u = v$, and we have written $\yon{d}$ for the map $\yon{f}\co\yon{d}\to \yon{c}$ and $\alpha(x)$ for the projection $p_{\alpha(x)} \co \typecomp{\yon{c}}{\alpha(x)}\to \yon{c}$ and similarly for $\T\alpha(x)$.  

By \cref{eq:restrictiontopartial}, the condition \eqref{forces:partialelementtoiso2} thus becomes
\[\textstyle
d \forces u =^{\varphi} \funapp(s(f),(\varphi, \lambda u)).
\]
By \cref{lemma:restriction.forcing}, the map $j = \funapp(s(f),(\varphi, \lambda u))$ therefore fits on the diagonal in the lefthand diagram of \eqref{diag:pluseqtfib} and makes the top triangle  commute.

Take $\sigma_f$ with $\lambda \sigma_f = s(f)$, so that  
\[
j = \funapp(s(f),(\varphi, \lambda u)) = \funapp(\lambda\sigma_f , (\varphi, \lambda u)) = \sigma_f (\varphi, \lambda u).
\]
Then we claim that $\sigma_f \,\eta = \id$, which will give the required righthand side of \eqref{eq:forcingTalgisoTfib3}.

To prove the claim, recall that \eqref{forces:partialelementtoiso2}  holds for all  $f \co d\to c$ and all $(\varphi, \lambda u)$ with
\[
d \forces (\varphi, \lambda u) \oftype \T\alpha(xf) \mathrlap{.}
\]  
Since for any $c \forces a \oftype \alpha(x)$, we can take $c \forces (\tcof, \lambda a) \oftype \T\alpha(x)$ to get~$d \forces (\tcof, \lambda a(f)) \oftype \T\alpha(xf)$,  we will then have $a(f) = \sigma_f \,(\tcof, a(f))$. Indeed, in this case we have $i = \id \co \pat{\tcof}\mono \yon{d}$ and so $u = a(f)$ as on the right in \eqref{diag:pluseqtfib} above, so the (partial) commutativity of the diagram on the left specialises on the right to give $a(f) = \sigma_f\,(\tcof, a(f)) $.  

Since this holds for all $f \co d\to c$ and is uniform with respect to all $g\co e \to d$, and since~$(\tcof, a_f) = \eta(a_f)$, we conclude that 
\begin{equation*}
\textstyle
c \forces \dprd{a \oftype \alpha(x)} (  \sigma\, \eta(a) = a \big) \mathrlap{.}
\end{equation*}

Conversely, $\id =s \circ \eta$ implies $u = s \circ \eta \circ u$, and since $\eta\circ u = (\varphi, \lambda u) \circ i$, it follows that~$u = s\circ (\varphi,u) \circ i$, and thus $u =^\varphi \funapp(s, (\varphi, u))$.
\end{proof}

\begin{cor}[Pullback stability of $\TFib$] \label{cor:TFibstable}
Let $\alpha \co X \to \UV$. Then the object $\TFib(\alpha)$ over~$X$ is stable under pullback along any map $t \co Y \to X$, in the sense that the right hand square below is a pullback whenever the left hand one is:
 \[
 \begin{tikzcd}
\typecomp{Y}{\beta} \arrow[d] \arrow[r] & \typecomp{X}{\alpha} \arrow[d] \\ 
Y \ar[r, swap, "t"] & X \mathrlap{,} 
 \end{tikzcd} \qquad
  \begin{tikzcd}
 \TFib(\beta)  \arrow[d] \ar[r] & \TFib(\alpha) \arrow[d] \\
Y \ar[r, swap, "t"] & X \mathrlap{.}
  \end{tikzcd}
 \]
In other words, $\pbk{t} \TFib(\alpha) \cong \TFib(\alpha(t))$. 
 \end{cor}
 
 \begin{proof}
 Note that we have simplified the notation $\typecomp{X}{\TFib(\alpha)} \to X$ to $\TFib(\alpha) \to X$.
 The description in~\eqref{diag:algebra-strutures-of-partial-classifier-polynomial} involves only $\T$-algebras, exponentials, and pullbacks, all of which are stable under all pullbacks (as is $\T$ itself).
 \end{proof}
 
It also follows from \cref{prop:plusistfib} that, for every small $A\in\tE$, the unique map $\T{A} \to 1$ is a uniform trivial fibration, as it can be equipped with the $(\T, \eta)$-algebra structure  $\mu_A \maps \T{\T{A}} \to \T{A}$.  Thus, we have the following result.

\begin{cor}\label{rmk:tfib-as-algebras} Every small map can be factored as a cofibration followed by a uniform trivial fibration. 
\end{cor}
 
\begin{proof} Consider $p \maps A \to X$  as an object of the slice category $\tE/_{{X}}$ and apply the functor $\T_X \maps \tE/_{X} \to \tE/_{X}$ to form the object $\T_X (p) \co \dom({\T_X (p)}) \to X$ in $\tE/_{X}$. The map~$\eta_p \maps p \to \T_X(p)$ is a cofibration in $\tE/_{X}$, hence in $\tE$, while $\T_X (p)$ is a uniform trivial fibration, as was just observed.  
\end{proof}

We conclude the section by constructing a universal uniform trivial fibration, using the method of \cite[Proposition 85]{awodey-cubical-git}.
For $\alpha \co X \to \UV$, applying \cref{cor:TFibstable} to the pullback in~\eqref{diag:context-extension}
we obtain a pullback diagram as the outside square below, where $\id \co \UV \to \UV$ of course classifies $\pi \co \UVptd\to \UV$.
 \begin{equation}\label{diag:alphafactors}
 \begin{tikzcd}
  \TFib(\alpha)  \arrow[d] \ar[r] &[20pt] \TFib(\id) \arrow[d]  \\
X \ar[r, swap, "\alpha"]  \arrow[u, bend left=30, swap, dotted]  \arrow[ru, dotted, "\bar{\alpha}"]  &  \UV \mathrlap{.}
 \end{tikzcd}
 \end{equation}
Trivial fibration structures on $\alpha$, \ie sections of $\TFib(\alpha)$, then correspond uniquely to 
factorisations~$\bar{\alpha}$, as indicated.

\begin{thm}[Universal trivial fibration]
\label{prop:classifier-trivfib}
There exists a \strong{universal small uniform trivial fibration} $\TFibptd \to \TFib$, 
in the sense that every small uniform trivial fibration $p \co A \to X$ is a pullback of $\TFibptd \to \TFib$ along a canonically determined classifying map $\bar{\alpha} \co X \to \TFib$.
\begin{equation}
\begin{tikzcd}
     A \ar[r, dotted] \ar[d, swap, "p"]  & \TFibptd 
      \ar[d] \\
    X \ar[r, swap,dotted, "\bar{\alpha}"] & \TFib \mathrlap{.} 
    \end{tikzcd}   
\end{equation}
\end{thm}

\begin{proof}
Define $\TFib = \TFib(\id)$ and  $\TFibptd \to \TFib$ by the pullback indicated on the right below.
\[
\begin{tikzcd}[column sep=small, row sep = small]
& A
\arrow[rd, dotted,swap] 
\arrow[ddd]
\arrow[rr] && \UVptd
\arrow[ddd, "\dspu"] \\
&& \TFibptd 
\arrow[ddd,swap,pos=0.3,crossing over]
\arrow[ur, swap]
  \\ 
  \\ 
  & 
X \arrow[rd, dotted,swap, "\bar{\alpha}"] 
    \arrow[rr, "\alpha",pos=0.3] && \UV \mathrlap{.}  \\ 
&& \TFib
\arrow[ur,swap,sloped]
\end{tikzcd}
\]
As was just shown \eqref{diag:alphafactors}, a trivial fibration structure on $A \to X$ corresponds to a factorisation~$\bar{\alpha}$ of~$\alpha$, thus making a pullback square as indicated on the left above.  In particular, taking $X = \TFib$ and~$\alpha$ the projection $\TFib\to \UV$ and $\bar{\alpha} = \id$, we see that $\TFibptd\to \TFib$ is itself a trivial fibration, and thus it is a universal one.
\end{proof}

\begin{rmk}\label{rmk:TFib}
Using variables, the identity map $\id \co \UV \to \UV$  is the interpretation of the judgement $\alpha \co\UV \vdash \alpha \co\UV$, from which we also have $\alpha \co\UV \vdash \TFib(\alpha)\co\UV$.  Thus for the type $\TFib = \UV.\TFib(\id)$ we have $\TFib =  \dsm{\alpha\co\UV}\!\TFib(\alpha)$ 
and then $\TFibptd = \dsm{(\alpha,t)\co\TFib}\alpha$. 
\end{rmk}


\section{Uniform fibrations}
\label{sec:unifib}

The definition of a uniform fibration was originally introduced in the setting of categories of cubical sets in~\cite{CoquandT:modttc,CoquandT:cubttc} to provide a constructive model of the Univalence Axiom. The notion was analysed in terms of algebraic weak factorisation systems in~\cite{SwanA:algwfs} and generalised 
to the setting of presheaf categories equipped with a class of cofibrations
and a suitable interval object in~\cite{gambino2017frobenius,awodey-cubical-git}. The definition of~\cite[Section~8.2]{CoquandT:cubttc} was rephrased in the internal type theory of such a category in~\cite[Definition~5.6]{orton-pitts}.
Here, we relate precisely the category-theoretic and the type-theoretic formulations of the notion of uniform fibration in the latter setting. In particular, we show how this notion arises naturally from our Kripke-Joyal forcing for the internal type theory of a presheaf category $\tE = \PshC$ of cubical sets with connections.
We expect that a similar treatment could be given for other cases, such as \cite{ABC,awodey-cubical-git}.

We assume a class of cofibrations as in \cref{defn:cofibrations}, classified by $\Cof \mono \Omega$.

\begin{defn}\label{def:IntervalWithConnections} An \strong{interval with connections} is an object  $\II$ in $\tE$ 
equipped with endpoints $\delta^k \co 1 \to \II$
and connections $\conn_k \co \II \times \II \to \II$, for $k=0,1$,
satisfying the following axioms.
\begin{enumerate}
\item The maps $\delta^k \co 1 \to \II$ are cofibrations, for $k=0,1$.
\item The pullback of $\delta^0$ and $\delta^1$ is the initial object $0$,
\[
\begin{tikzcd}
0 \ar[d]  \ar[r] & 1 \ar[d, "{\delta^1}"] \\
1 \ar[r, swap, "{\delta^0}"] & \II . 
\end{tikzcd} 
\] 
\item The diagrams
\[
\begin{tikzcd}
\II \ar[r, "{(\delta^k, 1)}"]  \ar[d] &[20pt] \II \times \II \ar[d, "{\conn_k}"] &[20pt] \II \ar[l, swap, "{(1,\delta^k)}"]  \ar[d] \\
1 \ar[r, swap, "{\delta^k}"] & \II & 1 \ar[l, "{\delta^k}"] \mathrlap{,} 
\end{tikzcd}  \qquad
\begin{tikzcd} 
\II \ar[dr, equal] \ar[r, "{(\delta^{1-k}, 1)}"] &[20pt] \II \times \II \ar[d, "{\conn_{k}}"] &[20pt] \II \ar[l, swap, "{(1, {\delta^{1-k}})}"]  \ar[dl, equal] \\
 & \II & 
 \end{tikzcd} 
\] 
commute, for $k = 0,1$.
\end{enumerate} 
\end{defn}

We recall the definition of a uniform fibration from \cite[Section~7]{gambino2017frobenius}. This involves the pushout-product construction, described in \cite[Construction 11.1.7]{riehl-cat-homotopy} under the name of Leibniz construction. We define a \strong{naive trivial cofibration} to be the pushout-product map $m \otimes\delta^k$ in
\[
\begin{tikzcd}[column sep = large]
S \ar[r, tail, "1 \times \delta^k"] \ar[d, tail, "m"'] & S \times \II \ar[d] \ar[ddr, tail, bend left = 30, "m \times 1"] &  \\
T \ar[r]  \ar[drr, bend right = 20, tail,  "1 \times \delta^k"'] & T +_S (S \times \II) \ar[dr, tail, "m \otimes \delta^k" pos=0.4]  & \\
 & & T \times\II \mathrlap{,}
\end{tikzcd}
\]
where $m \co S \mono T$ is an arbitrary cofibration and $\delta^k \co 1\mono\II$, for $k=0,1$ is an endpoint.  Note that $m \otimes\delta^k$ is indeed a cofibration by Remark \ref{remark:cofibrations}.
Naive trivial cofibrations are stable under pullback along arbitrary maps $t \co T' \to T$, in the sense that one then has a pullback square
\begin{equation}
\label{diag:uniformcofibration1}
\begin{tikzcd}
T' +_{S'} (S'\times\II) \arrow[r] \arrow[d, tail,swap,"m' \otimes\delta^k"] 
  & T +_S (S \times\II) \arrow[d, tail, "m \otimes\delta^k"]   \\
T' \times \II  \arrow[r, swap, "t\times\II"] 
  &  T \times \II  \mathrlap{,}
\end{tikzcd}
\end{equation}
where $m' \co S'  \mono T'$ is the pullback of $m$ along $t$.

\begin{defn}\label{def:uniformfibration}
Let $p \co A \to X$ be a small map. A \emph{uniform  fibration structure} on $p$ consists of 
a function $j$ that assigns a dotted filler $j(m,u,v) \co T \times \II \to A$ to every diagram
of solid arrows
\[
\begin{tikzcd}
T +_S (S \times\II) \arrow[r, "u"] \arrow[d, tail,  swap, "m \otimes\delta^k"]  & A \ar[d] \\
T \times \II  \arrow[r, swap, "v"] \arrow[ru, dotted] 
& X \mathrlap{,} 
\end{tikzcd}
\]
where  $m$  is a cofibration, and $k=0,1$, subject to the following
uniformity condition: for any map $t \co T' \to T$ and induced pullback square on the left,
\begin{equation}
\label{diag:uniformcofibration2}
\begin{tikzcd}
T' +_{S'} (S'\times\II) 
	\arrow[r] 
	\arrow[d, tail,swap,"m'\otimes\delta^k"] 
&[20pt] T +_{S} (S\times\II) 
	\arrow[r, "u"] 
	\arrow[d, tail] 
 &[20pt] A
 	\arrow[d] \\
T' \times \II  
	\arrow[r, swap, "t\times\II"] 
	\arrow[rru,dotted] 
&  T \times \II  
	\arrow[r, swap, "v"]  
	\arrow[ru, dotted] 
& X \mathrlap{,}
\end{tikzcd}
\end{equation}
we have 
\[
j(m', u (t\times i)', v (t\times i)) = j(m, u, v) \circ (t\times\II)\,,
\]
 where $(t\times \II)'$ is the evident pullback of $(t\times \II)$.
\end{defn} 

As a functor on the arrow category $\tE^\ra$, the pushout-product $(-)\otimes\delta^k$ has a right adjoint $\delta^k \Raw(-)$, the pullback-hom, taking $p \co A \to X$ to the map $\delta^k \Rightarrow p$ indicated in the following diagram.
\begin{equation}
\label{diagram:pullbackhom2}
\begin{tikzcd} 
A^\II \ar[rdd, swap, "{p^\II}", bend right = 20] \ar[rd, dotted, "{\delta^k \Rightarrow{p}}"]  \ar[rrd, bend left = 20, "{A^{\delta^k}}"]  && \\
&X^\II \times_{X} A \ar[d] \ar[r] & A \ar[d, "p"] \\
& X^\II \ar[r, swap, "{X^{\delta^k}}"] &  X \mathrlap{.} 
\end{tikzcd} 
\end{equation}
By adjointness,  these operations can be seen to satisfy 
\begin{equation}\label{eq:leibnizadjunction}
(m \otimes \delta^k ) \pitchfork p \quad \text{if and only if} \quad m \pitchfork (\delta^k \Raw f) \mathrlap{,}
\end{equation}
where the notation $u \pitchfork v$ means, as usual, that $u$ has the left lifting property with respect to $v$.  Indeed, more is true:\ by adjointness there is a natural bijection between the diagonal fillers witnessing the statement $(m \otimes \delta^k ) \pitchfork p$ and those establishing $m \pitchfork (\delta^k \Raw f)$. 

The condition in~\eqref{diag:uniformcofibration2} for $p \co A \to X$ to be a uniform fibration can therefore be reformulated as the requirement that, for all cofibrations $m \co S \mono T$, we have $m \pitchfork (\delta^k \Raw p)$ uniformly in $m$.  But this of course just says that $\delta^k \Raw p$ is a uniform \emph{trivial} fibration in the sense of \cref{defn:uniformtrivialfibration}.  We record this for later use (cf.~\cite[Proposition~7.4]{gambino2017frobenius}).

\begin{prop}\label{prop:fibistfib}
A small map $p \co A \to X$ is a uniform fibration if and only if the map $\delta^k \Raw p \co A^\II \to X^\II \times_X A$ is a uniform trivial fibration, for $k=0,1$. \qed
\end{prop}

\begin{rmk}\label{rem:gentrivcof}
As was the case for uniform trivial fibrations (\cref{thm:triv-cof-againts-generating}), a uniform fibration structure is determined by its values on the \strong{generating trivial cofibrations}, defined as those of the form $m \otimes \delta^k $  where $m \co S\mono \yon{c}$ is a cofibration with a representable codomain. We leave it to the reader to reformulate \cref{def:uniformfibration} of a uniform fibration structure on a map $p \co A \to X$ in these terms.
\end{rmk}

Our aim now is to show how \cref{def:uniformfibration} can be formulated in the internal type theory of $\tE$ and how it is thereby related to forcing.  In light of \cref{prop:fibistfib} we would like to use \cref{prop:ext-implies-tfib} to show that there is an element of $\TFib(\delta^k \Rightarrow{f})$ over $X$, but for this to make sense~$\delta^k \Rightarrow{f}$ would need to be given as a display map.  Thus let $\alpha \co X \to \UV$ classify $A\to X$. Consider the case $k = 0$ and let $\mathsf{ev_0}  \co X^\II\ra X$ be the evaluation at the point~$\delta^0 \co 1\to \II$  (i.e.\ $\mathsf{ev_0}  = X^{\delta^0}$), so that the pullback in \eqref{diagram:pullbackhom2} above becomes

\begin{equation}\label{diagram:uniformfib1}
\begin{tikzcd} 
\typecomp{X^\II}{\alpha(\mathsf{ev_0})} \ar[d] \ar[r]  & \typecomp{X}{\alpha} \ar[d] \\
X^\II \ar[r, swap, "{\mathsf{ev_0}}"] &  X \mathrlap{.}
\end{tikzcd} 
\end{equation}

In what follows, we adopt the indexed-family style of writing a display map $X.\alpha\to~X$ as $\big(\dsm{x \co X} \alpha(x)\big)\to X$ for easier comparison with \cite{orton-pitts}.
For any $x\co \yon{c} \to X^\II$ and section $(x,a) \co \yon{c} \to \alpha(x_0)$, for $k = 0, 1$, we  define 
$F_k(x,a) \co \yon{c} \to \UV $ as 
\begin{equation}
\label{def:fibertypeF0}
\textstyle
F_k(x,a) \
  \defeq\ \dsm{s \oftype \dprd{i \oftype \II} \alpha(x_i)} s_k =  a \mathrlap{.} 
\end{equation}

\begin{prop} \label{prop:leibniz-trick}
Let $\alpha \co X  \to \UV$. Then the following conditions are equivalent.
\begin{enumerate}
\item The projection $p\co \big( \dsm{x \oftype X} \alpha(x) \big) \to X$ admits a uniform fibration structure.
\item\label{item:classtypeTFib} There is a term of type
\[
\textstyle
\dprd{x \oftype X^\II}  \big( \dprd{a \oftype \alpha(x_0)} \TFib \, F_0 (x,a) \big) \times 
\big(\dprd{a \oftype \alpha(x_1)} \TFib \, F_1(x,a) \big) \,,
\]
where $\TFib$ is as defined in \eqref{def:TFib}.
\end{enumerate}
Moreover, there is a bijection between such structures and such terms.
\end{prop}

\begin{proof} Writing $x_0 = \mathsf{ev_0}(x)$, for \eqref{diagram:uniformfib1} we then have,
\begin{equation*}\label{diagram:uniformfib2}
\begin{tikzcd} 
\dsm{x \oftype X^\II} \alpha(x_0)  \ar[d] \ar[r]  & \dsm{x \oftype X} \alpha(x) \ar[d] \\
X^\II \ar[r, swap, "{\mathsf{ev_0}}"] &  X \mathrlap{.}
\end{tikzcd} 
\end{equation*}
Using the isomorphism
\[
\textstyle
A^\II\ \cong\ \big( \dsm{x \oftype X} \alpha(x) \big)^\II\ \cong\ \dprd{i \oftype \II} \dsm{x \oftype X}\alpha(x)\  
\cong\ \dsm{x \oftype X^\II} \dprd{i \oftype \II} \alpha(x_i) \mathrlap{,} 
\]
the map~$p^\II \co A^\II\ra X^\II$ may  be rewritten as a display map over $X^\II$ in the form 
\[\textstyle
\big( \dsm{x \oftype X^\II} \dprd{i \oftype \II} \alpha(x_i) \big) \to X^\II\,.
\]
Up to isomorphism, our previous diagram \eqref{diagram:pullbackhom2} then becomes
\[
\begin{tikzcd} 
\dsm{x  \oftype X^\II} \dprd{i \oftype \II} \alpha(x_i)  \ar[rdd, bend right = 15] \ar[rd, "{\delta^0 \Rightarrow{p}}"] \ar[rrd, bend left = 12] && \\
& \dsm{x \oftype X^\II} \alpha(x_0) \ar[d] \ar[r]  & \dsm{x \oftype X} \alpha(x)  \ar[d] \\
& X^\II \ar[r,swap, "{\varev_0}"] & X \mathrlap{.}
\end{tikzcd}
\]
Thus for any $x \co \yon{c} \to X^\II$ the map $(\delta^0 \Raw p)(x)$ over $\yon{c}$ is the $0^{th}$ projection of the type family $\dprd{i \oftype \II} \alpha(x_i)$, 
\[
\textstyle
\pi_0 \co \big( \dprd{i \oftype \II} \alpha(x_i) \big)  \to \alpha(x_0) \mathrlap{,}
\] 
as indicated in the following diagram:
\begin{equation}\label{diagram:pullbackhom4}
\begin{tikzcd} 
{\dprd{i \oftype \II}} \alpha(x_i) 
	\ar[rdd, bend right = 15pt]  
	\ar[rd, swap, "\pi_0" ] \ar[r] 
 & \dsm{ x \oftype X^\II} \dprd{i \oftype \II} \alpha(x_i)
 	\ar[rd, "{\delta^0 \Rightarrow{p}}"]  
	\ar[rrd, bend left =12] && \\
& \alpha(x_0) 
	\ar[d] 
	\ar[r] 
& \dsm{x \oftype X^\II} \alpha(x_0)
	\ar[d] 
	\ar[r] 
&\dsm{x \oftype X} \alpha(x) \ar[d]  \\
& \yon{c} 
	\ar[r, swap, "x"] 
& X^\II 
	\ar[r, swap, "{\varev_0}"]
&  X \mathrlap{.}
\end{tikzcd} 
\end{equation}

Thus, for $x\co \yon{c} \to X^\II$ and section $(x,a) \co \yon{c} \to \alpha(x_0)$, the type $F_0(x,a) \co \yon{c} \to \UV $ defined in~\eqref{def:fibertypeF0}
can be regarded as the fiber of $\pi_0$ at $(x, a)$.
The display map 
\[\textstyle
\big( \dsm{(x,a) \oftype \dsm{x \co X^\II} \alpha(x_0)} F_0(x,a) \big) \too\ \dsm{x\co X^\II} \alpha(x_0) 
\]
is then isomorphic to ${\delta^0 \Rightarrow{p}}$ over $\dsm{x \co X^\II} \alpha(x_0)$.
Similarly, for $(x,a) \co \yon{c} \to \alpha(x_1)$, let $F_1(x,a) \co \yon{c} \to \UV$ be the analogous type consisting of fibers of $\pi_1\co \big(\dprd{i \oftype \II} \alpha(x_i)\big) \to \alpha(x_1)$ rather than $\pi_0$.
The type in \cref{item:classtypeTFib} then classifies uniform fibration structures on $p\co\big(\dsm{x \oftype X} \alpha(x)\big) \to X$ by \cref{prop:ext-implies-tfib}. 
\end{proof}

Recall from \cite[Definition~5.6]{orton-pitts}  the following definition of the type of \emph{$0$-directed filling structures} $\Fill_0(\alpha) \co \UV$ on an $\II$-indexed family of types $\alpha \co \II\to \UV$.
\begin{multline}\label{eqn:pplift2}
\textstyle
\Fill_0(\alpha) = \\
\textstyle
\big( \dprd{\varphi \oftype \Cof} \dprd{v \oftype \pat{\varphi} \to \prod_{i \oftype \II} \alpha_i} \dprd{a \oftype \alpha_0} 
({\tilde{v}_0}=\lambda{a}) \big)
\to\ \sum_{ s \oftype \prod_{i \oftype \II} \alpha_i} (s_0 =_{\alpha_0} a ) \times ({v}=\lambda{s}) \,,
\end{multline}
where 
\[ 
\textstyle
\tilde{(-)}\ \maps \big( \pat{\varphi} \to \dprd{i \oftype \II} \alpha_i \big) \iso\ \dprd{i \oftype \II} \pat{\varphi} \to \alpha_i 
\] 
is the canonical isomorphism, and the type $(v=\lambda s)$ is as in \eqref{eq:restrictiontopartial}.
There is an analogous condition $\Fill_1(\alpha)$ in which $1$ replaces $0$ everywhere, describing $1$-directed filling from the other end of the interval $\II$.  Note that for $\alpha \co \II \to \UV$, the types $\Fill_0(\alpha)$ and $\Fill_1(\alpha)$ are \emph{closed}.  

Now for any object $X$ and family of types $\alpha \co X\to\UV$, we have the family $x \co X^\II \vdash \alpha\circ x \co \II\to\UV$ obtained by ``restricting $\alpha$ along paths in $X$'', i.e.\ composing $\alpha$ and $\ev \co X^\II \times \II \to X$.  Thus we have the closed type
\begin{equation}\label{eq:OPfibtype}\textstyle
 \Fib(\alpha)\ \defeq\ \dprd{x \oftype  X^\II} \Fill_0( \alpha\circ{x} ) \times  \Fill_1( \alpha\circ{x} )\,.
 \end{equation} 
We  show below that $\Fib(\alpha)$ is isomorphic to the type in \cref{item:classtypeTFib} of \cref{prop:leibniz-trick} and therefore also classifies uniform fibration structures on $\dsm{x \oftype X} \alpha(x)\to X$. 

\begin{prop} \label{prop:FibvsFill}
For any object $X$ and any $\alpha \co X \to \UV$, there is an isomorphism,
\[\textstyle
\Fib(\alpha)\ \cong\ 
\dprd{x \oftype X^\II} \big( \dprd{a \oftype \alpha(x_0)} \TFib \, F_0(x ,a) \big) \times \big( 
\dprd{a \oftype \alpha(x_1)} \TFib \, F_1 (x ,a) \big) \,.
\]
\end{prop}

\begin{proof}
It clearly suffices to show  
\[\textstyle
\Fill_0( \alpha\circ x )\ \cong\ \prod_{a \oftype \alpha(x_0)} \TFib \, F_0(x ,a) \,,
\]
in context $x \co X^\II$, and similarly for $\Fill_1$.  Using \eqref{def:fibertypeF0}, the type on the right becomes
\begin{equation}\label{eq:fibfromtfib3}
\textstyle
\dprd{a \oftype \alpha(x_0)} \TFib \left( \sum_{s \oftype \prod_{i \oftype \II} \alpha(x_i)} 
s_0 =  a \right) \,.
\end{equation}
Inserting the definition of $\TFib$ from \eqref{def:TFib} results in the  type
\begin{equation}
\label{eq:fibfromtfib4}
\textstyle
\dprd{a \oftype \alpha(x_0)}\ 
\dprd{\varphi \oftype \Cof}\ 
\dprd{v \oftype {\pat{\varphi}}\to\sum_{s \oftype \prod_{i \oftype \II} \alpha(x_i) } s_0 =  a}\ 
\dsm{t \oftype \sum_{s \oftype \prod_{i \oftype \II} \alpha(x_i)} s_0 = a }\  
v =\lambda{t}\,,
\end{equation}
which simplifies  to
\begin{multline}\label{eq:fibfromtfib5}
\textstyle
\big( \dprd{a \oftype \alpha(x_0)}\ 
\dprd{\varphi \oftype \Cof}\
\dprd{v \oftype \pat{\varphi}\to \dprd{i \oftype \II} \alpha(x_i)} ({\tilde{v}_0}=\lambda{a}) \big)\\ 
\textstyle
\to\ \dsm{s \oftype \prod_{i \oftype \II} 
\alpha(x_i)} (s_0 = a) \times (v=\lambda{s})\,.
\end{multline}
Comparing with \eqref{eqn:pplift2}, we see that this is indeed $\Fill_0( \alpha\circ x)$.  The case of $\Fill_1$ is entirely analogous.
\end{proof}

It follows that the  type $\Fib(\alpha)$ from \eqref{eq:OPfibtype} also classifies uniform fibration structures on the type family $\alpha \co X\to\UV$.
\begin{cor} \label{prop:uniform-fibration-from-fib-type}
Let $\alpha \co X \to \UV$. Then the following conditions are equivalent.
\begin{enumerate}
\item The projection map $p\co \big( \dsm{x \oftype X} \alpha(x) \big) \to X$  is a uniform fibration.
\item There is a term $t \oftype \Fib(\alpha)$.
\end{enumerate}
Moreover, there is a bijection between uniform fibration structures as in \cref{def:uniformfibration} and terms $t \oftype \Fib(\alpha)$.
\end{cor}

The  construction of a universal uniform fibration given below (following \cite{awodey-cubical-git,coquand-slides,LOPS18}) requires the interval $\II$ to be \emph{tiny}, meaning that the ``path space'' functor $X^\II$ has a right adjoint $X_\II$ called the $\II^{th}$-\emph{root}.  
This is the case, for instance, when $\tE = \PshC$ is the category of cubical sets with $\cC$  the Dedekind cube category \cite{BuchholtzMoorehouse} and $\II = \yon{[1]}$  the 1-cube, which then also satisfies the conditions in \cref{def:IntervalWithConnections}.

In order to use the approach of \cref{prop:classifier-trivfib} to construct a universal fibration, we need the analogue  of \cref{cor:TFibstable} for $\Fib(\alpha)$, i.e.\ stability under pullback.  Unlike $\TFib(\alpha)$, however, the object  $\Fib(\alpha)$ is not  indexed over $X$, but instead has the form
 \begin{equation}
 \label{eq:deffillA}
 \textstyle
 \Fib(\alpha) = \dprd{x \oftype X^I}\Fill (\alpha\circ x) \mathrlap{,}
   \end{equation}
where $\Fill(\alpha\circ x) \defeq \Fill_0( \alpha\circ x) \times \Fill_1( \alpha\circ x)$. Now $\Fill(\alpha\circ x)$ is also not indexed over $X$, but rather over $X^\II$.  However, we can use it to construct a pullback stable family $\Fib^*(\alpha)$ over $X$ as follows (cf.~\cite{coquand-slides,LOPS18}).  Applying the root functor to $\Fill(\alpha\circ x)$, regarded as a display map over $X^\II$, and pulling back along the unit of the adjunction, 
 \begin{equation}\label{diagFibstar}
 \begin{tikzcd}
 \Fib^*(\alpha)\arrow[d] \arrow[r]  \pbmark & \Fill(\alpha\circ x)_\II\arrow[d] \\ 
X \ar[r,swap,"\eta"] & (X^\II)_\II \mathrlap{,} 
 \end{tikzcd}
 \end{equation}
we obtain a family $\Fib^*(\alpha)$ on $X$ that parametrises fibration structures on $\alpha$, in the sense made precise in \cref{prop:Fibstable}
below.  It may be noted that this construction of~$\Fib^*(\alpha)$ is categorical, rather than in the type theory of~$\tE$ from \cref{sec:types}. In fact, such an internal construction is impossible, as is shown in~\cite{LOPS18}, where a type theoretic definition of~$\Fib^*(\alpha)$ is obtained by extending of the internal type theory of $\tE$ by a modal operator describing the comonad $\Delta\circ\Gamma$ on $\tE$.

\begin{prop}\label{prop:Fibstable}
Let $\alpha \co X \to \UV$. 
\begin{enumerate}
\item There is a natural bijection between sections of $\Fib^*(\alpha)$ over $X$ and global sections $1\to\Fib(\alpha)$.
\item The object $\Fib^*(\alpha)$ over $X$ is stable under pullback along any map $t \co Y \to X$, in the sense that the 
right hand square below is a pullback whenever the left hand one is.
 \[
 \begin{tikzcd}
Y. \alpha(t)  \arrow[d] \arrow[r] & X.\alpha \arrow[d] \\ 
Y \ar[r, swap,"t"] & X \mathrlap{,} 
 \end{tikzcd}  \qquad
  \begin{tikzcd}
  \Fib^*(\alpha(t))  \arrow[d] \ar[r] & \Fib^*(\alpha) \arrow[d] \\ 
   Y\ar[r, swap,"t"] & X \mathrlap{.} 
  \end{tikzcd}
 \]
\end{enumerate}
\end{prop}
\begin{proof}
For (i), by the specification~\eqref{eq:deffillA}, points $1\to\Fib(\alpha)$ correspond bijectively to sections 
of~$\Fill(\alpha\circ x)$ over $X^\II$.  Let $\mathsf{fill}(\alpha) \co X^\II \to \UV$ classify  $\Fill(\alpha\circ x) \to X^\II$.  We then have the correspondence indicated below between maps $j_1,j_2,j_3,j_4$,
\[
\begin{tikzcd}
 \Fill(\alpha\circ x) \arrow[d] \arrow[r] & \UVptd\arrow[d] \\ 
X^\II \ar[r,swap,"\mathsf{fill}(\alpha)"] \ar[u,bend left=30,dotted,"j_1"] \ar[ur,dotted,"j_2"] & \UV \mathrlap{,} 
 \end{tikzcd}
\qquad
 \begin{tikzcd}
X. \mathsf{fill}(\alpha)^\sharp \arrow[d] \arrow[r] & (\UVptd)_\II \arrow[d] \\ 
X \ar[r,swap,"{\mathsf{fill}(\alpha)^\sharp}"] \ar[u,bend left=30,dotted,"j_4"] \ar[ur,dotted,"j_3"] & \UV_\II \mathrlap{,} 
 \end{tikzcd}
 \]
where $\mathsf{fill}(\alpha)^\sharp$ is the adjoint transpose of $\mathsf{fill}(\alpha)$, and both squares are pullbacks.  

It thus suffices to show that the following is also a pullback.
 \[
 \begin{tikzcd}
\Fib^*(\alpha) \arrow[d] \arrow[r] & (\UVptd)_\II \arrow[d] \\ 
X \ar[r,swap,"{\mathsf{fill}(\alpha)^\sharp}"]  & \UV_\II \mathrlap{.} 
 \end{tikzcd}
 \]
 To see this, factor the transpose $\mathsf{fill}(\alpha)^\sharp$ as $\mathsf{fill}(\alpha)_\II \circ \eta$, as indicated below, noting that the root preserves pullbacks.
\[
 \begin{tikzcd}
 \Fib^*(\alpha) \arrow[d] \arrow[r] & \Fill(\alpha\circ x)_\II\arrow[d] \arrow[r] & {\UVptd}_\II\arrow[d] \\ 
X \ar[r,"\eta"] \ar[rr,bend right=20,swap,"\mathsf{fill}(\alpha)^\sharp"] & (X^\II)_\II \arrow[r,"\mathsf{fill}(\alpha)_\II"] & \UV_\II
 \end{tikzcd}
 \]
The square on the left is a pullback by the definition \eqref{diagFibstar} of $\Fib^*(\alpha)$. 

For (ii), it now suffices to show that for every $t \co Y \to X$,
\begin{equation}\label{eq:fillAnatural}
\mathsf{fill}(\alpha)^\sharp \circ t = \mathsf{fill}({\alpha(t))}^\sharp \mathrlap{.}
\end{equation}
Since $\mathsf{fill}(\alpha)^\sharp \circ t = (\mathsf{fill}(\alpha) \circ t^\II)^\sharp$ by the naturality of transposition, it suffices to show
\begin{equation}\label{eq:needforFillstable}
\mathsf{fill}(\alpha) \circ t^\II = \mathsf{fill}(\alpha(t)) \mathrlap{.}
\end{equation}
Now, by definition, $\mathsf{fill}(\alpha) \co X^\II \to \UV$ classifies $\Fill(\alpha\circ x) = \Fill_0( \alpha\circ x) \times \Fill_1( \alpha\circ x)$ (over $X^\II$), so we just need to show that $\Fill_0(\alpha\circ x) \to X^\II$ is stable under pullback along $t^\II \co Y^\II \to X^\II$ (the case $\Fill_1( \alpha\circ x)$ is analogous).  As shown in the proof of  \cref{prop:FibvsFill},  over $X^\II$ we have
\[\textstyle
\Fill_0( \alpha\circ x )\ \cong\ \prod_{a \oftype \alpha(x_0)} \TFib \, F_0(x ,a)\mathrlap{.}
\]
Moreover, by the definition \eqref{def:fibertypeF0} the display map  for $F_0(x ,a)$, namely
\[\textstyle
\big( \dsm{(x,a) \oftype (\dsm{x\co X^\II} \alpha(x_0))} F_0(x,a) \big)  \too\ \dsm{x\co X^\II} \alpha(x_0) \,,
\]
is isomorphic to the pullback-hom ${\delta^0 \Rightarrow{p_\alpha}}$ over the isomorphism
\[\textstyle
\dsm{x\co X^\II} \alpha(x_0)\ \cong\ X^\II\times_X X.p_\alpha\, \mathrlap{.}
\]
Applying the functor ${\delta^0 \Rightarrow{(-)}}$ to the pullback square on the left below results in the upper one on the right: 
\[
 \begin{tikzcd}
Y.\alpha(t) \arrow[d,swap, "{p_\alpha(t)}"] \arrow[r] \pbmark & X.\alpha \arrow[d,"{p_\alpha}"]  \\ 
Y \ar[r,swap,"t"] & X \mathrlap{,} 
 \end{tikzcd}
 \qquad \qquad
 \begin{tikzcd}
Y.\alpha(t)^\II \arrow[d,swap, "{\delta^0 \Rightarrow{p_{\alpha(t)}}}"] \arrow[r] \pbmark & X.\alpha^\II \arrow[d,"{\delta^0 \Rightarrow{p_\alpha}}"]  \\ 
Y^\II\times_Y Y.p_{\alpha(t)} \ar[r,swap,""] \ar[d] \pbmark & X^\II\times_X X.p_\alpha \ar[d]  \\
Y^\II \ar[r,swap,"t^\II"] & X^\II \mathrlap{.} 
 \end{tikzcd}
 \]
So we are done by the stability of $\TFib$ (\cref{cor:TFibstable}) and the Beck-Chevalley condition for $\Pi$, using the lower pullback on the right.
\end{proof}

\begin{thm}[Universal fibration] 
\label{prop:classifier-fib}
There is a \strong{universal small uniform fibration}, $\Fibptd \to \Fib$, 
in the sense that every small fibration $A \to X$ is a pullback of $\Fibptd \to \Fib$ along a canonically determined classifying map $X \to \Fib$.
\begin{equation}
\begin{tikzcd}
      A \ar[r, dotted] \ar[d] \pbmark & \Fibptd
      \ar[d] \\
      X \ar[r, swap,dotted] & \Fib
    \end{tikzcd}   
\end{equation}
\end{thm}

\begin{proof}
Define $\Fib = \Fib^*(\id)$ and  $\Fibptd \to \Fib$ by the pullback indicated below.
\[
\begin{tikzcd}
\Fibptd \arrow[d] \arrow[r]  \pbmark & \UVptd \arrow[d, "\dspu"] \\
\Fib \arrow[r] & \UV
\end{tikzcd}
\]
The rest of the proof is the same as for \cref{prop:classifier-trivfib}.
\end{proof}

\begin{rmk}\label{rmk:universalfib}
As in \cref{rmk:TFib}, in terms of variables we have $\Fib = \dsm{\alpha\co\UV}\Fib^*(\alpha)$.  Thus a uniform fibration is a small family $\alpha\co\UV$ equipped with a fibration structure $f\co\Fib^*(\alpha)$, and the universal one is indexed by the type of all such pairs $(\alpha,f)$.
\end{rmk}


\section{Path types} 
\label{sec:path-types} 

The fundamental insight of Homotopy Type Theory \cite{awodey-warren:homotopy-idtype,gambino-garner:idtypewfs} is that identity terms $p :\Id_A(a, b)$ in (intensional) Martin-L\"of type theory \cite{nordstrom-petersson-smith:ml} behave like continuous paths  $p \co [0,1] \to A$ in a \emph{space} $A$, and the identity type $\Id_A$ itself then acts like a \emph{pathspace} $A^{[0,1]}$.  Moreover, the rules of the type theory then permit a formal derivation of the \emph{homotopy lifting property} for type families $X\vdash A \type$, thus forcing them to be regarded as \emph{fibrations} $A\twoheadrightarrow X$. 

To make this insight precise, for any object $A\in \tE$, define the \emph{path object} $A^\II$  to be the exponential by the interval $\II$ (\cref{def:IntervalWithConnections}). The path object is indexed over $A\times A$ by the pair of \emph{endpoint evaluations} $\epsilon= ( \epsilon_0, \epsilon_1 ) \co  A^\II \too A\times A$, where $\epsilon_k = A^{\delta_k}$ for $k = 0,1$.  Together with the \emph{constant path} map $\con = A^{!} \co  A \to A^\II$, where $! \co  A \to 1$, these maps make the following diagram commute.
\begin{equation}\label{diag:pathtype}
\begin{tikzcd}
A \ar[d, swap, "="] \ar[r, "\con" ]  &  A^\II  \ar[d, "{\epsilon}"]  \\
A  \ar[r, swap,"\Delta_A" ] &  A\times A \mathrlap{.}
\end{tikzcd}
\end{equation}

More generally, for $p \co A \to X$  any (small) map, with classifying map $\alpha \co  X \to \UV$, the \strong{path object} of $p \cong p_{\alpha}$ in $\tE/_X$ is the exponential of $p_{\alpha}$ by $\II_X = (\pi_1\co X\times \II \to X)$, which can be constructed in $\tE$ as the following pullback along $\con_X \co  X\to X^\II$.
\[
\begin{tikzcd}
  \typecomp{X}{\alpha^\II} \ar[r] \ar[d,swap, "{p_{\alpha^\II}}"] \pbmark  & (\typecomp{X}{\alpha})^\II \ar[d, "(p_{\alpha})^\II"] \\ 
X  \ar[r, swap, "\con_X"] & X^\II  \mathrlap{.}
\end{tikzcd}
\]
Let $\alpha^\II \co   X\to \UV$ classify $p_{\alpha^\II} \co  \typecomp{X}{\alpha^\II} \to X$, so that in the type theory of $\tE$ we have the \emph{path type} $X\vdash \alpha^\II \co  \UV$.
Terms $X\vdash u\co \alpha^\II$ are called \strong{paths} in $\alpha$.  
There are of course maps over $X$ analogous to those in \eqref{diag:pathtype},  of the form: 
\[
\begin{tikzcd}
\typecomp{X}{\alpha} \ar[d, swap,"="] \ar[rr, "\con_{\alpha}" ] 
		&&  \typecomp{X}{\alpha^\II} \ar[d,  "{\epsilon}"] \\
\typecomp{X}{\alpha} \ar[r, swap,"\Delta_\alpha" ] & 
	 \typecomp{X}{\alpha}\times_X \typecomp{X}{\alpha}  \ar[r, swap,"\cong"]  
	 	& \typecomp{X}{(\alpha\times\alpha)} \mathrlap{.}   
\end{tikzcd}
\]
Given  any $x \co  \yon{c} \to X$ and elements $c  \forces a,b \co  \alpha(x)$ (corresponding to sections of $\yon{c}.\alpha(x)\to\yon{c}$),  a path $c  \forces u\co \alpha(x)^\II(a,b)$ thus represents a \emph{homotopy} $a\sim b$ (over $\yon{c}$).

Regarding paths $a\path b$ as ``identifications'' as in Homotopy Type Theory, the path type $\alpha^\II$ acts as an \emph{identity type} for $\alpha$.   Under the propositions-as-types view, we should then expect a type family $X.\alpha \vdash \beta \co  \UV$ to respect identifications $a\sim b$ in $\alpha$, by the principle of ``indistinguishability of identicals'':  if $a\sim b$ and $\beta(a)$, then $\beta(b)$,  regarding $\beta \co   X.\alpha \to \UV$ as a (type-valued) ``propositional function" on $\alpha$. Indeed, we have the following:

\begin{lemma}\label{prop:transport}
For any $X\in\tE$ and small map $A \to X$, let $B \to A$ be a small uniform fibration.  For the associated families of types $\alpha\co  X \to \UV$ and $\beta \co  X.\alpha \to\UV$, and for any $x \co  \yon{c} \to X$, suppose we have terms  $c\forces a, b\co  \alpha(x)$ and a path between them $c\forces u \co  \alpha(x)^\II(a,b)$.  Then for each $e$ with $c\forces e \co  \beta(x,a)$ there is an associated $u*e$ with $c\forces u*e\co  \beta(x,b)$. 
\end{lemma}
\begin{proof}
Consider the following diagram, in which $\lambda{u'} = u$. 
\[
\begin{tikzcd}[column sep=large]
\yon{c}\times 1 \arrow[d, tail,  swap, "{\yon{c}\times\delta_0}"] \arrow[r, "{(x,a,e)}"]  & X.\alpha.\beta \ar[d,"{p_\beta}"] \\
\yon{c}\times \II  \arrow[r, swap, "u'"] \arrow[ru, dotted,swap, "{j}"] \arrow[d]  & X.\alpha  \ar[d] \\
\yon{c} \arrow[r,swap, "{x}"]  &  X  \mathrlap{.}
\end{tikzcd}
\]
The solid upper square commutes because $u'(\id\times\delta_0) = (x,a) = p_\beta(x,a,e)\co  \yon{c} \to X.\alpha$, and by assumption  $c\forces e \co  \beta(x,a)$.  As the reader can check, for the map on the left we have $\yon{c}\times\delta_0 = 0_{\yon{c}} \otimes \delta_0$ with the cofibration $0_{\yon{c}}\co  0\mono \yon{c}$, and thus it is a (generating) trivial cofibration in the sense of \cref{rem:gentrivcof}.
By assumption, we have a uniform fibration structure $f \co  \Fib(\beta)$ over $X.\alpha$, so by \cref{prop:uniform-fibration-from-fib-type} we obtain a diagonal filler $j = f(0_{\yon{c}},\, (x,a,e),\, u')$ as indicated.
We can then take  $u*e \defeq \pi_2\, j(\yon{c}\times\delta_1)$, so that $c\forces u*e \co  \beta(\pi_1\, j(\yon{c}\times\delta_1))$.  But 
\[
\beta(\pi_1\, j(\yon{c}\times\delta_1)) = \beta(p_\beta\, j(\yon{c}\times\delta_1)) = \beta(u'(\yon{c}\times\delta_1))= \beta(x,b)\,. \qedhere
\]
\end{proof}

The term $u*e$ is called the \emph{transport} of $e$ along $u$.  It is in virtue of this transport operation that fibrations are used as the type families in Homotopy Type Theory, as we shall see in \cref{prop:IdTypeRules} .  Of course, we then also need the path type $\alpha^\II$ itself to be a fibration over  $\alpha\times\alpha$.  Fortunately, we have the following result.

\begin{lemma}\label{prop:IdisFib}
Let $X\in\tE $ and $\alpha\co  X \to \UV$. Given a uniform fibration structure $f\co  \Fib(\alpha)$ over $X$,
we can construct a uniform fibration structure $f'\co \Fib(\alpha^\II)$ over $X.\alpha.\alpha \cong X.\alpha \times_X X.\alpha$.
\end{lemma}

\begin{proof}
Consider a diagonal filling problem for the display map $p_{\alpha^\del} \co  X.\alpha^\II \to X.\alpha.\alpha$ with respect to a generating trivial cofibration $m \otimes\delta_k$ with cofibration $m \co S \mono \yon{c}$ (see \cref{rem:gentrivcof}):
\begin{equation}\label{diag:diagfillerforalphaI}
\begin{tikzcd}[column sep=large] 
\yon{c} +_S (S \times\II) \arrow[r, "u"] \arrow[d, tail,  swap, "{m \otimes\delta_k\, }"]   & X.\alpha^\II \ar[d , "{\, p_{\alpha^\del} }"] \\
\yon{c} \times \II  \arrow[r, swap, "v"] \arrow[ru, dotted] & X.\alpha.\alpha \mathrlap{.}
\end{tikzcd}
\end{equation}
In terms of the adjoint functors $\otimes$ and $\Raw$, we have $p_{\alpha^\del} = (\del\Raw p_{\alpha})$, where $\del = [\delta_0, \delta_1] \co  1+1\mono \II$, as the reader can check.   So by the adjoint relation \eqref{eq:leibnizadjunction}, the above diagonal filling problem can be reformulated equivalently as follows, where $D = \dom((m \otimes\delta_k)\otimes \del)$.
\[
\begin{tikzcd}[column sep=large] 
D \arrow[r, "u' "] \arrow[d, tail,  swap, "{(m \otimes\delta_k)\otimes \del\, }"]   & X.\alpha \ar[d, "{\, p_{\alpha} }" ] \\
(\yon{c} \times \II) \times \II  \arrow[r, swap, "v' "] \arrow[ru, dotted,swap, "j"] & X \mathrlap{.}
\end{tikzcd}
\]
Since $(m \otimes\delta_k)\otimes \del \cong (\del\otimes m) \otimes\delta_k$, and $\del$ is a cofibration, and cofbrations are closed under $\otimes$, the map on the left is a trivial cofibration.  Thus from the assumed fibration structure $f\co  \Fib(\alpha)$ we obtain a diagonal filler $j = f(\del\otimes m, u',v')$ as indicated.  Transposing $j \co  (\yon{c} \times \II) \times \II \to X.\alpha$ to $j' \co  \yon{c} \times \II \to  X.\alpha^\II$ provides a diagonal filler for \eqref{diag:diagfillerforalphaI}, which we may take as the corresponding value of $f'\co \Fib(\alpha^\II)$,
\[
f'(m,u,v) \defeq  f(\del\otimes m, u', v')'
\]
(up to the iso $(m \otimes\delta_k)\otimes \del \cong (\del\otimes m) \otimes\delta_k$). Uniformity of $f'$ then follows from  that of $f$ and the naturality of adjoint transposition.  
\end{proof}

Recalling \cref{def:IntervalWithConnections}, a consequence of the connection $\conn_0 \co \II \times \II \to \II$ is the contractibility of (the fibers of) the $0^{th}$-endpoint map $\epsilon_0 \co  \typecomp{X}{\alpha^\II}\too \typecomp{X}{\alpha}$  (and similarly for $\conn_1$ and $\epsilon_1$).

\begin{lemma}\label{prop:pathcontr}
Let $\alpha \co  X\to\UV$ with elements $c \forces a,b \co  \alpha(x)$.  For any path $u \co  a\sim b$ in $\alpha(x)$,  there is a path $\varepsilon_u \co  \con_a \sim u$ in the path type $\alpha(x)^\II$.  Formally, from $c \forces u\co  \alpha(x)^\II(a,b)$,
we can construct $c \forces \varepsilon_u  \co  (\alpha(x)^\II)^\II(\con_a,u)$.
\end{lemma}
\begin{proof}
Over $X$, we of course have an iso $ X.(\alpha(x)^\II)^\II \cong X.\alpha(x)^{\II\times\II}$, and so we can use the connection $\conn_0 \co \II \times \II \to \II$ to obtain a map over $X$ of the form
\[
\varepsilon \co  X.\alpha(x)^\II \too X.\alpha(x)^{\II\times\II} \cong X.(\alpha(x)^\II)^\II\,.
\] 
It follows from the connection diagrams in \cref{def:IntervalWithConnections}(iii) that $\varepsilon_u$ satifies the required forcing condition.
\end{proof}

\cref{prop:pathcontr} can be understood informally as follows: the fiber  at $a \co  \alpha$ of the endpoint map $\epsilon_0 \co  \typecomp{X}{\alpha^\II}\too \typecomp{X}{\alpha}$ may be regarded as a ``homotopy singleton'',
\[
\textstyle \mathsf{ho}{\pat{a}} = \dsm{b \oftype \alpha} a\sim b\,.
\]  
The lemma says that, up to homotopy, the only element in $\mathsf{ho}{\pat{a}}$ is the constant path $\con_{a} \co  a\sim~a$.
\smallskip

Using the previous three lemmas we can show that Martin-L\"of's rules \cite{nordstrom-petersson-smith:ml} for identity types $\Id_\alpha$  are satisfied by the path types  $\alpha^\II$, provided that we interpret types as (uniform) fibrations (cf.~\cite{awodey-warren:homotopy-idtype}).  In the following, we take $A\to 1$ rather than the general case $A\to X$ merely for notational convenience, and use displayed variables for a more familiar presentation.  The judgement $\Gamma \vdash A \type$ will be interpreted to mean  $ \Gamma \vdash A\co  \Fib$, where $\Fib = \dsm{\alpha\co \UV}\Fib^*(\alpha)$ is the base of the universal (small, uniform) fibration (\cref{rmk:universalfib}), and so we can write $ A\co  \Fib$ as a pair $A = (\alpha, f)$ with $\alpha \co \UV$  and $f \co  \Fib^*(\alpha)$.

\begin{thm}\label{prop:IdTypeRules} Suppose the presheaf topos $\tE$ has a tiny interval $\II$ with connections.   Then in any context $\Gamma$, for any $A = (\alpha, f) \co  \Fib$ with $\alpha \co  \UV$ and $f\co  \Fib^*(\alpha)$, 
let  $A^\II = (\alpha^\II, f') \co  \Fib$ for $\alpha^\II \co  \UV$ and $f'\co  \Fib^*(\alpha^\II)$.   The following standard \emph{formation}, \emph{introduction}, and \emph{elimination} rules for identity types then hold in the internal type theory of $\tE$.
\medskip
\[
\begin{prooftree}
A  \type
\justifies
x \oftype A, \, y \oftype A\, \der\,   A^\II(x,y) \type
\end{prooftree}
\qquad\qquad
\begin{prooftree}
\strut
\justifies
x\oftype A \, \der\,  \con(x) \co  A^\II(x,x)
\end{prooftree}
\]
\medskip
\[
\begin{prooftree}
x \co  A, \, 
y \co  A, \,
z \co  A^\II(x,y)\, 
\der\,  B(x, y, z) \type  \qquad\qquad
x \co  A \, 
\der\,  b(x) \co  B(x, x, \con(x))
\justifies
x \co  A, \, 
y \co  A, \, 
z \co  A^\II(x,y) \, 
\der\,  \mathsf{j}_b(x,y,z) \co  B(x,y,z) \mathrlap{.} 
\end{prooftree}
\]
\medskip

\noindent If, moreover, the constant path map $\con\co  A\to A^\II$ is a cofibration, then the following \emph{computation} rule also holds.
\begin{equation}\label{eq:Idcomp}
x \oftype A \, \der\, \mathsf{j}_b(x,x,\con(x)) = b(x) \co  B(x, x, \con(x) )
\end{equation}
\end{thm}
\begin{proof} We follow the approach in \cite[Section 2.4]{Awo18-naturalModels}.
Reasoning first in an arbitrary context, which we leave implicit, let $A = (\alpha, f) \co  \Fib$  with $\alpha \co  \UV$ and $f\co  \Fib^*(\alpha)$.  Then by \cref{prop:IdisFib} we have $f' \co  \Fib^*(\alpha^\II)$, whence  $A^\II = (\alpha^\II, f') \co  \Fib$.  Thus we have the formation rule. The introduction rule follows from \eqref{diag:pathtype}.  The underlying types in those rules are displayed in the following two diagrams.
\begin{equation}\label{diag:Idformintro}
\begin{tikzcd}
 1.\alpha.\alpha \ar[r] \pbmark \ar[d] & 1.\alpha \ar[r] \ar[d,swap,"p_\alpha"] \pbmark & \UVptd \ar[d] \\
1.\alpha \ar[r,swap, "p_\alpha"]  & 1  \arrow[r, swap, "\alpha"] & \UV 
\end{tikzcd}  \qquad \qquad
\begin{tikzcd}
1.\alpha  \arrow[r, "\con"] \ar[d, swap, "\Delta"] & \UVptd \arrow[d] \\
1.\alpha.\alpha \ar[r,  swap, "{\alpha^\II}"] & \UV 
\end{tikzcd} 
\end{equation}
The premises of the elimination rule are then interpreted as the solid part of the following, letting $B = (\beta, g)\co \Fib$.
\begin{equation}\label{diag:Idelim}
\begin{tikzcd}
1.\alpha.\alpha.\alpha^{\II}  \ar[r] \ar[d] \pbmark  & \UVptd  \ar[d] \\
1.\alpha.\alpha  \ar[r, swap, "{\alpha^\II}"]  & \UV
\end{tikzcd} \qquad  \qquad
\begin{tikzcd}
1.\alpha \ar[d,swap, "{(\Delta, \con)}"] \ar[rr, "b"]  & & \UVptd  \ar[d] \\
1.\alpha.\alpha.\alpha^{\II} \ar[rru, swap, dotted, "j"] \ar[rr, swap, "\beta"] && \UV 
\end{tikzcd}
\end{equation}
The conclusion of elimination asserts the existence of the indicated diagonal arrow $j$ making the \emph{lower} triangle commute, while the computation rule states that the \emph{upper} one does.  
To build such a map $j$ it suffices to show that for any $c$ and $x \co  \yon{c} \to 1.\alpha.\alpha.\alpha^{\II}$, there is an element $j_{x}  \co  \yon{c} \to \UVptd$ such that
\[
c \forces j_{x} \co \beta(x)\,,
\]
uniformly in $c$. Since $x$ can be rewritten as $(a,a',u)\co  \yon{c} \to 1.\alpha.\alpha.\alpha^{\II}$, we have $c\forces u\co  \alpha^{\II}(a,a')$.  By \cref{prop:pathcontr} we then obtain $c \forces \varepsilon_u  \co  (\alpha(x)^\II)^\II(\con{(a)},u)$.  Substituting $a$ in the second premise of the elimination rule yields $c\forces b(a) \co \beta(a,a,\con(a))$. Transporting the latter along $\varepsilon_u$ by  \cref{prop:transport} then provides $c\forces \varepsilon_u * b(a)  \co  \beta(a,a',u)$.  Now  set
\[
j_{(a,a',u)} \defeq\ \varepsilon_u * b(a)\,. 
\]
Uniformity in $c$ follows from the same for each of the terms in the construction of $j_{(a,a',u)}$.  Indeed,  given $f \co  d\to c$, we have 
\begin{equation}\label{eq:uniformJ}
d \forces j_{(a,a',u)}(f) = (\varepsilon_u * b(a))(f) = \varepsilon_{uf} * b(af) = j_{(af,a'f,uf)} \co  \beta(af,a'f,uf)\,.
\end{equation}

Finally, if $\con\co  A\to A^\II$ is a cofibration then, looking at \eqref{diag:Idelim}, the map $j$ can be (uniformly) replaced by a homotopic one $j \sim j'$ that \emph{also} makes the upper triangle commute, by a standard argument using the fibration on the right (as in e.g.~\cite[Lemma 56]{awodey-cubical-git}).  The new map $j'$ then validates \emph{both} the elimination and computation rules.  
 
Note that the required coherence conditions for $A^\II$ and $j$ with respect to all context maps $Y \to X$ are ensured by the same device used for $\Pi$ in \eqref{diagram:Picomp}; namely, we apply the above construction once in a universal case constructed from small map classifier $\UVptd\to\UV$, and then obtain all of the instances in a uniform way by composition. See \cite[Section 2.4]{Awo18-naturalModels} for details.
\end{proof}

\begin{rmk} 
The condition that the constant path map $\con \co  A\to A^\II$ is a cofibration holds, for instance, when \emph{all} monos are taken as cofibrations, as in a Cisinski model structure.  When this fails, as it may for presheaves not valued in sets, 
one can still recover the computation rule \eqref{eq:Idcomp} by factoring $\con\co  A\to A^\II$ into a cofibration followed by a (uniform) trivial fibration (\cref{rmk:tfib-as-algebras}), as shown in~\cite{swan2018identity}.   
\end{rmk}


\section{Conclusions}
\label{sec:con}

We conclude the paper by outlining some directions for future work. First, while here we considered $\cSet$-valued presheaves, where $\cSet$ is the
category of sets associated to ZFC extended with an inaccessible cardinal, it is also interesting to consider presheaf categories
valued in other kinds of categories or categories of internal presheaves in some topos. Such a project may also connect with the idea of developing
the material here in a constructive metatheory, as suggested to us by Thierry Coquand and Michael Rathjen. Secondly, one could try to extend 
our results to categories of sheaves. One preliminary issue for this is to find an adequate counterpart of the Hofmann-Streicher universe and
of the classifier of small maps. One may also try to develop the material here without using a universe, adopting the approach of~\cite{shulman:stack} to
replace quantification over the universe. Finally, one may wish to consider extending the Kripke-Joyal forcing introduced here to other forms of type theory, such as the modal type theory studied in~\cite{LOPS18}.

\end{document}

%% file: prooftree.tex
\newdimen\proofrulebreadth \proofrulebreadth=.05em
\newdimen\proofdotseparation \proofdotseparation=1.25ex
\newdimen\proofrulebaseline \proofrulebaseline=2ex
\newcount\proofdotnumber \proofdotnumber=3
\let\then\relax
\def\hfi{\hskip0pt plus.0001fil}
\mathchardef\squigto="3A3B
\newif\ifinsideprooftree\insideprooftreefalse
\newif\ifonleftofproofrule\onleftofproofrulefalse
\newif\ifproofdots\proofdotsfalse
\newif\ifdoubleproof\doubleprooffalse
\let\wereinproofbit\relax
\newdimen\shortenproofleft
\newdimen\shortenproofright
\newdimen\proofbelowshift
\newbox\proofabove
\newbox\proofbelow
\newbox\proofrulename
\def\shiftproofbelow{\let\next\relax\afterassignment\setshiftproofbelow\dimen0 }
\def\shiftproofbelowneg{\def\next{\multiply\dimen0 by-1 }%
\afterassignment\setshiftproofbelow\dimen0 }
\def\setshiftproofbelow{\next\proofbelowshift=\dimen0 }
\def\setproofrulebreadth{\proofrulebreadth}

\def\prooftree{
\ifnum  \lastpenalty=1
\then   \unpenalty
\else   \onleftofproofrulefalse
\fi
\ifonleftofproofrule
\else   \ifinsideprooftree
        \then   \hskip.5em plus1fil
        \fi
\fi
\bgroup
\setbox\proofbelow=\hbox{}\setbox\proofrulename=\hbox{}%
\let\justifies\proofover\let\leadsto\proofoverdots\let\Justifies\proofoverdbl
\let\using\proofusing\let\[\prooftree
\ifinsideprooftree\let\]\endprooftree\fi
\proofdotsfalse\doubleprooffalse
\let\thickness\setproofrulebreadth
\let\shiftright\shiftproofbelow \let\shift\shiftproofbelow
\let\shiftleft\shiftproofbelowneg
\let\ifwasinsideprooftree\ifinsideprooftree
\insideprooftreetrue
\setbox\proofabove=\hbox\bgroup$\displaystyle 
\let\wereinproofbit\prooftree
\shortenproofleft=0pt \shortenproofright=0pt \proofbelowshift=0pt
\onleftofproofruletrue\penalty1
}

\def\eproofbit{
\ifx    \wereinproofbit\prooftree
\then   \ifcase \lastpenalty
        \then   \shortenproofright=0pt  
        \or     \unpenalty\hfil         
        \or     \unpenalty\unskip       
        \else   \shortenproofright=0pt  
        \fi
\fi
\global\dimen0=\shortenproofleft
\global\dimen1=\shortenproofright
\global\dimen2=\proofrulebreadth
\global\dimen3=\proofbelowshift
\global\dimen4=\proofdotseparation
\global\count255=\proofdotnumber
$\egroup  
\shortenproofleft=\dimen0
\shortenproofright=\dimen1
\proofrulebreadth=\dimen2
\proofbelowshift=\dimen3
\proofdotseparation=\dimen4
\proofdotnumber=\count255
}

\def\proofover{
\eproofbit 
\setbox\proofbelow=\hbox\bgroup 
\let\wereinproofbit\proofover
$\displaystyle
}%
\def\proofoverdbl{
\eproofbit 
\doubleprooftrue
\setbox\proofbelow=\hbox\bgroup 
\let\wereinproofbit\proofoverdbl
$\displaystyle
}%
\def\proofoverdots{
\eproofbit 
\proofdotstrue
\setbox\proofbelow=\hbox\bgroup 
\let\wereinproofbit\proofoverdots
$\displaystyle
}%
\def\proofusing{
\eproofbit 
\setbox\proofrulename=\hbox\bgroup 
\let\wereinproofbit\proofusing
\kern0.3em$
}

\def\endprooftree{
\eproofbit 
  \dimen5 =0pt
\dimen0=\wd\proofabove \advance\dimen0-\shortenproofleft
\advance\dimen0-\shortenproofright
\dimen1=.5\dimen0 \advance\dimen1-.5\wd\proofbelow
\dimen4=\dimen1
\advance\dimen1\proofbelowshift \advance\dimen4-\proofbelowshift
\ifdim  \dimen1<0pt
\then   \advance\shortenproofleft\dimen1
        \advance\dimen0-\dimen1
        \dimen1=0pt
        \ifdim  \shortenproofleft<0pt
        \then   \setbox\proofabove=\hbox{%
                        \kern-\shortenproofleft\unhbox\proofabove}%
                \shortenproofleft=0pt
        \fi
\fi
\ifdim  \dimen4<0pt
\then   \advance\shortenproofright\dimen4
        \advance\dimen0-\dimen4
        \dimen4=0pt
\fi
\ifdim  \shortenproofright<\wd\proofrulename
\then   \shortenproofright=\wd\proofrulename
\fi
\dimen2=\shortenproofleft \advance\dimen2 by\dimen1
\dimen3=\shortenproofright\advance\dimen3 by\dimen4
\ifproofdots
\then
        \dimen6=\shortenproofleft \advance\dimen6 .5\dimen0
        \setbox1=\vbox to\proofdotseparation{\vss\hbox{$\cdot$}\vss}%
        \setbox0=\hbox{%
                \advance\dimen6-.5\wd1
                \kern\dimen6
                $\vcenter to\proofdotnumber\proofdotseparation
                        {\leaders\box1\vfill}$%
                \unhbox\proofrulename}%
\else   \dimen6=\fontdimen22\the\textfont2 
        \dimen7=\dimen6
        \advance\dimen6by.5\proofrulebreadth
        \advance\dimen7by-.5\proofrulebreadth
        \setbox0=\hbox{%
                \kern\shortenproofleft
                \ifdoubleproof
                \then   \hbox to\dimen0{%
                        $\mathsurround0pt\mathord=\mkern-6mu%
                        \cleaders\hbox{$\mkern-2mu=\mkern-2mu$}\hfill
                        \mkern-6mu\mathord=$}%
                \else   \vrule height\dimen6 depth-\dimen7 width\dimen0
                \fi
                \unhbox\proofrulename}%
        \ht0=\dimen6 \dp0=-\dimen7
\fi
\let\doll\relax
\ifwasinsideprooftree
\then   \let\VBOX\vbox
\else   \ifmmode\else$\let\doll=$\fi
        \let\VBOX\vcenter
\fi
\VBOX   {\baselineskip\proofrulebaseline \lineskip.2ex
        \expandafter\lineskiplimit\ifproofdots0ex\else-0.6ex\fi
        \hbox   spread\dimen5   {\hfi\unhbox\proofabove\hfi}%
        \hbox{\box0}%
        \hbox   {\kern\dimen2 \box\proofbelow}}\doll%
\global\dimen2=\dimen2
\global\dimen3=\dimen3
\egroup 
\ifonleftofproofrule
\then   \shortenproofleft=\dimen2
\fi
\shortenproofright=\dimen3
\onleftofproofrulefalse
\ifinsideprooftree
\then   \hskip.5em plus 1fil \penalty2
\fi
}
